\documentclass[a4 paper, 12pt]{amsart}
\usepackage[utf8]{inputenc}
\usepackage{amssymb}
\usepackage{amsmath}
\usepackage{mathtools}
\usepackage{amsthm}
\usepackage{tikz-cd}
\usepackage{parskip}
\usepackage{enumitem}
\usepackage{breqn}
\usepackage[cmtip,all]{xy}
\usepackage[hyphens]{url}
\usepackage{hyperref}
\hypersetup{hidelinks, colorlinks=true, citecolor=blue, linkcolor=black, urlcolor=blue}
\usepackage[hyphenbreaks]{breakurl}
\usepackage{amsfonts}
\usepackage{mathrsfs}
\usepackage{bbm}
\usepackage{tikzit}

\tikzstyle{red dot}=[fill=red, draw=black, shape=circle]
\tikzstyle{green dot}=[fill=green, draw=black, shape=circle]
\tikzstyle{white circle}=[fill=white, draw=black, shape=circle, scale=.7, minimum size=1cm]
\tikzstyle{black dot}=[fill=black, draw=black, shape=circle, minimum size=1pt, inner sep=0pt]
\tikzstyle{dashed line}=[ultra thick, draw=black]

\tikzstyle{arrowhead}=[->]

\usepackage{cite}

\newtheorem{theorem}{Theorem}[section]
\newtheorem{theoremA}{Theorem}

\newtheorem{theoremB}{Theorem}

\newtheorem{lemma}[theorem]{Lemma}
\newtheorem{proposition}[theorem]{Proposition}
\newtheorem{corollary}[theorem]{Corollary}

\theoremstyle{definition}
\newtheorem{definition}[theorem]{Definition}
\newtheorem{example}[theorem]{Example}
\newtheorem{remark}[theorem]{Remark}
\newtheorem{question}[theorem]{Question}
\newtheorem{notation}[theorem]{Notation}

\newcommand{\nsg}{\vartriangleleft}
\newcommand{\Leaf}{\ensuremath{\mathrm{Leaf}}}
\newcommand{\RR}{\mathbb{R}}
\newcommand{\QQ}{\mathbb{Q}}
\newcommand{\ZZ}{\mathbb{Z}}
\newcommand{\NN}{\mathbb{N}}
\newcommand{\catname}[1]{\ensuremath{\mathsf{{#1}}}}
\newcommand{\op}{^\text{op}}

\newcommand{\im}[1]{\ensuremath{\text{im}\left(#1\right)}}

\newcommand{\wt}{\widetilde}
\newcommand{\ad}{\mathrm{ad}}
\newcommand{\supp}{\mathrm{supp}}

\newcommand{\stabv}{\mathrm{Stab}_V}
\newcommand{\fixv}{\mathrm{Fix}_V}
\newcommand{\epi}{\chi}
\newcommand{\surj}{\twoheadrightarrow}
\newcommand{\aut}{\mathrm{Aut}}
\newcommand{\stabn}{\mathrm{Stab}_N(\mathbb{Q}_2)}
\newcommand{\nvc}{\mathrm{N}_{H(\mathfrak{C})}(V)}

\newcommand{\bg}{\mathcal{B}(G)}
\newcommand{\longsquiggly}{\xymatrix{{}\ar@{~>}[r]&{}}}
\newcommand{\fg}[1]{\widehat{#1}}
\newcommand{\cat}[1]{\catname{#1}}
\newcommand{\frc}[1]{\mathrm{Frac}\left({#1}\right)}
\newcommand{\joneso}[1]{K\left({#1} \right)}
\newcommand{\jonesm}[1]{R\left({#1} \right)}
\newcommand{\prj}[1]{p\left({#1}\right)}
\newcommand{\caret}{\raisebox{-2pt}{\begin{tikzpicture}[scale=.3] \draw (0,0) -- (-.5,1); \draw(0,0) -- (.5,1);\end{tikzpicture}}}
\newcommand{\fcaret}{\raisebox{-2pt}{\begin{tikzpicture}[scale=.3, yscale=-1] \draw (0,0) -- (-.5,1); \draw(0,0) -- (.5,1);\end{tikzpicture}}}
\newcommand{\cantor}{\mathfrak{C}}

\newcommand{\nc}{\mathrm{NC}}
\newcommand{\limg}{\varprojlim \Gamma}
\newcommand{\lima}{\varprojlim \alpha}
\title{Rigidity and Automorphisms of Groups constructed using Jones' Technology}
\date{}
\author{Christian De Nicola Larsen}
\address{Christian De Nicola Larsen, School of Mathematics and Statistics, University of New South Wales, Sydney NSW 2052, Australia}
\email{c.denicolalarsen@unsw.edu.au}
\usepackage{fullpage}
\usepackage{subfiles}
\begin{document}
\begin{abstract}
Jones' technology, developed by Vaughan Jones during his exploration of the connections between conformal field theory and subfactors, is a powerful mechanism for generating actions of groups coming from categories, notably Richard Thompson's groups $F \subseteq T \subseteq V$.

We give a structure theorem for the isomorphisms between split extensions of Thompson's group $V$ arising from Jones' technology, generalising results of Brothier. Using this structure theorem, we classify a family of unrestricted, twisted permutational wreath products up to isomorphism, and decompose their automorphism groups in the untwisted case.

These unrestricted wreath products arise from applying Jones’ technology to contravariant monoidal functors. In contrast, using covariant functors, Brothier constructed a large class of \textit{restricted} wreath products, classified them up to isomorphism, and completely described their automorphism groups. Our work broadens Brothier's findings and highlights the duality between groups constructed using covariant and contravariant functors.
\end{abstract}
\maketitle
\tableofcontents

\section{Introduction}\label{sec:intro}
In this article, we extend the study of groups constructed using a framework due to Jones \cite{jonesunitary}, initiated by Brothier \cite{jonestechI, jonestechII, haagwreath}. This framework is called Jones' technology, discovered by Jones while studying the connections between conformal field theories and subfactor theory (see \cite{brothiersurvey} for a wonderful survey). It is a powerful method of generating actions of groups coming from categories, such as Richard Thompson's groups \cite{cfp, belkF} and, more generally, forest-skein groups \cite{fscI, fscII}.

The input to Jones' technology is a functor $\Phi: \cat{C} \to \cat{D}$, and the output is a functor $\widehat{\Phi}: \widehat{\cat{C}} \to \cat{D}$, where $\widehat{\cat{C}}$ is a groupoid obtained from including inverses to the morphisms of $\cat{C}$ (the Gabriel-Zisman localisation of $\cat{C}$ \cite{gz}). Restricting $\widehat{\Phi}$ to the automorphism groups of $\widehat{\cat{C}}$ yields group actions on objects of $\cat{D}$. The idea is to choose the category $\cat{C}$ such that these automorphism groups are interesting, while functors $\Phi: \cat{C} \to \cat{D}$ are easy to construct.

The most studied example is when $\cat{C}$ is the category $\cat{F}$ of binary forests, objects being natural numbers, and a morphism $m \to n$ being a binary forest with $m$ roots and $n$ leaves. Composition is given by connecting roots and leaves of forests. The automorphism group of an object $n \in \NN$ in $\widehat{\cat{F}}$ is the Higman-Thompson group $F_{2,n}$, first defined by Brown \cite{brown}. In particular, $F_{2,1}$ is equal to Thompson's group $F$, but $F_{2,n} \cong F_{2,1}$ for all $n \in \NN$ since $\widehat{\cat{F}}$ is connected.

Jones found a functor $\Phi: \cat{F} \to \cat{CT}$ into the category of Conway tangles that assigns a knot or link to each element of Thompson's group $F$, and showed that every knot and link arises in this way \cite{jonesunitary}. This sparked a connection between Thompson's groups, braid groups, and knot theory (see \cite{jonesknotsurvey, aiellosurvey, joneslectures} for an overview). Jones also defined the \textit{oriented subgroup} $\vec{F} \leq F$, and showed that every oriented link arises from an element of $\vec{F}$, up to distant unions with unknots \cite{jonesunitary}. Aiello then showed how to obtain every oriented link exactly \cite{aielloalexander}. Golan and Sapir \cite{golansapira, golansapirb} demonstrated that the group $\vec{F}$ is the first example of an infinite index maximal subgroup of $F$ that is not a stabiliser of a point in the unit interval \cite{savchuka, savchukb}, among having many other remarkable properties. Many interesting maximal subgroups of Thompson-related groups have since been discovered \cite{orientedbrownthompson, threecolourable, typesys, golanmaximal, golangen}, along with other interesting subgroups related to Jones' technology \cite{nogo, ren, eventhreecolourable, tlj, orientedbrownthompson}.

A trick to find many functors $\Phi: \cat{F} \to \cat{D}$ is to exploit the monoidal structure $\otimes$ on $\cat{F}$, given by horizontal concatenation of forests. If $\cat{D}$ is a monoidal category, a monoidal functor $\cat{F} \to \cat{D}$ is specified by an object $a \in \cat{D}$ and a morphism $\alpha: a \to a \otimes a$ (covariant), or a morphism $\omega: a \otimes a \to a$ (contravariant). Thus, such morphisms produce actions $F \curvearrowright K(\alpha)$, $K(\omega)$ on objects $K(\alpha)$, $K(\omega)$ of \cat{D}. We call these \textit{Jones actions}.

Taking $(\cat{D}, \otimes)$ to be the category of Hilbert spaces (morphisms being isometries) equipped with the direct sum yields a rich class of unitary representations of $F$ called Pythagorean representations \cite{pythagI, pythagII, pythagIII, pythagIV}. These provide a wealth of irreducible, pairwise non-isomorphic representations of $F$. They also witness the first known examples of representations of $F$ that are not the induction of finite-dimensional representations. Replacing the direct sum $\oplus$ with the usual tensor product $\otimes$ of Hilbert spaces, one obtains unitary representations of $F$ that can be extended to the larger Thompson's groups $F \subseteq T \subseteq V$, and offer concise proofs \cite{kazdhanhaagF} of existing results: Thompson's group $V$ does not have Kazdhan's property (T) \cite{rez, navas,gs}, and $T$ has the Haagerup property \cite{far}.

In this article, we are interested in monoidal functors $\cat{F} \to (\cat{Grp}, \times)$, the category of groups equipped with the direct product $\times$. Many of our results concern groups constructed using contravariant monoidal functors $\cat{F} \to \cat{Grp}$, and highlight the duality to the covariant case studied by Brothier in \cite{jonestechI, jonestechII, haagwreath}. A covariant monoidal functor $\cat{F} \to \cat{Grp}$ is specified by a group morphism $\alpha: \Gamma \to \Gamma \times \Gamma$. Jones' technology then outputs a functor $\widehat{F} \to \cat{Grp}$, and thus an action of Thompson's group $F$ on a group $K(\alpha)$. Extending this action to the larger Thompson's group $V$, the semidirect product $G(\alpha) := K(\alpha) \rtimes V$ produces a family of groups $\left(G(\alpha)\right)_{\alpha: \Gamma \to \Gamma^2}$ that has been been studied independently by Brothier \cite{jonestechI, jonestechII, haagwreath}, Tanushevski \cite{tan16}, and Witzel-Zaremsky \cite{wz}.

If $\alpha: \Gamma \to \Gamma^2, g \mapsto (\beta(g),e)$ for some $\beta \in \aut(\Gamma)$, the group $G(\alpha)$ is isomorphic to a restricted twisted wreath product $\Gamma \wr V = \oplus_{\QQ_2} \Gamma \rtimes V$ \cite{haagwreath, jonestechI}, where Thompson's group $V$ acts classically on the dyadic rationals $\QQ_2$ contained in the unit interval \cite{cfp, belkF}. It was shown in \cite{haagwreath} that $G(\alpha)$ has the Haagerup property when $\Gamma$ does, providing the first examples of finitely presented wreath products having the Haagerup property for a non-trivial reason, since $V$ is non-amenable and $\QQ_2$ is infinite. The action $\pi: V \curvearrowright \oplus_{\QQ_2} \Gamma$ is twisted using the automorphism $\beta \in \aut(\Gamma)$ as follows: 
\[\pi(v)(a)(x) = \beta^{\log_2(v'(v^{-1}x))}(a(v^{-1}x)), \ v \in V, \ a \in \oplus_{\QQ_2} \Gamma, \ x \in \QQ_2. \tag{$*$} \]
Here we interpret an element of Richard thompson's group $V$ as a piecewise linear bijection of the unit interval, letting $v'(x)$ be the slope of $v$ at $x \in \QQ_2$ \cite{cfp}. The slope of an element $v \in V$ at any dyadic rational is a power of $2$, ensuring an integer power of $\beta$ in $(*)$.

In \cite{jonestechI}, it was shown that these wreath products $G(\alpha)$ built from group automorphisms $\beta \in \aut(\Gamma)$ exhibit a remarkable rigidity property. If $\beta$ and $\wt{\beta}$ are automorphisms of groups $\Gamma$ and $\wt{\Gamma}$, and $G$, $\wt{G}$ are the wreath products built using $\beta$ and $\wt{\beta}$ respectively, then we have that $G \cong \wt{G}$ if and only if there exists an isomorphism $\gamma: \Gamma \to \wt{\Gamma}$ and an element $\wt{h} \in \wt{\Gamma}$ such that 
    \[\wt{\beta} = \ad(\wt{h}) \circ \gamma \beta \gamma^{-1},\]
where $\ad(\wt{h})$ is the inner automorphism $\wt{\Gamma}$ corresponding to $\wt{h}.$ I.e., $G$ is isomorphic to $\wt{G}$ if and only if $\beta$ and $\wt{\beta}$ are outer-conjugate.

This rigidity follows from the striking result that any isomorphism $\theta: G \to \wt{G}$ is \textit{spatial}, meaning that $\theta$ restricts to an isomorphism $\kappa: \oplus_{\QQ_2} \Gamma \to \oplus_{\QQ_2} \wt{\Gamma}$, and $\kappa$ preserves support of elements of $\oplus_{\QQ_2} \Gamma$, up to a bijection of $\mathbb{Q}_2$. More precisely, there is a isomorphism $\kappa: \oplus_{\QQ_2} \Gamma \to \oplus_{\QQ_2} \wt{\Gamma}$ and a bijection $\varphi: \mathbb{Q}_2 \to \mathbb{Q}_2$ satisfying \[\theta(a) = \kappa(a)\] and \[\supp(\kappa(a)) = \varphi(\supp(a))\] for all $a \in \oplus_{\QQ_2} \Gamma$. This result also led to a complete description of the automorphism group of $G(\alpha)$ in the untwisted case $\beta = \text{id}.$

In \cite{jonestechII}, isomorphisms between groups $G(\alpha)$ constructed using two automorphisms of a group $\Gamma$ were observed to be spatial, up to a multiple by a centre-valued group morphism. Here $\alpha: \Gamma \to \Gamma^2, g \mapsto (\alpha_0(g), \alpha_1(g))$, with $\alpha_0, \alpha_1 \in \aut(\Gamma)$. In this case, $G(\alpha)$ is a semidirect product $L \Gamma \rtimes V$, where $L\Gamma$ denotes the continuous maps from the Cantor space $\cantor := \{0,1\}^{\NN}$ to $\Gamma$ (a discrete group). These groups appear naturally in field theories where physical space is approximated by $\QQ_2$ \cite{bsquant, bsgauge}.

\subsection{Main results}
In this article, we extend the rigidity results for wreath products and split extensions of $V$ by loop groups $L \Gamma$ obtained by Brothier in \cite{jonestechI, jonestechII}: \textit{any} isomorphism $\theta: G(\alpha) \to G(\wt{\alpha})$ is spatial modulo a centre-valued morphism, even when $\alpha: \Gamma \to \Gamma^2$ and $ \wt{\alpha}: \wt{\Gamma} \to \wt{\Gamma}^2$ are not specified by automorphisms of $\Gamma$ and $\wt{\Gamma}$.
    \begin{theoremA}\label{thm:covariantrigidity} Suppose that $\alpha: \Gamma \to \Gamma^2$ and $ \wt{\alpha}: \wt{\Gamma} \to \wt{\Gamma}^2$ are group morphisms, and that $\theta: G(\alpha) \to G(\wt{\alpha})$ is an isomorphism of groups. Then there exists an isomorphism $\kappa^0: K(\alpha) \to K(\wt{\alpha})$, a group morphism $\zeta: G(\alpha) \to Z\left(G(\wt{\alpha})\right)$, and a homeomorphism $\varphi$ of the Cantor space $\cantor = \{0,1\}^{\mathbb{N}}$ such that \[\theta(a) = \kappa^0(a) \cdot \zeta(a)\] and \[\supp(\kappa^0(a)) = \varphi(\supp(a)) \text{, for all $a \in K(\alpha)$. }\] Moreover, the map \[\theta^0: G(\alpha) \to G(\wt{\alpha}), av \mapsto \kappa^0(a) \cdot \theta(v)\] is an isomorphism of groups.
    \end{theoremA}
The elements $a \in K(\alpha)$ are not always obviously groups of functions on some set, so another notion of support must be used (Definition \ref{def:gensupp}, Remark \ref{rmk:covariantmonoidalexample}). The support of an element $a \in K(\alpha)$ is always a closed subset of the Cantor space $\cantor = \{0,1\}^{\mathbb{N}}$; for the wreath products $\oplus_{\QQ_2} \Gamma \rtimes V$ of \cite{jonestechI}, once we identify $\mathbb{Q}_2$ with the finitely supported sequences in $\cantor$, the usual notion of support for elements of $\oplus_{\QQ_2} \Gamma$ is recovered.

The proof of Theorem \ref{thm:covariantrigidity} relies on the following observation. We know that each group morphism $\alpha: \Gamma \to \Gamma^2$ defines a monoidal functor $\Phi_\alpha: \cat{F} \to \cat{Grp}$, and Jones' technology outputs a functor $\widehat{\Phi}_\alpha: \widehat{\cat{F}} \to \cat{Grp}$. We then observe that the functor $\widehat{\Phi}_\alpha$ is monoidal as well; there is an isomorphism $R(\alpha): K(\alpha) \to K(\alpha)^2$ of groups such that $\widehat{\Phi}_\alpha = \Phi_{R(\alpha)}$. The functor $\Phi_{R(\alpha)}: \fg{\cat{F}} \to \cat{Grp}$ is defined in the same way as $\Phi_\alpha: \cat{F} \to \cat{Grp}$, but can be obviously extended to $\fg{\cat{F}}$ since $R(\alpha)$ is an isomorphism.

Jones' technology for covariant monoidal functors $\cat{F} \to (\cat{Grp}, \times)$ can then be described as assigning each group morphism $\alpha: \Gamma \to \Gamma^2$ a group \textit{isomorphism} $R(\alpha): K(\alpha) \to K(\alpha)^2$, the best approximation to $\alpha$ by an isomorphism. The functor $\Phi_{R(\alpha)}: \widehat{\cat{F}} \to \cat{Grp}$ then restricts to an action $F \curvearrowright K(\alpha)$, which we extend to the larger Thompson group $V$, obtaining the group $G(\alpha) := K(\alpha) \rtimes V$.

Thus, the class of groups $G(\alpha) = K(\alpha) \rtimes V$ obtained from group morphisms $\alpha: \Gamma \to \Gamma^2$ is precisely the class of groups $K \rtimes V$ obtained from group isomorphisms $R: K \to K^2$. This lets us rephrase Theorem \ref{thm:covariantrigidity}.
    \begin{theoremA}\label{thm:genrigidity}
    Suppose that $R: K \to K^2$ and $\wt{R}: \wt{K} \to \wt{K}^2$ are group isomorphisms, and that $\theta: G \to \wt{G}$ is an isomorphism of groups, where $G := K \rtimes V$ and $\wt{G} := \wt{K} \rtimes V.$ Then there exists an isomorphism $\kappa^0: K \to \wt{K}$, a group morphism $\zeta: G \to Z \wt{G}$, and a homeomorphism $\varphi$ of the cantor space such that \[\theta(a) = \kappa^0(a) \cdot \zeta(a),\] and \[\supp(\kappa ^0(a)) = \varphi(\supp(a)) \text{ for all $a \in K$.}\] Moreover, the map $\theta^0: G \to \wt{G}$ defined by the formula \[\theta^0(av) := \kappa^0(a) \cdot \theta(v)\] is an isomorphism of groups.
    \end{theoremA}
The remainder of the article is an exploration of Jones' technology applied to \textit{contravariant} monoidal functors $\cat{F} \to (\cat{Grp}, \times)$, which is dual to the covariant case studied in \cite{jonestechI, jonestechII, haagwreath} by Brothier. Rather than a group morphism $\alpha: \Gamma \to \Gamma^2$, a contravariant monoidal functor $\cat{F} \to (\cat{Grp}, \times)$ is specified by a group morphism $\omega: \Gamma^2 \to \Gamma$. We introduce the group $K(\omega)$ and an isomorphism $R(\omega): K(\omega) \to K(\omega)^2$, which can be thought of as the best approximation to $\omega$ by an isomorphism. The isomorphism $R(\omega)$ then defines a functor $\widehat{\cat{F}} \to \cat{Grp}$, and thus an action of Thompson's group $V \curvearrowright K(\omega)$, obtaining a group $G(\omega) := K(\omega) \rtimes V.$ This is the dual process of constructing an action $V \curvearrowright K(\alpha)$ from a group morphism $\alpha: \Gamma \to \Gamma^2$.

We aim to answer the following: do the groups $G(\omega)$ admit a nice description, analagous to the groups $G(\alpha)$ constructed using covariant functors? Do they have nice rigidity properties, and do their automorphism groups have a nice characterisation? Answers to the second and third questions will be provided by applications of Theorem \ref{thm:genrigidity}.

Our starting point for the first question will be group morphisms $\omega: \Gamma^2 \to \Gamma, (g,h) \mapsto \beta(g)$, for some automorphism $\beta \in \aut(\Gamma).$ In Section \ref{sec:wreath} we show that $G(\omega) = K(\omega) \rtimes V$ is isomorphic to the \textit{unrestricted} permutational wreath product $\prod_{\QQ_2} \Gamma \rtimes V$ with action $\pi: V \curvearrowright \prod_{\QQ_2} \Gamma$ given by the formula \[\pi(v)(a)(x) = \beta^{-\log_2(v'(v^{-1}x))}(a(v^{-1}x)), \ v \in V, \ a \in \prod_{\QQ_2} \Gamma, \ x \in \QQ_2. \tag{$**$}\] Notice the similarity between the formulae $(*)$ and $(**)$.

Such a group $G(\omega) = K(\omega) \rtimes V$ comes from an isomorphism $R(\omega): K(\omega) \to K(\omega)^2$, so we may apply Theorem \ref{thm:genrigidity} to start classifying these wreath products up to isomorphism, and study their automorphisms. Indeed, in Section \ref{sec:wreath}, we use Theorem \ref{thm:genrigidity} to prove that these unrestricted wreath products have the same classification as their restricted counterparts.
\begin{theoremB}\label{thm:thinclassificationintro}
    Suppose that $\Gamma, \wt{\Gamma}$ are groups and that $\beta, \wt{\beta}$ are automorphisms of $\Gamma$ and $\wt{\Gamma}$, respectively. This yields group morphisms $\omega: \Gamma^2 \to \Gamma, (g,h) \mapsto \beta(g)$, and $\wt{\omega}: \wt{\Gamma}^2 \to \wt{\Gamma}, (\wt{g}, \wt{h}) \mapsto \wt{\beta}(\wt{g}).$
    
    We have that $G(\omega) \cong G(\wt{\omega})$ if and only if $\wt{\beta} = \ad(\wt{h}) \circ \gamma \beta \gamma^{-1}$ for some $\wt{h} \in \wt{\Gamma}$ and some group isomorphism $\gamma: \Gamma \to \wt{\Gamma}.$
\end{theoremB}
At the end of Section \ref{sec:wreath} we relax the condition that $\beta \in \aut(\Gamma)$ to $\beta$ being an aribtrary endomorphism of $\Gamma$, at the cost of weakening the classification of Theorem \ref{thm:thinclassificationintro}. We show that if $\beta \in \text{End}(\Gamma)$, and $\omega: (g,h) \mapsto \beta(g)$, there exists a group $\varprojlim \Gamma$ and an automorphism $\varprojlim \beta \in \aut(\varprojlim \Gamma)$ satisfying $G(\omega) \cong G(\varprojlim \omega)$, where $\varprojlim \omega: \varprojlim \Gamma^2 \to \varprojlim \Gamma$ is defined in the same way using the automorphism $\varprojlim \beta.$ Theorem \ref{thm:thinclassificationintro} yields the following weaker classification of the groups $G(\omega)$ coming from endomorphisms $\beta \in \text{End}(\Gamma).$
\begin{theoremB}\label{thm:weakclassification}
    Suppose that $\Gamma, \wt{\Gamma}$ are groups, and that $\beta, \wt{\beta}$ are endomorphisms of $\Gamma$ and $\wt{\Gamma}$, respectively. Consider the group morphisms $\omega: \Gamma^2 \to \Gamma, (g,h) \mapsto \beta(g)$, and $\wt{\omega}: \wt{\Gamma}^2 \to \wt{\Gamma}, (\wt{g}, \wt{h}) \mapsto \wt{\beta}(g).$ Then \[G(\omega) \cong G(\wt{\omega}) \text{ if and only if } \varprojlim \wt{\beta} = \ad(\wt{h}) \circ \gamma  \circ \varprojlim \beta \circ \gamma^{-1} \tag{$\lozenge$} \] for some $\wt{h} \in \varprojlim \Gamma $ and some isomorphism $\gamma: \varprojlim \Gamma \to \varprojlim \wt{\Gamma}.$
\end{theoremB} 
Dually, in \cite{jonestechI} it was shown that for each group $\Gamma$ and endomorphism $\beta \in \Gamma$, there exists a group $\varinjlim \Gamma$ and an automorphism $\varinjlim \beta$ of $\varinjlim \Gamma$ such that $G(\alpha) \cong G(\varinjlim \alpha)$, where $\alpha: \Gamma \to \Gamma^2, g \mapsto (\beta(g), e)$, and $\varinjlim \alpha: \varinjlim \Gamma \to \varinjlim \Gamma^2$ is defined in the same way using the automorphism $\varinjlim \beta \in \aut\left(\varinjlim \Gamma\right)$. Moreover, Brothier shows that if $\wt{\Gamma}$ is another group and $\wt{\beta} \in \mathrm{End}(\wt{\Gamma})$, then
\[G(\alpha) \cong G(\wt{\alpha}) \text{ if and only if } \varinjlim \wt{\beta} = \ad(\wt{h}) \circ \gamma \circ \varinjlim \beta \circ  \gamma^{-1} \tag{$\lozenge \lozenge$} 
\]
for some $\wt{h} \in \varinjlim \wt{\Gamma}$ some isomorphism $\gamma: \varinjlim \Gamma \to \varinjlim \wt{\Gamma}$. 

The classifications $(\lozenge)$ and $(\lozenge \lozenge)$ should be thought of as dual. Indeed, if $\Gamma = \ZZ$ and $\beta \in \text{End}(\ZZ)$ is given by multiplication by a natural number $q \geq 2$, then $\varinjlim \Gamma \cong \ZZ[\frac{1}{q}]$, and $\varinjlim \beta \in \aut\left(\ZZ[\frac{1}{q}]\right)$ is again given by multiplication by $q$ \cite[Example 3.4]{jonestechI}. However, $\varprojlim \Gamma$ is the trivial group (Example \ref{eg:triviallimgrp}). If we instead consider the dual endomorphism $\widehat{\beta}: S^1 \to S^1, z \mapsto z^q$, then $\varprojlim \Gamma \cong \widehat{\ZZ[\frac{1}{q}]}$, the Pontryagin dual of $\ZZ[\frac{1}{q}]$, and $\varprojlim \widehat{\beta}$ is the automorphism of $\widehat{Z[\frac{1}{q}]}$ obtained by precomposing characters with $\varinjlim \beta$ (Example \ref{eg:pontduality}). Symbolically, $\varprojlim \widehat{\beta} = \widehat{\varinjlim \beta}.$

The next point of entry for describing groups $G(\omega)$ coming from group morphisms $\omega: \Gamma^2 \to \Gamma$ are those in which $\Gamma$ is abelian, and $\omega: (g,h) \mapsto \beta_0(g) \cdot \beta_1(h)$ for some automorphisms $\beta_0, \beta_1 \in \Gamma$. A nice description and classification of these groups $G(\omega)$ up to isomorphism is desirable but could be difficult, since the analagous question for group morphisms $\alpha: \Gamma \to \Gamma^2, g \mapsto (\beta_0(g), \beta_1(g))$ was only partially answered in \cite{jonestechII}. We do not consider these groups in this article, and this would be a natural direction for future work.

Our final application of Theorem \ref{thm:genrigidity} is a decomposition of $\aut(G)$ into an iterated semidirect product of groups $A_1, \dots, A_6$, where $G$ is an unrestricted, untwisted wreath product $\prod_{\QQ_2} \Gamma \rtimes V.$ Here $G \cong G(\omega)$, where $\omega: \Gamma^2 \to \Gamma, (g,h) \mapsto g.$ Using the notation of Theorem \ref{thm:genrigidity}, given an automorphism $\theta \in \aut(G)$, we may write
\[\theta(a) = \kappa^0(a) \cdot \zeta(a), \ a \in \prod_{\mathbb{Q}_2} \Gamma,\] where 
\[
\supp(\kappa^0(a)) = \varphi(\supp(a))
\] for all $a \in \prod_{\mathbb{Q}_2} \Gamma$. One can show that $\varphi$ is uniquely determined, yielding a map $\epi_1: \aut(G) \to \mathrm{Homeo}(\cantor)$, the homeomorphisms of the Cantor space $\cantor = \{0,1\}^{\mathbb{N}}$. It turns out that $\epi_1$ is a split epimorphism onto its image, allowing us to write $\aut(G) = \ker \epi_1 \rtimes A_1$, where $A_1 := \mathrm{im}(\epi_1).$ We then define a split epimorphism $\epi_2: \ker \epi_1 \to \aut(\Gamma)$, obtaining that $\aut(G) = (\ker \chi_2 \rtimes A_2) \rtimes A_1$, where $A_2 := \aut(\Gamma).$ Continuing in this fashion, we arrive at a decomposition of $\aut(G).$ 

\begin{theoremB}\label{thm:introautdecomp}
    We have the decomposition \[\aut(G) =    ((((A_6\rtimes A_5) \rtimes A_4)\rtimes A_3)\rtimes A_2)\rtimes A_1, \] where
\begin{itemize}
    \item $A_1$ is the group of homeomorphisms of the Cantor space which stabilise $\mathbb{Q}_2$ and normalise $V$,
    \item $A_2 = \aut(\Gamma)$,
    \item $A_3 = Z \Gamma$,
    \item $A_4 = \left(\{h \in \prod_{\mathbb{Q}_2} \Gamma \ | \ h(0) \in Z \Gamma \}\right) / Z \Gamma$,
    \item $A_6 = \text{Hom}(G,ZG).$
\end{itemize} The group $A_5$ consists of the automorphisms $\theta \in \aut(G)$ satisfying the following conditions:
\begin{itemize}
    \item $\supp(\theta(a)) = \supp(a)$ for all $a \in K$,
    \item $\theta$ restricts to the identity on $\oplus_{\QQ_2} \Gamma$ and $V.$
\end{itemize}
\end{theoremB}
Hereit makes sense to say that $A_1$ is contained in the normaliser of $V$ in $\mathrm{Homeo}(\cantor)$, since $V$ acts faithfully on the Cantor space $\cantor = \{0,1\}^{\mathbb{N}}$ via prefix replacement. This restricts to the action of $V$ on the dyadic rationals $\mathbb{Q}_2 = \{0,1\}^{(\mathbb{N})}$. Furthermore, $Z\Gamma$ is viewed as the group of constant maps $\QQ_2 \to Z\Gamma$ for the quotient appearing in $A_4$. Finally, for each $a \in K$, we have that \[\supp(a) = \overline{\{x \in \QQ_2 \ | \ a(x) \neq e_\Gamma\}},\] where the closure is taken in the Cantor space $\cantor$ with respect to the Tychonoff topology. 

The group $A_5$ appearing in Theorem \ref{thm:introautdecomp} is mysterious. In Remark \ref{rmk:exoticauto} we note that for each $\theta \in A_5$, there exists a group morphism $\eta: \prod_{\QQ_2} \Gamma \to \prod_{\mathbb{Q}_2} Z \Gamma$ such that
\[\theta(a) = a \cdot \eta(a) \text{ for all } a \in K.\]
It follows that if either $\Gamma$ is a perfect group or $Z\Gamma$ is trivial it must be that $A_5 = \{\text{id}_G\}$, yielding a more complete description of $\aut(G)$. It is unknown however if the group $A_5$ can be non-trivial in the simplest case.
\begin{question}\label{question:z2auto}
    If $\Gamma = \ZZ_2$, does there exist a non-trivial automorphism in $A_5?$
\end{question}
In this case we can identify the group $\prod_{\QQ_2} \Gamma$ with the group $\mathcal{P}(\QQ_2)$ of all subsets of $\QQ_2$ equipped with symmetric difference $\Delta$. The action $V \curvearrowright \mathcal{P}(\QQ_2)$ is the pointwise action $vX = \{v(x) \ | \ x \in X\}$ for all $v \in V$ and $X \subseteq \QQ_2.$ This lets us rephrase Question \ref{question:z2auto}:
\begin{question}\label{question:z2autorephrased}
    Is there a bijection $\kappa: \mathcal{P}(\QQ_2) \to \mathcal{P}(\QQ_2)$ satisfying the following?
    \begin{itemize}
        \item $\kappa(X \Delta Y) = \kappa(X) \Delta \kappa(Y)$,
        \item $\kappa(vX) = v\kappa(X)$,
        \item $\overline{\kappa(X)} = \overline{X}$,
        \item $\kappa(Z) = Z$,
    \end{itemize} for all $X, Y \subseteq \QQ_2, \ v \in V$ and finite subsets $Z \subseteq \QQ_2.$ Here the closure $\overline{X}$ is taken in $\cantor = \{0,1\}^{\mathbb{N}}$.
\end{question} 
Theorems \ref{thm:covariantrigidity}, \ref{thm:genrigidity}, \ref{thm:thinclassificationintro}, \ref{thm:weakclassification}, and \ref{thm:introautdecomp} all concern semidirect products in which Thompson's group $V$ acts. We do not consider the cases when the smaller Thompson's groups $F$ and $T$ are acting, but it would be interesting to see the extent to which our results hold in these cases.
\subsection{Other results on wreath products}
Recall that if $\Gamma$ and $B$ are groups, the restricted standard wreath product of $\Gamma$ by $B$ is the semidirect product $\oplus_{B} \Gamma \rtimes B$, where $B \curvearrowright B$ by left multiplication, and $B \curvearrowright \oplus_{B} \Gamma$ by translation. The corresponding unrestricted wreath product is defined using the direct product instead of direct sum. We call $\Gamma$ the bottom group, $\oplus_{B} \Gamma$ the base group, and $B$ the acting group. In \cite{neuwreath}, Neumann asks the question: ``when can two standard wreath products be isomorphic?" Neumann found a precise answer, that two standard wreath products, restricted or unrestricted, are isomorphic if and only if their bottom and acting groups are pairwise isomorphic.

In \cite{jonestechI}, Brothier proved that every isomorphism between two twisted permutational restricted wreath products $G = \oplus_{\QQ_2} \Gamma \rtimes V$, $\wt{G} = \oplus_{\QQ_2} \wt{\Gamma} \rtimes V$ is \textit{spatial}, i.e. preserves supports of elements of the base groups up to a homeomorphism of the Cantor space $\cantor = \{0,1\}^{\mathbb{N}}$ which stabilises the dyadic rationals $\QQ_2$. Thus, the image of a coordinate subgroup $\Gamma_x \leq \oplus_{\QQ_2} \Gamma$ (maps $\QQ_2 \to \Gamma$ supported at a fixed point $x \in \QQ_2$) is again a coordinate subgroup in $\oplus_{\QQ_2} \wt{\Gamma}$. This establishes an isomorphism $\Gamma \cong \wt{\Gamma}$, and with some work, a relation satisfied by the automorphisms of $\Gamma$ and $\wt{\Gamma}$ used to twist the actions $V \curvearrowright \oplus_{\QQ_2} \Gamma$, $\oplus_{\QQ_2} \wt{\Gamma}$. 

This is comparable to the approach of Neumann, who analysed images of coordinate subgroups in standard wreath products $\oplus_{B} \Gamma \rtimes B$ under isomorphisms. These isomorphisms are not necessarily spatial, so the image of a coordinate subgroup is not always a coordinate subgroup. Unhindered by this, from an isomorphism 
\[\theta: \oplus_{B} \Gamma \rtimes B \to \oplus_{\wt{B}} \wt{\Gamma} \rtimes \wt{B}\]
between two restricted standard wreath products, Neumann obtains an isomorphism $\Gamma \to \wt{\Gamma}$ as follows. Fix $b \in B$ and let $g \in \Gamma$, interpreted as an element of the coordinate subgroup of $\oplus_{B} \Gamma$ at $b$. Take the product of all of the non-trivial values of the map $\wt{B} \to \wt{\Gamma}$ obtained from $\theta(g)$ with respect to a fixed order. This defines an isomorphism $\Gamma \to \wt{\Gamma}$. It is striking that one doesn't need to take such products for the restricted permutational wreath products considered by Brothier in \cite{jonestechI}, due to the rigid structure of the isomorphisms between them.

To conclude that an isomorphism between \textit{unrestricted} standard wreath products descends to an isomorphism of the bottom groups, Neumann takes a different approach, using the fact that in an unrestricted standard wreath product $G = \prod_{B} \Gamma \rtimes B$, any two complements of the base group $\prod_{B} \Gamma$ are conjugate. This is no longer the case in general for the permutational wreath products that we consider (Remark \ref{rmk:autpermwreath}). However, by Theorem \ref{thm:genrigidity}, one can extract a \textit{spatial} isomorphism from any isomorphism $\theta: \prod_{\QQ_2} \Gamma \rtimes V \to \prod_{\QQ_2} \wt{\Gamma} \rtimes V$ between two unrestricted, twisted and permutational wreath products coming from Jones' technology. This lets us mimic the approach for restricted wreath products, considering images of coordinate subgroups to establish that $\Gamma \cong \wt{\Gamma}$ and obtain the classification of Theorem \ref{thm:thinclassificationintro}.

We emphasise that Neumann's results apply to standard (untwisted) wreath products constructed using arbitrary base and acting groups, while in \cite{jonestechI} and the present article, we leave our acting group fixed (Thompson's group $V$). Our wreath products are not standard (i.e. permutational), since we use the action $V \curvearrowright \QQ_2$, and we may twist the action of $V$ on the base group by an automorphism of the bottom group.

In \cite{autstandwreath}, Houghton decomposes the automorphism group $\aut(W)$ of an unrestricted standard wreath product $\prod_{B} \Gamma \rtimes B$, with $B \curvearrowright B$ via left multiplication, with an eye towards the diagonal subgroup $D \leq \prod_{B} \Gamma$. More precisely, Houghton gives the decomposition 
\[
\aut(W) = ( \left(A_{D, B} \rtimes \aut(\Gamma)\right) \cdot I) \rtimes \aut(B), \tag{$\clubsuit$}
\]
where $A_{D, B}$ denotes the automorphisms of $W$ which fix the diagonal subgroup $D \leq \prod_{B} \Gamma$ and $B$ pointwise, and $I$ denotes the inner automorphisms of $W$ induced by elements of $\prod_{B} \Gamma$. In \cite{autpermwreath}, Hassanabadi shows that a similar decomposition can be done for wreath products $W = \prod_{B/H} \Gamma \rtimes B$, where $H \leq B$, for a certain subgroup of $\mathcal{B}(W) \leq \aut(W)$ (see Remark \ref{rmk:autpermwreath}).

When one takes $B := V$ and $H := \mathrm{Stab}(x)$ for some $x \in \mathbb{Q}_2$, the group $W$ becomes the untwisted wreath product $G = \prod_{\QQ_2} \Gamma \rtimes V$ of Theorem \ref{thm:introautdecomp}. While the decomposition of $(\clubsuit)$ focuses on the diagonal subgroup of $\prod_{\QQ_2} \Gamma$, the decomposition of $\aut(G)$ in Theorem \ref{thm:introautdecomp} uses the fact that automorphisms of $G$ act nicely on the subgroup $\oplus_{\QQ_2} \Gamma \leq G$, which comes from Theorem \ref{thm:genrigidity}. In particular, one can use Theorem \ref{thm:introautdecomp} to give an alternative description of $\mathcal{B}(G)$ than \cite{autpermwreath} (Remark \ref{rmk:autpermwreath}).

In \cite{jonestechI}, Brothier completely describes the automorphism group of the restricted wreath product $G_0 := \oplus_{\QQ_2} \Gamma \rtimes V$ by identifying four subgroups $\wt{A}_1, \wt{A}_2, \wt{A}_3, \wt{A}_4 \leq \aut(G_0)$ that generate $\aut(G_0)$, whose elements are referred to as \textit{elementary} automorphisms. Brothier then constructs a group $Q = (\wt{A}_4 \times \wt{A}_3) \rtimes (\wt{A}_2 \times \wt{A}_1)$ and an epimorphism $\Xi : Q \to \aut(G_0)$ with a small kernel, so that every automorphism of $G_0$ can be expressed as a product of elementary automorphisms in an essentially unique way.

In Remark \ref{rmk:restrictedautembedding}, we mention that $\aut(G_0)$ can be decomposed in a similar way to Theorem \ref{thm:introautdecomp} as an iterated semidirect product
\[
\aut(G_0) =  ((C \rtimes \wt{A}_3)\rtimes \wt{A}_2)\rtimes \wt{A}_1,
\] where $C$ is a large subgroup of $\wt{A}_4$. Thanks to the fact that automorphisms of $G_0$ are spatial, originally proved in \cite{jonestechI} and recovered by Theorem \ref{thm:genrigidity}, every automorphism of $G_0$ can be extended to an automorphism of $G = \prod_{\mathbb{Q}_2} \Gamma \rtimes V$. This results in an embedding $\iota: \aut(G_0) \hookrightarrow \aut(G)$ that maps $\wt{A}_1$ to $A_1$, $\wt{A}_2$ to $A_2$, and embeds $C$ into $A_4$. The image of $\wt{A}_3$ under $\iota$ is not always contained in $A_3$, since elements of $A_3$ do not always restrict to automorphisms of $G_0$.

Since $\wt{A}_i$ can be identified with $A_i$ via the embedding $\iota$ for $i = 1, 2$, and the formulae for automorphisms in $A_3$ and $A_4$ closely resemble those of automorphisms in $\wt{A}_3$ and $\wt{A}_4$ respectively, we may regard the groups $A_1,\dots,A_4$ as consisting of elementary automorphisms of $G.$ We also consider the elements of $A_6$ as elementary due to their simple formulae. Question \ref{question:z2autorephrased} then asks if there are any ``exotic" automorphisms of $G$.

\subsection{Covariant vs. contravariant functors}
Applying Jones' technology to group morphisms $\alpha: \Gamma \to \Gamma^2$ and $\omega: \Gamma^2 \to \Gamma$ yields two classes of groups $G(\alpha) = K(\alpha) \rtimes V$ and $G(\omega) = K(\omega) \rtimes V$. These classes are the same; they are both the class of groups $K \rtimes V$ obtained from group isomorphisms $R: K \to K \times K$. However, the novelty of Jones' technology is to construct interesting examples of these groups $K \rtimes V$ easily. This is done by choosing group morphisms $\alpha$ and $\omega$ that are easy to write down, e.g. using an automorphism $\beta$ of a group $\Gamma$. In these situations, the kinds of groups $G(\alpha)$ and $G(\omega)$ one gets from Jones' technology are dual.

This is illustrated by fixing $\alpha: g \mapsto (g,e)$ and $\omega: (g,h) \mapsto g$. The group $G(\alpha)$ is a restricted wreath product $\oplus_{\QQ_2} \Gamma \rtimes V$ \cite{jonestechI, haagwreath}, while $G(\omega)$ is the corresponding unrestricted wreath product $\prod_{\QQ_2} \Gamma \rtimes V$ (Corollary \ref{cor:wreathiso}). The group $G(\omega)$ is always uncountable, thus never finitely generated, however $G(\alpha)$ is of type $\mathrm{F}_n$ if $\Gamma$ is, for all $n \geq 1$. Here $\mathrm{F}_1$ means finitely generated and $\mathrm{F}_2$ means finitely presented, the other $\mathrm{F}_n$ being stronger, higher-dimensional analogues \cite[Section 7]{geoghegan}. These good finiteness properties of $G(\alpha)$ follow from the work of Cornulier for $n \leq 2$ \cite[Theorem 1.1]{presentedwreath}, and Bartholdi-Cornulier-Kochloukova \cite[Proposition 4, Corollary 2]{fpwreath} for $n > 2$. They can also be obtained by adapting \cite{fscII} a different approach of Tanushevski \cite{tanthesis}, outlined by Witzel-Zaremsky \cite[Section 6]{wz}. This all fails for standard restricted wreath products $\Gamma \wr B$, which are never finitely presented as soon as $B$ is infinite, by a result of Baumslag \cite{baumwreath}.

In general, $K(\alpha)$ is a direct limit of powers of $\Gamma$, while $K(\omega)$ is an inverse limit of powers of $\Gamma$. If $\Gamma$ is finite, $K(\alpha)$ is discrete, while $K(\omega)$ is profinite. Thus, one can think of $G(\alpha)$ as a discrete object, and $G(\omega)$ as a locally compact object. There are also many examples of dense embeddings $G(\alpha) \hookrightarrow G(\omega)$, analagous to the embedding of a restricted wreath product into an unrestricted one (Remark \ref{rmk:functoriality}).

A nice feature of Theorem \ref{thm:genrigidity} is that it applies to both constructions of groups $G(\alpha)$ and $G(\omega)$, extending the settings in which we can establish rigidity properties, and study automorphism groups.
\subsection{Plan of the article}
Section \ref{sec:prelims} introduces the necessary background for Thompson's groups and Jones' technology. In Section \ref{sec:characteristic} we analyse the structure of isomorphisms between semidirect products $K \rtimes V$ coming from group isomorphisms $R: K \to K^2$, the main result being Theorem \ref{thm:isodecomp}. We conclude Section \ref{sec:characteristic} by constructing obvious isomorphisms between these groups.

In Subsection \ref{subsec:unipropfunc} we define the groups $G(\omega) = K(\omega) \rtimes V$ coming from group morphisms $\omega: \Gamma^2 \to \Gamma$, show that this construction is functorial, and give sufficient conditions for isomorphism, extending the results of Subsection \ref{subsec:obviousiso}. In Subsection \ref{subsec:wreathproducts}, we show that for certain morphisms $\omega$, the group $G(\omega)$ is an unrestricted twisted wreath product, and classify these wreath products up to isomorphism using Theorem \ref{thm:isodecomp}. Section \ref{sec:wreath} concludes with an extension of this classification to a larger class of morphisms $\omega$ in Subsection \ref{subsec:endtoaut}. 

Section \ref{sec:automorphisms} features a decomposition of the automorphism groups of these wreath products that are untwisted.

\newpage
\begin{center}{\textbf{Acknowledgements}}\end{center}
I thank Arnaud Brothier for the many meetings and invaluable feedback on drafts. I also thank Ian doust for his writing course and helpful comments on an early version of this article. I am grateful for useful exchanges with Ryan Seelig and Dilshan Wijesena, and for Renee Lim's proofreading. Most diagrams were made using TikZit, and I thank the contributors for developing this incredible software.

\section{Preliminaries}\label{sec:prelims}
Jones' technology takes a functor $\Phi : \catname{C} \to \catname{D}$, and produces a functor $\widehat{\Phi}: \fg{\catname{C}} \to \catname{D}$. Here $\fg{\catname{C}}$ is a groupoid formed by formally inverting the morphisms in $\catname{C}$, called the localisation of \cat{C}. The functor $\widehat{\Phi}$ then restricts to representations of the automorphism groups at the objects of $\fg{\catname{C}}$. Ideally, these automorphism groups are interesting (e.g. Thompson's groups), and functors that produce interesting representations are easy to construct.

The main points we want to illustrate are:
\begin{itemize}
    \item if \cat{C} = \cat{F}, the category of binary forests, and \cat{D} is a monoidal category, a monoidal functor $\cat{C} \to \cat{D}$ is encoded by a morphism $\alpha: a \to a \otimes a$ or $\omega: a \otimes a \to a$ in \cat{D}. 
    \item In this context, Jones' technology outputs the best approximation to such a morphism by an isomorphism $R: K \xrightarrow{\sim} K \otimes K$ in \cat{D}.
    \item The result is a representation of Thompson's group $F$ in \cat{D}, which has a straightforward description in terms of the isomorphism $R$. If \cat{D} is symmetric monoidal, this representation can be extended to Thompson's group $V$.
\end{itemize} 
\subsection{Localisation}
\begin{definition}\label{def:localisation}
A \textit{localisation} of a category \cat{C} is a groupoid $\fg{\catname{C}}$ and a functor $Q: \catname{C} \to \fg{\catname{C}}$, such that precomposition with $Q$ establishes a bijection between functors $\fg{\catname{C}} \to \catname{D}$, and functors $\catname{C} \to \catname{D}$ that send morphisms of $\catname{C}$ to isomorphisms.
\end{definition}
A localisation $(\fg{\cat{C}}, Q)$ of a category $\cat{C}$ always exists, and is unique up to isomorphism by the universal property of Definition \ref{def:localisation}. The following construction is due to Gabriel and Zisman \cite{gz}. The objects of $\fg{\cat{C}}$ are the objects of $\cat{C}$, and the morphisms of $\fg{\cat{C}}$ are equivalence classes of zigzags, which are diagrams in \catname{C} that look something like \[ \xrightarrow{f_1} \ \xleftarrow{f_2} \ \xrightarrow{f_3} \ \xrightarrow{f_4} \  \xleftarrow{f_5}. \] Equivalence of zigzags is specified by the following relations:
\begin{itemize}
    \item  any instance of $ \xrightarrow{f} \xrightarrow{g}  $ may be replaced with $ \xrightarrow{gf} $, and $ \xleftarrow{g}  \xleftarrow{f} $ with $ \xleftarrow{gf}$, for all composable morphisms $f$ and $g$ in \cat{C},
    \item $ \xleftarrow{f}  \xrightarrow{f} $ and $ \xrightarrow{f} \xleftarrow{f} $ may be removed for all morphisms $f$ of $\catname{C}$,
    \item identity morphisms pointing left or right may be removed.
\end{itemize} Composition is given by concatenation. The required functor $Q: \cat{C} \to \fg{\cat{C}}$ is the identity on objects, and maps a morphism $f$ of $\catname{C}$ to the equivalence class of the zigzag $ \xrightarrow{f}$. The inverse of a morphism in $\fg{\cat{C}}$ is obtained by reversing the order and direction of the arrows. In particular, for each morphism $f$ in $\cat{C},$ the equivalence class of the zigzag $\xleftarrow{f}$ is a formal inverse to $f$. 
    \begin{example}\label{eg:freegrp} If $M$ is the free monoid on an alphabet $X$, regarded as a category with one object, then $\fg{M}$ is the free group on $X$.
    \end{example}
    \begin{example}
    If $\cat{C}$ is already a groupoid, then $\fg{\cat{C}} \cong \cat{C},$ since $(\catname{C}, \mathrm{id}_C)$ is a localisation for $\cat{C}.$ 
    \end{example}
    \begin{remark}
    The localisation of a category $\catname{C}$ with respect to a subcategory $\cat{W}$ is constructed similarly; only morphisms from $\cat{W}$ point to the left. As a result, the localisation may not be a groupoid.
    \end{remark}
\subsection{Calculus of fractions}
Example \ref{eg:freegrp} shows that a morphism in $\fg{\cat{C}}$ can be arbitrarily long. There are conditions where each morphism of $\fg{\cat{C}}$ can be reduced to the form $\xrightarrow{f} \xleftarrow{g}$, thought of as a fraction $g^{-1} f$. We state the definitions and results that we need and defer to \cite{gz, fracrevisited} for detailed proofs.
    \begin{definition}\label{def:leftcalcfrac}
    A category $\catname{C}$ has a \textit{left calculus of fractions} if 
    \begin{enumerate}
        \item[(F1)] For all morphisms $f,g$ with common source, there are morphisms $p$ and $q$ such that $pf = qg$.
        \item[(F2)] If $f,g,h$ are morphisms with $gf = hf,$ there is a morphism $k$ such that $kg = kh.$
    \end{enumerate}
    \end{definition}
Fix a category $\cat{C}$ having a left calculus of fractions. A \textit{cospan} in $\cat{C}$ is a diagram of the form $\xrightarrow{f} \xleftarrow{g}$.
\begin{lemma} Define a relation $\sim$ on the cospans in $\cat{C}$ given by $\xrightarrow{f} \xleftarrow{g} \  \sim \ \xrightarrow{h} \xleftarrow{k}$ if and only if there are morphisms $p$ and $q$ such that $pf = qh$ and $pg = qk.$ Then $\sim$ is an equivalence relation.
\end{lemma}
We are now ready to give a construction of $\widehat{C}$ in terms of equivalence classes of cospans (fractions), rather than arbitrary zigzags. We will not distinguish between a cospan and its equivalence class in our notation.
\begin{proposition}
There is a category $\frc{\cat{C}}$, the fraction groupoid of \cat{C}, in which
\begin{itemize}
    \item the objects of $\frc{\cat{C}}$ are the objects of $\cat{C}.$
    \item The morphisms of $\frc{\cat{C}}$ are equivalence classes of cospans in \cat{C}, with respect to the equivalence relation $\sim$.
    \item The source of a morphism $\xrightarrow{f} \xleftarrow{g} \ \in \frc{\cat{C}}$ is the source of $f$, and the target is the source of $g$.
    \item Given morphisms $\xrightarrow{f} \xleftarrow{g}$ and $\xrightarrow{h} \xleftarrow{k}$ of $\frc{\cat{C}}$ with the source of $g$ equal to the source of $h$, the composition $(\xrightarrow{h} \xleftarrow{k})(\xrightarrow{f} \xleftarrow{g}) = \xrightarrow{pf} \xleftarrow{qk}$, where $p$ and $q$ are any morphisms such that $pg = qk.$
    \item The identity at an object $c$ of $\frc{\cat{C}}$ is $\xrightarrow{\mathrm{id}_c} \xleftarrow{\mathrm{id}_c}$.
\end{itemize}
\end{proposition}
\begin{proposition}\label{prop:locfunctorfrac}
The category $\frc{\cat{C}}$ along with the identity-on-objects functor $\cat{C} \to \frc{\cat{C}}$, mapping a morphism $f$ of $\cat{C}$ to $\xrightarrow{f}\xleftarrow{\mathrm{id}}$, is a localisation of $\cat{C}$.
\end{proposition}
\begin{notation}
    For each morphism $f \in \cat{C}$, we denote the fraction $\xrightarrow{f} \xleftarrow{\mathrm{id}} \ \in \frc{\cat{C}}$ by $f$, and the fraction $\xrightarrow{\mathrm{id}} \xleftarrow{f} \ \in \frc{\cat{C}}$ by $f^{-1}$. To avoid ambiguity, we will specify whether we are viewing $f$ and $f^{-1}$ as elements of $\cat{C}$ or $\frc{\cat{C}}$. Using this notation, we have the equality 
        \[
        \xrightarrow{f} \xleftarrow{g} \ = g^{-1} f
        \] 
    for all morphisms $\xrightarrow{f} \xleftarrow{g} \ \in \frc{\cat{C}}$.
\end{notation}
\begin{remark}
    The canonical isomorphism $ \frc{\cat{C}} \cong \fg{C}$ between the fraction groupoid $\frc{\cat{C}}$ and the Gabriel-Zisman localisation $\fg{C}$ maps a fraction $g^{-1} f$ in $\frc{\cat{C}}$ to the zigzag $\xrightarrow{g} \xleftarrow{f}$ in $\fg{C}$. Thus, conditions (F1) and (F2) of Definition \ref{def:leftcalcfrac} provide conditions for the reduction of each zigzag in $\fg{C}$ to a fraction.
\end{remark}
\begin{remark}
Reversing the arrows in Definition \ref{def:leftcalcfrac} yields the notion of a calculus of right fractions. If $\catname{D}$ is a category with a calculus of right fractions, the opposite category $\catname{D}\op$ has a calculus of left fractions. The category $\frc{\cat{D}} := \frc{\catname{D}\op}\op$ is then a localisation of $\cat{D}$. The morphisms of $\frc{\cat{D}}$ are equivalence classes of a cospans in $\cat{D}\op,$ i.e. spans in $\cat{D}$. The equivalence class of a span $\xleftarrow{f} \xrightarrow{g}$ can be thought of as a right fraction $gf^{-1}$. 
\end{remark}
\begin{example}
Suppose that $M$ is a commutative and cancellative monoid. Then $M$ has a calculus of left fractions, and the group $\frc{M}$ is the Grothendieck group of $M$. In particular, when $M = (\NN \cup \{0\}, +)$, we have $\frc{M} \cong (\ZZ, +)$, and if $M = (\ZZ \setminus \{0\}, \times)$, then $\frc{M} \cong (\QQ \setminus \{0\}, \times)$.
\end{example}
\begin{example}\label{eg:nonfaithfullocfunctor}
Consider the monoid $M = (\mathbb{N} \cup \{0\}, \times).$ Then $M$ has a calculus of left fractions, though the group $\frc{M}$ is trivial. Given a fraction $n^{-1} m \  \in \frc{M}$,
\begin{align*}
    n^{-1} m &= 0 0^{-1} n^{-1} m 0 0^{-1} \\ &= 0 \ (n0)^{-1} \ (m0) 0^{-1} \\ &= 00^{-1} 00^{-1},
\end{align*} which is the identity element of $\frc{M}$. In particular, the localisation functor $M \to \frc{M}$ is not faithful. 
\end{example}
\begin{lemma}\label{lem:faithfullocfunctor}
The localisation functor $Q: \cat{C} \to \frc{\cat{C}}$ is faithful as soon as \cat{C} is left cancellative, i.e. $fg = fh$ implies that $g = h$ for all compatible morphisms $f,g,h$ in \cat{C}.
\end{lemma}
\subsection{Thompson groupoids}\label{subsec:thompsongroupoids}
We now introduce the categories $\cat{F}$ of binary forests and \cat{SF} of symmetric forests, and see how their localisations contain Richard Thompson's groups $F$ and $V$.

The objects of the category \cat{F} are the natural numbers, and the morphisms are rooted, ordered, and planar binary forests. The source and target of a forest are its number of roots and leaves respectively. For example, the forest
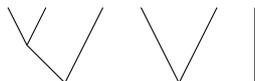
\[
\begin{tikzpicture}
	\begin{pgfonlayer}{nodelayer}
		\node [style=none] (0) at (0, 0) {};
		\node [style=none] (1) at (-1, 1) {};
		\node [style=none] (2) at (-1.5, 2) {};
		\node [style=none] (3) at (-0.5, 2) {};
		\node [style=none] (4) at (1, 2) {};
		\node [style=none] (5) at (2, 2) {};
		\node [style=none] (6) at (4, 2) {};
		\node [style=none] (7) at (3, 0) {};
		\node [style=none] (8) at (5, 0) {};
		\node [style=none] (9) at (5, 2) {};
	\end{pgfonlayer}
	\begin{pgfonlayer}{edgelayer}
		\draw (0.center) to (1.center);
		\draw (1.center) to (2.center);
		\draw (1.center) to (3.center);
		\draw (0.center) to (4.center);
		\draw (7.center) to (5.center);
		\draw (7.center) to (6.center);
		\draw (8.center) to (9.center);
	\end{pgfonlayer}
\end{tikzpicture}
\] is a morphism $3 \to 6$ in $\cat{F}.$ We compose forests by stacking them vertically, so that $g f$ is the forest obtained by attaching the $i^{\text{th}}$ leaf of $f$ to the $i^{\text{th}}$ root of $g.$ E.g. if
\[
\begin{tikzpicture}
	\begin{pgfonlayer}{nodelayer}
		\node [style=none] (0) at (0, 0) {};
		\node [style=none] (1) at (-1, 1) {};
		\node [style=none] (2) at (-1.5, 2) {};
		\node [style=none] (3) at (-0.5, 2) {};
		\node [style=none] (4) at (1, 2) {};
		\node [style=none] (5) at (-2.5, 1) {$t =$};
	\end{pgfonlayer}
	\begin{pgfonlayer}{edgelayer}
		\draw (0.center) to (1.center);
		\draw (1.center) to (2.center);
		\draw (1.center) to (3.center);
		\draw (0.center) to (4.center);
	\end{pgfonlayer}
\end{tikzpicture}
\]
and
\[
\begin{tikzpicture}
	\begin{pgfonlayer}{nodelayer}
		\node [style=none] (0) at (0, 0) {};
		\node [style=none] (1) at (-1, 1) {};
		\node [style=none] (2) at (-1.5, 2) {};
		\node [style=none] (3) at (-0.5, 2) {};
		\node [style=none] (4) at (1, 2) {};
		\node [style=none] (5) at (2, 2) {};
		\node [style=none] (6) at (4, 2) {};
		\node [style=none] (7) at (3, 0) {};
		\node [style=none] (8) at (5, 0) {};
		\node [style=none] (9) at (5, 2) {};
		\node [style=none] (10) at (-2.5, 1) {$f =$};
		\node [style=none] (11) at (6, 0.75) {,};
	\end{pgfonlayer}
	\begin{pgfonlayer}{edgelayer}
		\draw (0.center) to (1.center);
		\draw (1.center) to (2.center);
		\draw (1.center) to (3.center);
		\draw (0.center) to (4.center);
		\draw (7.center) to (5.center);
		\draw (7.center) to (6.center);
		\draw (8.center) to (9.center);
	\end{pgfonlayer}
\end{tikzpicture}
\] then
\[
\begin{tikzpicture}
	\begin{pgfonlayer}{nodelayer}
		\node [style=none] (0) at (-1.5, 2) {};
		\node [style=none] (1) at (-2, 2.5) {};
		\node [style=none] (2) at (-2.25, 3) {};
		\node [style=none] (3) at (-1.75, 3) {};
		\node [style=none] (4) at (-1.25, 3) {};
		\node [style=none] (5) at (-0.75, 3) {};
		\node [style=none] (6) at (-0.25, 3) {};
		\node [style=none] (7) at (-0.5, 2) {};
		\node [style=none] (8) at (1, 2) {};
		\node [style=none] (9) at (1, 3) {};
		\node [style=none] (11) at (0, 0) {};
		\node [style=none] (12) at (-1, 1) {};
		\node [style=none] (13) at (-1.5, 2) {};
		\node [style=none] (14) at (-0.5, 2) {};
		\node [style=none] (15) at (1, 2) {};
		\node [style=none] (16) at (2.25, 1.5) {$=$};
		\node [style=none] (17) at (5, 0) {};
		\node [style=none] (18) at (3, 3) {};
		\node [style=none] (19) at (4.425, 0.875) {};
		\node [style=none] (20) at (5.5, 3) {};
		\node [style=none] (21) at (3.825, 1.775) {};
		\node [style=none] (22) at (4.25, 3) {};
		\node [style=none] (23) at (3.625, 3) {};
		\node [style=none] (24) at (3.425, 2.375) {};
		\node [style=none] (25) at (6.5, 3) {};
		\node [style=none] (26) at (-4.5, 1.5) {$f t =$};
		\node [style=none] (27) at (7, 1.5) {.};
		\node [style=none] (28) at (4.75, 3) {};
		\node [style=none] (29) at (5.075, 2.15) {};
	\end{pgfonlayer}
	\begin{pgfonlayer}{edgelayer}
		\draw (0.center) to (1.center);
		\draw (1.center) to (2.center);
		\draw (1.center) to (3.center);
		\draw (0.center) to (4.center);
		\draw (7.center) to (5.center);
		\draw (7.center) to (6.center);
		\draw (8.center) to (9.center);
		\draw (11.center) to (12.center);
		\draw (12.center) to (13.center);
		\draw (12.center) to (14.center);
		\draw (11.center) to (15.center);
		\draw (17.center) to (18.center);
		\draw (19.center) to (20.center);
		\draw (21.center) to (22.center);
		\draw (23.center) to (24.center);
		\draw (17.center) to (25.center);
		\draw (28.center) to (29.center);
	\end{pgfonlayer}
\end{tikzpicture}
\]
A binary tree is a binary forest with one root. Let $\mathfrak{T}$ be the set of all binary trees, and for each $t \in \mathfrak{T}$, write $\Leaf(t)$ for the set of leaves of $t$ and define $\ell(t) := |\Leaf(t)|$ (the number of leaves of $t$). We identify each leaf $l \in \Leaf(t)$ with the geodesic path from the root of $t$ to $l$, written as a word in $0$ and $1$, from left to right. A $0$ means a left turn, and a $1$ means a right turn. For example, if 
\[
\begin{tikzpicture}
	\begin{pgfonlayer}{nodelayer}
		\node [style=none] (0) at (0, 0) {};
		\node [style=none] (1) at (-2, 3) {};
		\node [style=none] (2) at (2, 3) {};
		\node [style=none] (3) at (-1.1, 1.675) {};
		\node [style=none] (4) at (0, 3) {};
		\node [style=none] (5) at (-0.55, 2.325) {};
		\node [style=none] (6) at (-1, 3) {};
		\node [style=none] (7) at (1.325, 2) {};
		\node [style=none] (8) at (0.75, 3) {};
		\node [style=none] (9) at (-2.75, 1.5) {$t =$};
		\node [style=none] (10) at (2.25, 1.25) {,};
	\end{pgfonlayer}
	\begin{pgfonlayer}{edgelayer}
		\draw (0.center) to (1.center);
		\draw (0.center) to (2.center);
		\draw (3.center) to (4.center);
		\draw (5.center) to (6.center);
		\draw (7.center) to (8.center);
	\end{pgfonlayer}
\end{tikzpicture}
\] then $\Leaf(t) = \{00, 010, 011, 10, 11\}$. We order $\Leaf(t)$ using the lexographical ordering, letting $l_t^i$ the $i^{\text{th}}$ leaf of $t.$ Sometimes we will identify $\Leaf(t)$ with the set $\{1,\dots, \ell(t)\}$ equipped with the usual linear order. Letting $\{0,1\}^*$ denote the monoid of finite words in $0$ and $1$, with $\epsilon$ denoting the empty word, the set of leaves of the tree $|$ with one root and one leaf is $\{\epsilon\}$ by convention.

Every binary forest $f$ is essentially a list $(f_1,\dots, f_n)$ of binary trees, where $n$ is the number of roots of $f$. We call $f_i$ the $i^{\text{th}}$ tree of $f$ for all $1 \leq i \leq n.$ The support of a forest is the set of indices $i$ for which $f_i$ is not the trivial tree $|$.

Introducing permutations to \cat{F} yields the category $\cat{SF}$. The objects of $\cat{SF}$ are again the natural numbers, and the morphisms are generated by the morphisms of $\catname{F}$ along with all permutations $\sigma \in S_n$ with $n \geq 1$, subject to some relations. A permutation $\sigma \in S_n$ is regarded as a morphism $n \to n,$ and is represented diagramatically as $n$ arrows joining two copies of $\{1,\dots,n\}$. For example, the diagram of the permutation $(123) \in S_3$ is
\[
\begin{tikzpicture}
	\begin{pgfonlayer}{nodelayer}
		\node [style=none, label={below:$1$}] (0) at (0, 0) {};
		\node [style=none, label={above:$2$}] (1) at (2, 2) {};
		\node [style=none, label={below:$2$}] (2) at (2, 0) {};
		\node [style=none, label={above:$3$}] (3) at (4, 2) {};
		\node [style=none, label={below:$3$}] (4) at (4, 0) {};
		\node [style=none, label={above:$1$}] (5) at (0, 2) {};
		\node [style=none] (6) at (4.5, 1) {.};
	\end{pgfonlayer}
	\begin{pgfonlayer}{edgelayer}
		\draw [style=arrowhead] (0.center) to (1.center);
		\draw [style=arrowhead] (2.center) to (3.center);
		\draw [style=arrowhead] (4.center) to (5.center);
	\end{pgfonlayer}
\end{tikzpicture}
\] To keep our diagrams neat, we omit arrow heads and labels, instead writing
\[
\begin{tikzpicture}
	\begin{pgfonlayer}{nodelayer}
		\node [style=none] (0) at (0, 0) {};
		\node [style=none] (1) at (2, 2) { };
		\node [style=none] (2) at (2, 0) {};
		\node [style=none] (3) at (4, 2) {};
		\node [style=none] (4) at (4, 0) {};
		\node [style=none] (5) at (0, 2) {};
		\node [style=none] (6) at (-1.5, 1) {$(123) =$};
		\node [style=none] (7) at (4.75, 1) {.};
	\end{pgfonlayer}
	\begin{pgfonlayer}{edgelayer}
		\draw (0.center) to (1.center);
		\draw (2.center) to (3.center);
		\draw (4.center) to (5.center);
	\end{pgfonlayer}
\end{tikzpicture}
\] 

\pagebreak
Composition in $\catname{SF}$ is again vertical stacking and isotopy, subject to the relations:
\begin{itemize}
    \item permutations are composed in the usual way, and the identity of $S_n$ is the identity at the object $n \in \catname{SF}$,
    \item if $f: m \to n$ is a forest and $\sigma \in S_m,$ then $f \sigma = S(f,\sigma) \sigma(f),$ where $\sigma(f)$ is the binary forest satisfying $\sigma(f)_i := f_{\sigma(i)}$, and $S(f,\sigma) \in S_n$ is the permutation on $n$ letters obtained by replacing the $i^{\text{th}}$ line segment of $\sigma$ with $\ell(\sigma(f)_i)$ parallel line segments.
\end{itemize} 
The second kind of relation above specifies that a tree can slide down a line segment attached to its root, e.g. if
\[
\begin{tikzpicture}
	\begin{pgfonlayer}{nodelayer}
		\node [style=none] (0) at (0, 0) {};
		\node [style=none] (1) at (-1, 1) {};
		\node [style=none] (2) at (-1.5, 2) {};
		\node [style=none] (3) at (-0.5, 2) {};
		\node [style=none] (4) at (1, 2) {};
		\node [style=none] (5) at (2, 2) {};
		\node [style=none] (6) at (4, 2) {};
		\node [style=none] (7) at (3, 0) {};
		\node [style=none] (8) at (5, 0) {};
		\node [style=none] (9) at (5, 2) {};
		\node [style=none] (10) at (-2.25, 1) {$f=$};
		\node [style=none] (11) at (5.75, 0.75) {,};
	\end{pgfonlayer}
	\begin{pgfonlayer}{edgelayer}
		\draw (0.center) to (1.center);
		\draw (1.center) to (2.center);
		\draw (1.center) to (3.center);
		\draw (0.center) to (4.center);
		\draw (7.center) to (5.center);
		\draw (7.center) to (6.center);
		\draw (8.center) to (9.center);
	\end{pgfonlayer}
\end{tikzpicture}
\]
and $\sigma = (123) \in S_3,$ then
\[
\begin{tikzpicture}
	\begin{pgfonlayer}{nodelayer}
		\node [style=none] (0) at (-0.25, 0) {};
		\node [style=none] (1) at (-1.25, 1) {};
		\node [style=none] (2) at (-1.75, 2) {};
		\node [style=none] (3) at (-0.75, 2) {};
		\node [style=none] (4) at (0.75, 2) {};
		\node [style=none] (5) at (1.75, 2) {};
		\node [style=none] (6) at (3.75, 2) {};
		\node [style=none] (7) at (2.75, 0) {};
		\node [style=none] (8) at (4.75, 0) {};
		\node [style=none] (9) at (4.75, 2) {};
		\node [style=none] (10) at (-3.25, 0) {$f \sigma =$};
		\node [style=none] (11) at (-0.25, -2) {};
		\node [style=none] (12) at (2.75, -2) {};
		\node [style=none] (13) at (4.75, -2) {};
		\node [style=none] (14) at (6.5, 0) {$=$};
		\node [style=none] (20) at (8, 0) {};
		\node [style=none] (21) at (10, 0) {};
		\node [style=none] (22) at (9, -2) {};
		\node [style=none] (23) at (11, -2) {};
		\node [style=none] (24) at (11, 0) {};
		\node [style=none] (25) at (13.5, -2) {};
		\node [style=none] (26) at (12.5, -1) {};
		\node [style=none] (27) at (12, 0) {};
		\node [style=none] (28) at (13, 0) {};
		\node [style=none] (29) at (14.5, 0) {};
		\node [style=none] (30) at (18.5, 0) {$= S(f,\sigma) \sigma(f)$.};
		\node [style=none] (31) at (8, 2) {};
		\node [style=none] (32) at (9, 2) {};
		\node [style=none] (33) at (10.25, 2) {};
		\node [style=none] (34) at (11.5, 2) {};
		\node [style=none] (35) at (13.25, 2) {};
		\node [style=none] (36) at (14.5, 2) {};
	\end{pgfonlayer}
	\begin{pgfonlayer}{edgelayer}
		\draw (0.center) to (1.center);
		\draw (1.center) to (2.center);
		\draw (1.center) to (3.center);
		\draw (0.center) to (4.center);
		\draw (7.center) to (5.center);
		\draw (7.center) to (6.center);
		\draw (8.center) to (9.center);
		\draw (11.center) to (7.center);
		\draw (12.center) to (8.center);
		\draw (13.center) to (0.center);
		\draw (22.center) to (20.center);
		\draw (22.center) to (21.center);
		\draw (23.center) to (24.center);
		\draw (25.center) to (26.center);
		\draw (26.center) to (27.center);
		\draw (26.center) to (28.center);
		\draw (25.center) to (29.center);
		\draw (27.center) to (31.center);
		\draw (28.center) to (32.center);
		\draw (29.center) to (33.center);
		\draw (20.center) to (34.center);
		\draw (21.center) to (35.center);
		\draw (24.center) to (36.center);
	\end{pgfonlayer}
\end{tikzpicture}
\]
These relations for the category \cat{SF} are those of a Brin-Zappa-Sz\'{e}p product of the category of binary forests \cat{F} and the permutation groupoid $\sqcup_{n} S_n$ (see \cite{brinstrand, wz, brinzappaszep} for more on this). Note that there is an identity-on-objects embedding $\cat{F} \hookrightarrow \cat{SF}$ mapping a forest $f \in \cat{F}$ to the symmetric forest $\sigma f$, where $\sigma$ is an identity permutation. This embedding identifies \cat{F} with the subcategory of \cat{SF} given by morphisms whose diagrams are planar.

We will now examine the fraction groupoids of \cat{F} and \cat{SF}. For each 
$n \geq 0$, let $t_n$ be the regular tree with $2^n$ leaves. Given forests $f,g \in \cat{F},$ we may find forests $p,q$ such that $pf = qg$, by completing each tree of $f$ and $g$ into a regular tree $t_n$. Thus, the category $\cat{F}$ satisfies condition $(F1)$ of Definition \ref{def:leftcalcfrac}, and a similar argument shows that the category $\cat{SF}$ also satisfies condition $(F1)$. 

It is visually apparent that the category \cat{F} also satisfies the condition $(F2)$, and it quickly follows that the category $\cat{SF}$ does as well. Thus, \cat{F} and \cat{SF} both admit a calculus of left fractions, yielding the fraction groupoids $\frc{\cat{F}}$ and $\frc{\cat{SF}}$. By Lemma \ref{lem:faithfullocfunctor}, the localisation functors $\cat{F} \to \frc{\cat{F}}$ and $\cat{SF} \to \frc{\cat{SF}}$ are faithful. These functors fit into a commutative diagram of embeddings
\[ 
\begin{tikzcd}
\cat{F} \arrow[rr, hook] \arrow[d, hook] &  & \cat{SF} \arrow[d, hook] \\
\frc{\cat{F}} \arrow[rr, hook]           &  & \frc{\cat{SF}}          
\end{tikzcd}.
\]
The embedding $\frc{\cat{F}} \hookrightarrow \frc{\cat{SF}}$ is the identity on objects, and maps a fraction $g^{-1} f \in \frc{\cat{F}}$ to the fraction $(\sigma g)^{-1} (\sigma f) \in \frc{\cat{SF}}$, where $\sigma$ is an identity permutation.

The morphisms of $\frc{\cat{F}}$ and $\frc{\cat{SF}}$ admit afford a powerful diagrammatic calculus. Each fraction $(\sigma f)^{-1} (\tau g)$ in $\frc{\cat{SF}}$ may be interpreted diagramatically as the diagram of $\sigma f$ placed upside down on top of the diagram of $\tau g$. For example, if
\[
\begin{tikzpicture}
	\begin{pgfonlayer}{nodelayer}
		\node [style=none] (0) at (0, -0.25) {};
		\node [style=none] (1) at (-1, 1) {};
		\node [style=none] (2) at (1.5, 2) {};
		\node [style=none] (3) at (-1.75, 2) {};
		\node [style=none] (4) at (-0.25, 2) {};
		\node [style=none] (5) at (-0.25, 3.75) {};
		\node [style=none] (6) at (-1.75, 3.75) {};
		\node [style=none] (7) at (1.5, 3.75) {};
		\node [style=none] (8) at (-3.75, 1.75) {$\tau g = $};
	\end{pgfonlayer}
	\begin{pgfonlayer}{edgelayer}
		\draw (0.center) to (1.center);
		\draw (0.center) to (2.center);
		\draw (1.center) to (4.center);
		\draw (1.center) to (3.center);
		\draw (3.center) to (5.center);
		\draw (4.center) to (6.center);
		\draw (2.center) to (7.center);
	\end{pgfonlayer}
\end{tikzpicture}
\]
and 
\[
\begin{tikzpicture}
	\begin{pgfonlayer}{nodelayer}
		\node [style=none] (1) at (0.5, 0) {};
		\node [style=none] (2) at (-2.25, 0) {};
		\node [style=none] (3) at (1.75, 2) {};
		\node [style=none] (4) at (-0.75, 2) {};
		\node [style=none] (7) at (-2.25, 2) {};
		\node [style=none] (8) at (-3.75, 2) {$\sigma f =$};
		\node [style=none] (9) at (-2.25, 4) {};
		\node [style=none] (10) at (-0.75, 4) {};
		\node [style=none] (11) at (1.75, 4) {};
		\node [style=none] (12) at (2.75, 1.75) {,};
	\end{pgfonlayer}
	\begin{pgfonlayer}{edgelayer}
		\draw (1.center) to (4.center);
		\draw (1.center) to (3.center);
		\draw (2.center) to (7.center);
		\draw (4.center) to (11.center);
		\draw (3.center) to (10.center);
		\draw (9.center) to (7.center);
	\end{pgfonlayer}
\end{tikzpicture}
\]
then $(\sigma f)^{-1} (\tau g)$ becomes
\[
\scalebox{1.3}{\begin{tikzpicture}
	\begin{pgfonlayer}{nodelayer}
		\node [style=none] (1) at (0.5, 5) {};
		\node [style=none] (2) at (-1.5, 5) {};
		\node [style=none] (3) at (1, 4) {};
		\node [style=none] (4) at (0, 4) {};
		\node [style=none] (7) at (-1.5, 4) {};
		\node [style=none] (9) at (-1.5, 3) {};
		\node [style=none] (10) at (-0.25, 3) {};
		\node [style=none] (11) at (1, 3) {};
		\node [style=none] (13) at (0, 0) {};
		\node [style=none] (14) at (-1, 1) {};
		\node [style=none] (15) at (1, 2) {};
		\node [style=none] (16) at (-1.5, 2) {};
		\node [style=none] (17) at (-0.5, 2) {};
		\node [style=none] (18) at (-0.25, 3) {};
		\node [style=none] (19) at (-1.5, 3) {};
		\node [style=none] (20) at (1, 3) {};
		\node [style=none] (21) at (3, 2.5) {$=$};
		\node [style=none] (22) at (6.5, 4.25) {};
		\node [style=none] (23) at (4.5, 4.25) {};
		\node [style=none] (24) at (7, 3.25) {};
		\node [style=none] (25) at (6, 3.25) {};
		\node [style=none] (26) at (4.5, 3.25) {};
		\node [style=none] (30) at (6, 0.25) {};
		\node [style=none] (31) at (5, 1.25) {};
		\node [style=none] (32) at (7, 2.25) {};
		\node [style=none] (33) at (4.5, 2.25) {};
		\node [style=none] (34) at (5.5, 2.25) {};
		\node [style=none] (35) at (8, 2.5) {$.$};
	\end{pgfonlayer}
	\begin{pgfonlayer}{edgelayer}
		\draw (1.center) to (4.center);
		\draw (1.center) to (3.center);
		\draw (2.center) to (7.center);
		\draw (4.center) to (11.center);
		\draw (3.center) to (10.center);
		\draw (9.center) to (7.center);
		\draw (13.center) to (14.center);
		\draw (13.center) to (15.center);
		\draw (14.center) to (17.center);
		\draw (14.center) to (16.center);
		\draw (16.center) to (18.center);
		\draw (17.center) to (19.center);
		\draw (15.center) to (20.center);
		\draw (22.center) to (25.center);
		\draw (22.center) to (24.center);
		\draw (23.center) to (26.center);
		\draw (30.center) to (31.center);
		\draw (30.center) to (32.center);
		\draw (31.center) to (34.center);
		\draw (31.center) to (33.center);
		\draw (33.center) to (24.center);
		\draw (34.center) to (26.center);
		\draw (32.center) to (25.center);
	\end{pgfonlayer}
\end{tikzpicture}}
\]
Composition of morphisms in $\frc{\cat{SF}}$ represented this way can be performed by vertical stacking and isotopy, the relations of \cat{SF}, and the relations
\[
\begin{tikzpicture}
	\begin{pgfonlayer}{nodelayer}
		\node [style=none] (0) at (-2, 2) {};
		\node [style=none] (1) at (0, 0) {};
		\node [style=none] (2) at (2, 2) {};
		\node [style=none] (3) at (0, -1) {};
		\node [style=none] (4) at (-2, -3) {};
		\node [style=none] (5) at (2, -3) {};
		\node [style=none] (6) at (4, -0.5) {$=$};
		\node [style=none] (7) at (7, -3) {};
		\node [style=none] (8) at (7, 2) {};
		\node [style=none] (9) at (9, -3) {};
		\node [style=none] (10) at (9, 2) {};
	\end{pgfonlayer}
	\begin{pgfonlayer}{edgelayer}
		\draw (4.center) to (3.center);
		\draw (5.center) to (3.center);
		\draw (3.center) to (1.center);
		\draw (1.center) to (0.center);
		\draw (1.center) to (2.center);
		\draw (7.center) to (8.center);
		\draw (9.center) to (10.center);
	\end{pgfonlayer}
\end{tikzpicture}
\tag{$*$}\] and 
\[ 
\begin{tikzpicture}
	\begin{pgfonlayer}{nodelayer}
		\node [style=none] (0) at (-2, 0) {};
		\node [style=none] (1) at (0, 2) {};
		\node [style=none] (2) at (2, 0) {};
		\node [style=none] (3) at (0, -2) {};
		\node [style=none] (4) at (-2, 0) {};
		\node [style=none] (5) at (2, 0) {};
		\node [style=none] (6) at (3.75, 0) {$=$};
		\node [style=none] (7) at (5.75, -2) {};
		\node [style=none] (8) at (5.75, 2) {};
		\node [style=none] (9) at (7, -0.25) {.};
	\end{pgfonlayer}
	\begin{pgfonlayer}{edgelayer}
		\draw (4.center) to (3.center);
		\draw (5.center) to (3.center);
		\draw (1.center) to (0.center);
		\draw (1.center) to (2.center);
		\draw (7.center) to (8.center);
	\end{pgfonlayer}
\end{tikzpicture} \tag{$**$}
\] The relations $(*)$ and $(**)$ assert that $\left(\caret\right)^{-1} = \fcaret$. For example, if 
\[\begin{tikzpicture}
	\begin{pgfonlayer}{nodelayer}
		\node [style=none] (0) at (0, 0) {};
		\node [style=none] (1) at (-2, 2) {};
		\node [style=none] (2) at (-3, 4) {};
		\node [style=none] (3) at (-1, 4) {};
		\node [style=none] (4) at (2, 4) {};
		\node [style=none] (5) at (-3, 6.5) {};
		\node [style=none] (6) at (-0.25, 6.5) {};
		\node [style=none] (7) at (2, 6.5) {};
		\node [style=none] (11) at (-5.25, 3) {$\sigma t =$};
		\node [style=none] (12) at (3, 2.75) {,};
	\end{pgfonlayer}
	\begin{pgfonlayer}{edgelayer}
		\draw (0.center) to (1.center);
		\draw (1.center) to (2.center);
		\draw (1.center) to (3.center);
		\draw (0.center) to (4.center);
		\draw (2.center) to (6.center);
		\draw (3.center) to (7.center);
		\draw (4.center) to (5.center);
	\end{pgfonlayer}
\end{tikzpicture}\]
then 
\[\begin{tikzpicture}
	\begin{pgfonlayer}{nodelayer}
		\node [style=none] (0) at (0, 6.5) {};
		\node [style=none] (1) at (-2, 4.5) {};
		\node [style=none] (2) at (-3, 2.5) {};
		\node [style=none] (3) at (-1, 2.5) {};
		\node [style=none] (4) at (2, 2.5) {};
		\node [style=none] (5) at (-3, 0) {};
		\node [style=none] (6) at (-0.25, 0) {};
		\node [style=none] (7) at (2, 0) {};
		\node [style=none] (11) at (-5.25, 2.5) {$(\sigma  t)^{-1} =$};
		\node [style=none] (12) at (3, 2.25) {,};
	\end{pgfonlayer}
	\begin{pgfonlayer}{edgelayer}
		\draw (0.center) to (1.center);
		\draw (1.center) to (2.center);
		\draw (1.center) to (3.center);
		\draw (0.center) to (4.center);
		\draw (2.center) to (6.center);
		\draw (3.center) to (7.center);
		\draw (4.center) to (5.center);
	\end{pgfonlayer}
\end{tikzpicture}\]
as evidenced by the computations
\[\begin{tikzpicture}
	\begin{pgfonlayer}{nodelayer}
		\node [style=none] (0) at (0, 0) {};
		\node [style=none] (1) at (-2, 2) {};
		\node [style=none] (2) at (-3, 4) {};
		\node [style=none] (3) at (-1, 4) {};
		\node [style=none] (4) at (2, 4) {};
		\node [style=none] (5) at (-3, 6) {};
		\node [style=none] (6) at (-0.25, 6) {};
		\node [style=none] (7) at (2, 6) {};
		\node [style=none] (8) at (-6.75, 6) {$(\sigma f)^{-1} (\sigma f) = $};
		\node [style=none] (9) at (0, 12) {};
		\node [style=none] (10) at (-2, 10) {};
		\node [style=none] (11) at (-3, 8) {};
		\node [style=none] (12) at (-1, 8) {};
		\node [style=none] (13) at (2, 8) {};
		\node [style=none] (14) at (-3, 6) {};
		\node [style=none] (15) at (-0.25, 6) {};
		\node [style=none] (16) at (2, 6) {};
		\node [style=none] (17) at (4, 6) {$=$};
		\node [style=none] (18) at (8.25, 10) {};
		\node [style=none] (19) at (6.25, 8) {};
		\node [style=none] (20) at (5.25, 6) {};
		\node [style=none] (21) at (7.25, 6) {};
		\node [style=none] (22) at (10.25, 6) {};
		\node [style=none] (23) at (8.25, 2) {};
		\node [style=none] (24) at (6.25, 4) {};
		\node [style=none] (25) at (5.25, 6) {};
		\node [style=none] (26) at (7.25, 6) {};
		\node [style=none] (27) at (10.25, 6) {};
		\node [style=none] (28) at (11.5, 6) {$=$};
		\node [style=none] (29) at (15.5, 10) {};
		\node [style=none] (30) at (13.5, 8) {};
		\node [style=none] (31) at (13.5, 6) {};
		\node [style=none] (33) at (17.5, 6) {};
		\node [style=none] (34) at (15.5, 2) {};
		\node [style=none] (35) at (13.5, 4) {};
		\node [style=none] (36) at (13.5, 6) {};
		\node [style=none] (38) at (17.5, 6) {};
		\node [style=none] (39) at (19, 6) {$=$};
		\node [style=none] (40) at (21, 8) {};
		\node [style=none] (41) at (21, 4) {};
	\end{pgfonlayer}
	\begin{pgfonlayer}{edgelayer}
		\draw (0.center) to (1.center);
		\draw (1.center) to (2.center);
		\draw (1.center) to (3.center);
		\draw (0.center) to (4.center);
		\draw (2.center) to (6.center);
		\draw (3.center) to (7.center);
		\draw (4.center) to (5.center);
		\draw (9.center) to (10.center);
		\draw (10.center) to (11.center);
		\draw (10.center) to (12.center);
		\draw (9.center) to (13.center);
		\draw (11.center) to (15.center);
		\draw (12.center) to (16.center);
		\draw (13.center) to (14.center);
		\draw (18.center) to (19.center);
		\draw (19.center) to (20.center);
		\draw (19.center) to (21.center);
		\draw (18.center) to (22.center);
		\draw (23.center) to (24.center);
		\draw (24.center) to (25.center);
		\draw (24.center) to (26.center);
		\draw (23.center) to (27.center);
		\draw (29.center) to (30.center);
		\draw (30.center) to (31.center);
		\draw (29.center) to (33.center);
		\draw (34.center) to (35.center);
		\draw (35.center) to (36.center);
		\draw (34.center) to (38.center);
		\draw (40.center) to (41.center);
	\end{pgfonlayer}
\end{tikzpicture}\] and \[\begin{tikzpicture}
	\begin{pgfonlayer}{nodelayer}
		\node [style=none] (0) at (0, 0) {};
		\node [style=none] (1) at (-2, 2) {};
		\node [style=none] (2) at (-3, 4) {};
		\node [style=none] (3) at (-1, 4) {};
		\node [style=none] (4) at (2, 4) {};
		\node [style=none] (5) at (-3, 6) {};
		\node [style=none] (6) at (-0.25, 6) {};
		\node [style=none] (7) at (2, 6) {};
		\node [style=none] (8) at (-3, 6) {};
		\node [style=none] (9) at (-0.25, 6) {};
		\node [style=none] (10) at (2, 6) {};
		\node [style=none] (11) at (-3, -6) {};
		\node [style=none] (12) at (-0.25, -6) {};
		\node [style=none] (13) at (2, -6) {};
		\node [style=none] (14) at (0, 0) {};
		\node [style=none] (15) at (-2, -2) {};
		\node [style=none] (16) at (-3, -4) {};
		\node [style=none] (17) at (-1, -4) {};
		\node [style=none] (18) at (2, -4) {};
		\node [style=none] (19) at (-3, -6) {};
		\node [style=none] (20) at (-0.25, -6) {};
		\node [style=none] (21) at (2, -6) {};
		\node [style=none] (22) at (-7, 0) {$\sigma  f  (\sigma f)^{-1} = $};
		\node [style=none] (23) at (4, 0) {$=$};
		\node [style=none] (24) at (6, 1.25) {};
		\node [style=none] (25) at (5, 3.25) {};
		\node [style=none] (26) at (7, 3.25) {};
		\node [style=none] (27) at (10, 3.25) {};
		\node [style=none] (28) at (10, 1.25) {};
		\node [style=none] (29) at (6, 0) {};
		\node [style=none] (30) at (10, 0) {};
		\node [style=none] (31) at (6, -1.25) {};
		\node [style=none] (32) at (10, -1.25) {};
		\node [style=none] (33) at (5, -3.25) {};
		\node [style=none] (34) at (7, -3.25) {};
		\node [style=none] (35) at (10, -3.25) {};
		\node [style=none] (36) at (5, -5.25) {};
		\node [style=none] (37) at (7.75, -5.25) {};
		\node [style=none] (38) at (10, -5.25) {};
		\node [style=none] (39) at (5, -3.25) {};
		\node [style=none] (40) at (7, -3.25) {};
		\node [style=none] (41) at (10, -3.25) {};
		\node [style=none] (42) at (5, -5.25) {};
		\node [style=none] (43) at (7.75, -5.25) {};
		\node [style=none] (44) at (10, -5.25) {};
		\node [style=none] (45) at (5, 3.25) {};
		\node [style=none] (47) at (10, 3.25) {};
		\node [style=none] (51) at (5, 3.25) {};
		\node [style=none] (53) at (10, 3.25) {};
		\node [style=none] (54) at (12, 0) {$=$};
		\node [style=none] (55) at (13.75, -1.5) {};
		\node [style=none] (56) at (13.75, 1.5) {};
		\node [style=none] (57) at (15, -1.5) {};
		\node [style=none] (58) at (15, 1.5) {};
		\node [style=none] (59) at (16.25, -1.5) {};
		\node [style=none] (60) at (16.25, 1.5) {};
		\node [style=none] (61) at (5, 3.25) {};
		\node [style=none] (62) at (7, 3.25) {};
		\node [style=none] (63) at (10, 3.25) {};
		\node [style=none] (64) at (5, 5.25) {};
		\node [style=none] (65) at (7.75, 5.25) {};
		\node [style=none] (66) at (10, 5.25) {};
		\node [style=none] (67) at (5, 5.25) {};
		\node [style=none] (68) at (7.75, 5.25) {};
		\node [style=none] (69) at (10, 5.25) {};
		\node [style=none] (70) at (17, 0) {.};
	\end{pgfonlayer}
	\begin{pgfonlayer}{edgelayer}
		\draw (0.center) to (1.center);
		\draw (1.center) to (2.center);
		\draw (1.center) to (3.center);
		\draw (0.center) to (4.center);
		\draw (2.center) to (6.center);
		\draw (3.center) to (7.center);
		\draw (4.center) to (5.center);
		\draw (14.center) to (15.center);
		\draw (15.center) to (16.center);
		\draw (15.center) to (17.center);
		\draw (14.center) to (18.center);
		\draw (16.center) to (20.center);
		\draw (17.center) to (21.center);
		\draw (18.center) to (19.center);
		\draw (24.center) to (25.center);
		\draw (24.center) to (26.center);
		\draw (27.center) to (28.center);
		\draw (28.center) to (30.center);
		\draw (24.center) to (29.center);
		\draw (29.center) to (31.center);
		\draw (30.center) to (32.center);
		\draw (33.center) to (31.center);
		\draw (31.center) to (34.center);
		\draw (32.center) to (35.center);
		\draw (39.center) to (43.center);
		\draw (40.center) to (44.center);
		\draw (41.center) to (42.center);
		\draw (56.center) to (55.center);
		\draw (58.center) to (57.center);
		\draw (60.center) to (59.center);
		\draw (61.center) to (65.center);
		\draw (62.center) to (66.center);
		\draw (63.center) to (64.center);
	\end{pgfonlayer}
\end{tikzpicture}\]
The embedding $\frc{\cat{F}} \hookrightarrow \frc{\cat{SF}}$ identifies $\frc{\cat{F}}$ with the subcategory of planar diagrams in $\frc{\cat{SF}}$. These pictures of the morphisms in $\frc{\cat{SF}}$ are called strand diagrams \cite{belkF}, used by Brin in \cite{brinstrand} to represent elements of the braided Thompson group $BV$.

Strand diagrams with one vertex at the top and bottom, i.e. morphisms $1 \to 1$ in $\frc{\cat{SF}}$, form a group which is well-known to be isomorphic to Thompson's group $V$ \cite{belkF, brinstrand}. We take this to be our definition of Thompson's group $V$. Thompson's group $F$ consists of the planar diagrams in $V$.
\begin{definition}\label{def:thompsonsgroups}
    Define $V := \frc{\cat{SF}}(1,1)$ and $F := \frc{\cat{F}}(1,1)$.
\end{definition}
Each element $v \in V$ can be written as a fraction $(\tau s)^{-1} (\sigma t)$ for some trees $t,s$ with a common number of leaves $n$, and $\sigma, \tau \in S_n$, however we can write $v = s^{-1} \tau^{-1} \sigma t$, so that every element of $V$ is specified by two trees and a permutation. Restricting the embedding $\frc{\cat{F}} \hookrightarrow \frc{\cat{SF}}$ identifies $F$ with the subgroup of $V$ consisting of fractions $s^{-1} \sigma t$, where $\sigma$ is the identity permutation.

Since for each $m,n \in \NN$ there is a forest with $m$ roots and $n$ leaves, the groupoids $\frc{\cat{F}}$ and $\frc{\cat{SF}}$ are connected, and so they are each equivalent to their automorphism groups at the object $1 \in \NN$.
\subsection{Jones' technology}\label{subsec:jonesiso}
The category $\cat{SF}$ has a monoidal structure $\otimes$ given by addition of natural numbers, and horizontal concatenation of symmetric forests. For example,
\[
\begin{tikzpicture}
	\begin{pgfonlayer}{nodelayer}
		\node [style=none] (0) at (-5.75, 2) {};
		\node [style=none] (1) at (-4.75, 0) {};
		\node [style=none] (2) at (-3.75, 2) {};
		\node [style=none] (3) at (-5.75, 4) {};
		\node [style=none] (4) at (-3.75, 4) {};
		\node [style=none] (5) at (-2.5, 2) {$\otimes$};
		\node [style=none] (6) at (-1.25, 0) {};
		\node [style=none] (7) at (-1.25, 4) {};
		\node [style=none] (8) at (0, 2) {$=$};
		\node [style=none] (9) at (1.25, 2) {};
		\node [style=none] (10) at (2.25, 0) {};
		\node [style=none] (11) at (3.25, 2) {};
		\node [style=none] (12) at (1.25, 4) {};
		\node [style=none] (13) at (3.25, 4) {};
		\node [style=none] (15) at (4.75, 0) {};
		\node [style=none] (16) at (4.75, 4) {};
		\node [style=none] (17) at (5.5, 2) {.};
	\end{pgfonlayer}
	\begin{pgfonlayer}{edgelayer}
		\draw (1.center) to (2.center);
		\draw (1.center) to (0.center);
		\draw (0.center) to (4.center);
		\draw (2.center) to (3.center);
		\draw (6.center) to (7.center);
		\draw (10.center) to (11.center);
		\draw (10.center) to (9.center);
		\draw (9.center) to (13.center);
		\draw (11.center) to (12.center);
		\draw (15.center) to (16.center);
	\end{pgfonlayer}
\end{tikzpicture}
\]
The monoidal structure of \cat{SF} restricts to a monoidal structure on \cat{F}. These lift to monoidal structures on the groupoids $\frc{\cat{F}}$ and $\frc{\cat{SF}}$ via the formula $g^{-1} f \otimes k^{-1} h := (g \otimes k)^{-1} (f \otimes g)$, for all suitable symmetric forests $f,g,h,k \in \cat{SF}$. In terms of strand diagrams, the tensor product $\otimes$ is again horizontal concatenation of diagrams.

Every planar forest in \cat{F} can be written as a composition of elementary forests \[ f_{k,n} := |^ {\otimes(k-1)} \ \otimes \ \caret \ \otimes \ |^{\otimes (n - k)},\] for example \[\begin{tikzpicture}
	\begin{pgfonlayer}{nodelayer}
		\node [style=none] (0) at (-0.5, 0) {};
		\node [style=none] (1) at (-1.25, 1) {};
		\node [style=none] (2) at (0.5, 2) {};
		\node [style=none] (3) at (-2, 2) {};
		\node [style=none] (4) at (-0.5, 2) {};
		\node [style=none] (5) at (2.5, 0) {};
		\node [style=none] (6) at (3.25, 1) {};
		\node [style=none] (7) at (1.5, 2) {};
		\node [style=none] (8) at (4, 2) {};
		\node [style=none] (9) at (2.5, 2) {};
		\node [style=none] (10) at (7.5, -1.25) {};
		\node [style=none] (11) at (6.75, -0.25) {};
		\node [style=none] (12) at (8.25, -0.25) {};
		\node [style=none] (13) at (9.5, -1.25) {};
		\node [style=none] (14) at (9.5, -0.25) {};
		\node [style=none] (15) at (6.25, 0.75) {};
		\node [style=none] (16) at (7.25, 0.75) {};
		\node [style=none] (17) at (8.25, 0.75) {};
		\node [style=none] (18) at (9.5, 0.75) {};
		\node [style=none] (19) at (6.25, 1.75) {};
		\node [style=none] (20) at (7.25, 1.75) {};
		\node [style=none] (22) at (9, 1.75) {};
		\node [style=none] (23) at (8.25, 1.75) {};
		\node [style=none] (24) at (10, 1.75) {};
		\node [style=none] (25) at (6.25, 2.75) {};
		\node [style=none] (26) at (7.25, 2.75) {};
		\node [style=none] (27) at (8.25, 2.75) {};
		\node [style=none] (28) at (9, 2.75) {};
		\node [style=none] (29) at (9.5, 2.75) {};
		\node [style=none] (30) at (10.5, 2.75) {};
		\node [style=none] (31) at (5, 1) {$=$};
		\node [style=none] (32) at (6.75, -0.25) {};
		\node [style=none, label={right: $= f_{5,5} f_{4,4}  f_{1,3} f_{1,2}$.}] (33) at (11, 1) {};
	\end{pgfonlayer}
	\begin{pgfonlayer}{edgelayer}
		\draw (0.center) to (1.center);
		\draw (0.center) to (2.center);
		\draw (1.center) to (3.center);
		\draw (1.center) to (4.center);
		\draw (5.center) to (6.center);
		\draw (5.center) to (7.center);
		\draw (6.center) to (8.center);
		\draw (6.center) to (9.center);
		\draw (11.center) to (10.center);
		\draw (10.center) to (12.center);
		\draw (13.center) to (14.center);
		\draw (15.center) to (11.center);
		\draw (16.center) to (11.center);
		\draw (17.center) to (12.center);
		\draw (18.center) to (14.center);
		\draw (15.center) to (19.center);
		\draw (16.center) to (20.center);
		\draw (17.center) to (23.center);
		\draw (18.center) to (22.center);
		\draw (18.center) to (24.center);
		\draw (20.center) to (26.center);
		\draw (19.center) to (25.center);
		\draw (23.center) to (27.center);
		\draw (22.center) to (28.center);
		\draw (24.center) to (29.center);
		\draw (24.center) to (30.center);
	\end{pgfonlayer}
\end{tikzpicture}\] Thus, the morphisms of $\frc{\cat{F}}$ are generated by $\caret$ with respect to the operations of composition, $\otimes$, and $(-)^{-1}.$ Since every permutation $\sigma$ of a finite set can be written as a product of transpositions, the morphisms of $\frc{\cat{SF}}$ are generated by $\caret$ and the transposition $\sigma \in S_2$ as a monoidal groupoid. Thus, if $(\cat{D}, \otimes)$ is a monoidal category, then
\begin{itemize}
    \item any morphism of the form $\alpha: a \to a \otimes a$ yields a covariant monoidal functor $\Phi_\alpha : \cat{F} \to \cat{D},$  defined by $\Phi_\alpha(n) := a^{\otimes n}$ for all $n \in \NN$ and $\Phi_\alpha(\caret) := \alpha$. Similarly, a morphism $\omega: a \otimes a \to a \in \cat{D}$ defines a contravariant monoidal functor $\Phi_\omega: \cat{F} \to \cat{D}$ defined by $\Phi_\omega(\caret) := \omega$.
    \item If $R: K \to K \otimes K$ is an isomorphism in \cat{D}, $\Phi_R$ can be extended to a functor $\pi_R: \frc{\cat{F}} \to \cat{D}$ by setting $\pi_R(f^{-1}) := \Phi_R(f)^{-1}$ for all forests $f$.
    \item If $\cat{D}$ is symmetric monoidal, we may extend $\Phi_\alpha, \Phi_\omega$ to functors $\cat{SF} \to \cat{D}$, and extend $\Phi_R$ to a functor $ \pi_R: \frc{\cat{SF}} \to \cat{D}$. This is done by mapping the transposition $\sigma \in S_2$ to the canonical isomorphisms $a \otimes a \xrightarrow{\sim} a \otimes a$ for $\Phi_\alpha$ and $\Phi_\omega$, and $K \otimes K \xrightarrow{\sim} K \otimes K$ for $\pi_R$, using the symmetric structure of $\cat{D}$. Sometimes we will denote $\pi_R$ by $\Phi_R$.
\item Thus, if \cat{D} has a symmetric structure and $\sigma \in S_n$, $\Phi_\alpha(\sigma): a^{\otimes n} \to a^{\otimes n}$ is given by the action $S_n \curvearrowright a^{\otimes n}$, and $\Phi_\omega(\sigma) = \Phi_\alpha(\sigma)^{-1}$. Since $\pi_R :\frc{\cat{SF}} \to \cat{D}$ is covariant, $\pi_R(\sigma): K^{\otimes n} \to K^{\otimes n}$ is given by the action $S_n \curvearrowright K^{\otimes n}$.
\end{itemize}
\begin{example}
    If $(\cat{D}, \otimes) = (\cat{Set}, \times)$, the category of sets along with the cartesian product, $\alpha: X \to X \times X$ is a map of sets, and $\sigma \in S_n$, then $\Phi_\alpha(\sigma) : X^n \to X^n, (x_i) \mapsto (x_{\sigma^{-1}i})$.
\end{example}
\begin{remark}
    Given a morphism $\alpha: a \to a \otimes a$ in a monoidal category \cat{D}, the functor $\Phi_\alpha: \cat{F} \to \cat{D}$ has a nice visual description in terms of well-known graphical calculus for monoidal categories \cite{tensorcat}. Representing the morphism $\alpha$ diagramatically as a trivalent vertex
    \[
    \begin{tikzpicture}
	\begin{pgfonlayer}{nodelayer}
		\node [style=white circle] (0) at (0, 1) {$\alpha$};
		\node [style=none, label={left:{$\alpha =$}}] (1) at (-2.5, 1) {};
		\node [style=none] (2) at (0, -0.5) {};
		\node [style=none] (3) at (1.5, 2.5) {};
		\node [style=none] (4) at (-1.5, 2.5) {};
		\node [style=none] (5) at (2, 0.75) {,};
	\end{pgfonlayer}
	\begin{pgfonlayer}{edgelayer}
		\draw (0) to (3.center);
		\draw (0) to (4.center);
		\draw (0) to (2.center);
	\end{pgfonlayer}
\end{tikzpicture}
    \]
    the functor $\Phi_\alpha: \cat{F} \to \cat{D}$ is given by replacing every instance of a caret $\caret$ with the trivalent vertex corresponding to $\alpha$. For example,
    \[
    \begin{tikzpicture}
	\begin{pgfonlayer}{nodelayer}
		\node [style=none] (6) at (-7.25, 0) {};
		\node [style=none] (7) at (-8.25, 1) {};
		\node [style=none] (8) at (-6.25, 1) {};
		\node [style=none] (9) at (-9.25, 2) {};
		\node [style=none] (10) at (-7.25, 2) {};
		\node [style=none] (11) at (-9.25, 2) {};
		\node [style=none] (12) at (-6.25, 2) {};
		\node [style=white circle] (13) at (4, 0) {$\alpha$};
		\node [style=none] (14) at (4, -1.5) {};
		\node [style=none] (15) at (5.5, 1.5) {};
		\node [style=none] (16) at (2.5, 1.5) {};
		\node [style=white circle] (17) at (2.5, 2.5) {$\alpha$};
		\node [style=none] (18) at (2.5, 1.5) {};
		\node [style=none] (19) at (4, 4) {};
		\node [style=none] (20) at (1, 4) {};
		\node [style=none] (21) at (5.5, 4) {};
		\node [style=none, label={right:{$= \left((\alpha \otimes \text{id}_a) \circ \alpha\right) \ \otimes \ \alpha \ \otimes \ \text{id}_a$.}}] (22) at (12.5, 0.75) {};
		\node [style=none] (23) at (-0.75, 1) {};
		\node [style=none] (24) at (0.25, 1) {};
		\node [style=none] (25) at (-4, 0) {};
		\node [style=none] (26) at (-5, 2) {};
		\node [style=none] (27) at (-3, 2) {};
		\node [style=none] (28) at (-2, 2) {};
		\node [style=none] (29) at (-2, 0) {};
		\node [style=white circle] (30) at (8.25, 1.5) {$\alpha$};
		\node [style=none] (31) at (8.25, -2) {};
		\node [style=none] (32) at (9.75, 4) {};
		\node [style=none] (33) at (6.75, 4) {};
		\node [style=none] (35) at (11, -2) {};
		\node [style=none] (36) at (11, 4) {};
	\end{pgfonlayer}
	\begin{pgfonlayer}{edgelayer}
		\draw (6.center) to (8.center);
		\draw (6.center) to (7.center);
		\draw (7.center) to (10.center);
		\draw (7.center) to (11.center);
		\draw (8.center) to (12.center);
		\draw (13) to (15.center);
		\draw (13) to (16.center);
		\draw (13) to (14.center);
		\draw (17) to (19.center);
		\draw (17) to (20.center);
		\draw (18.center) to (17);
		\draw (15.center) to (21.center);
		\draw [{|->}] (23.center) to (24.center);
		\draw (25.center) to (26.center);
		\draw (25.center) to (27.center);
		\draw (28.center) to (29.center);
		\draw (30) to (32.center);
		\draw (30) to (33.center);
		\draw (30) to (31.center);
		\draw (36.center) to (35.center);
	\end{pgfonlayer}
\end{tikzpicture}
    \]
   The functor $\Phi_\omega$ corresponding to a morphism $\omega: a \otimes a \to a$ in \cat{D} can be visualised similarly.
   
   If $R : K \to K \otimes K$ is an isomorphism in \cat{D}, we can represent $R^{-1}$ diagramatically as a $180$ degree rotation of the trivalent vertex corresponding to $R$:
    \[
    \begin{tikzpicture}
	\begin{pgfonlayer}{nodelayer}
		\node [style=white circle] (0) at (0, 0) {$R$};
		\node [style=none, label={left:{$R^{-1}=$}}] (1) at (-3, 0) {};
		\node [style=none] (2) at (0, 1.5) {};
		\node [style=none] (3) at (1.5, -1.5) {};
		\node [style=none] (4) at (-1.5, -1.5) {};
		\node [style=none] (5) at (2, -0.25) {,};
	\end{pgfonlayer}
	\begin{pgfonlayer}{edgelayer}
		\draw (0) to (3.center);
		\draw (0) to (4.center);
		\draw (0) to (2.center);
	\end{pgfonlayer}
\end{tikzpicture}
    \]
    so that
    \[
    \begin{tikzpicture}
	\begin{pgfonlayer}{nodelayer}
		\node [style=white circle] (0) at (-0.5, 1.5) {$R$};
		\node [style=none, label={left:{$R^{-1} R =$}}] (1) at (-3, 0) {};
		\node [style=none] (2) at (-0.5, 3) {};
		\node [style=none] (3) at (1, 0) {};
		\node [style=none] (4) at (-2, 0) {};
		\node [style=white circle] (6) at (-0.5, -1.5) {$R$};
		\node [style=none] (7) at (-0.5, -3) {};
		\node [style=none] (8) at (1, 0) {};
		\node [style=none] (9) at (-2, 0) {};
		\node [style=none] (10) at (2, 0) {$=$};
		\node [style=none] (11) at (3, -1) {};
		\node [style=none] (12) at (3, 1) {};
		\node [style=none] (13) at (4, -0.25) {,};
	\end{pgfonlayer}
	\begin{pgfonlayer}{edgelayer}
		\draw (0) to (3.center);
		\draw (0) to (4.center);
		\draw (0) to (2.center);
		\draw (6) to (8.center);
		\draw (6) to (9.center);
		\draw (6) to (7.center);
		\draw (12.center) to (11.center);
	\end{pgfonlayer}
\end{tikzpicture} \tag{I}
    \]
    and
    \[
    \begin{tikzpicture}
	\begin{pgfonlayer}{nodelayer}
		\node [style=white circle] (0) at (0, -1.5) {$R$};
		\node [style=none, label={left:{$R^{-1} R =$}}] (1) at (-3, 0) {};
		\node [style=none] (2) at (0, 0) {};
		\node [style=none] (3) at (1.5, -3) {};
		\node [style=none] (4) at (-1.5, -3) {};
		\node [style=white circle] (5) at (0, 1.5) {$R$};
		\node [style=none] (6) at (0, 0) {};
		\node [style=none] (7) at (1.5, 3) {};
		\node [style=none] (8) at (-1.5, 3) {};
		\node [style=none] (9) at (2, 0) {$=$};
		\node [style=none] (10) at (3, -1) {};
		\node [style=none] (11) at (3, 1) {};
		\node [style=none] (12) at (5.25, -0.25) {.};
		\node [style=none] (13) at (4, -1) {};
		\node [style=none] (14) at (4, 1) {};
	\end{pgfonlayer}
	\begin{pgfonlayer}{edgelayer}
		\draw (0) to (3.center);
		\draw (0) to (4.center);
		\draw (0) to (2.center);
		\draw (5) to (7.center);
		\draw (5) to (8.center);
		\draw (5) to (6.center);
		\draw (11.center) to (10.center);
		\draw (14.center) to (13.center);
	\end{pgfonlayer}
\end{tikzpicture} \tag{II}
    \]
    Extending $\Phi_R$ to a functor $\pi_R: \frc{\cat{F}} \to \cat{D}$ and then restricting this functor to the group $F := \frc{\cat{F}}(1,1)$ yields a representation of $F$ in $\cat{D}$, for example
    
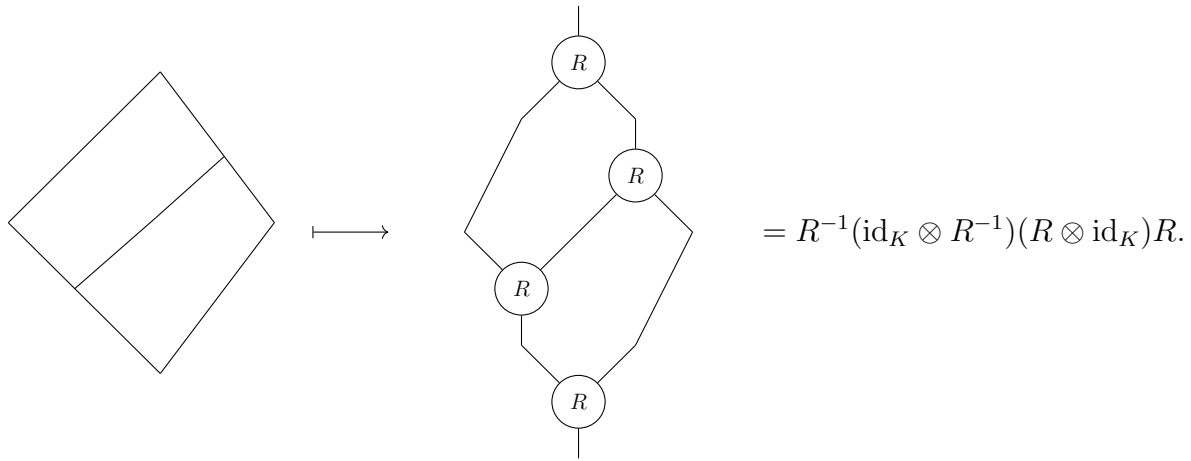
\[
    \begin{tikzpicture}
	\begin{pgfonlayer}{nodelayer}
		\node [style=none] (0) at (1, 0.5) {};
		\node [style=none] (1) at (-3, 4.5) {};
		\node [style=none] (2) at (4, 4.5) {};
		\node [style=none] (3) at (-1.25, 2.75) {};
		\node [style=none] (4) at (0.875, 4.625) {};
		\node [style=none] (5) at (1, 8.5) {};
		\node [style=none] (6) at (4, 4.5) {};
		\node [style=none] (7) at (-3, 4.5) {};
		\node [style=none] (8) at (2.675, 6.25) {};
		\node [style=none] (9) at (0.875, 4.625) {};
		\node [style=none] (11) at (5, 4.25) {};
		\node [style=none] (12) at (7, 4.25) {};
		\node [style=white circle] (13) at (12, -0.25) {$R$};
		\node [style=none] (14) at (12, -1.75) {};
		\node [style=none] (15) at (13.5, 1.25) {};
		\node [style=none] (16) at (10.5, 1.25) {};
		\node [style=white circle] (17) at (10.5, 2.75) {$R$};
		\node [style=none] (18) at (10.5, 1.25) {};
		\node [style=none] (19) at (12, 4.25) {};
		\node [style=none] (20) at (9, 4.25) {};
		\node [style=none] (21) at (13.5, 7.25) {};
		\node [style=white circle] (22) at (13.5, 5.75) {$R$};
		\node [style=none] (23) at (13.5, 7.25) {};
		\node [style=none] (24) at (15, 4.25) {};
		\node [style=none] (25) at (12, 4.25) {};
		\node [style=none] (26) at (10.5, 7.25) {};
		\node [style=none] (27) at (12, 10.25) {};
		\node [style=white circle] (28) at (12, 8.75) {$R$};
		\node [style=none] (29) at (12, 10.25) {};
		\node [style=none] (30) at (13.5, 7.25) {};
		\node [style=none] (31) at (10.5, 7.25) {};
		\node [style=none, label={right:{$= R^{-1} (\mathrm{id}_K \otimes R^{-1}) (R \otimes \mathrm{id}_K) R$.}}] (32) at (16.5, 4.25) {};
	\end{pgfonlayer}
	\begin{pgfonlayer}{edgelayer}
		\draw (0.center) to (1.center);
		\draw (0.center) to (2.center);
		\draw (3.center) to (4.center);
		\draw (5.center) to (6.center);
		\draw (5.center) to (7.center);
		\draw (8.center) to (9.center);
		\draw [{|->}] (11.center) to (12.center);
		\draw (13) to (15.center);
		\draw (13) to (14.center);
		\draw (13) to (16.center);
		\draw (17) to (19.center);
		\draw (17) to (18.center);
		\draw (17) to (20.center);
		\draw (22) to (24.center);
		\draw (22) to (23.center);
		\draw (22) to (25.center);
		\draw (28) to (30.center);
		\draw (28) to (29.center);
		\draw (28) to (31.center);
		\draw (15.center) to (24.center);
		\draw (20.center) to (31.center);
	\end{pgfonlayer}
\end{tikzpicture}
    \]
    If $\cat{D}$ is symmetric monoidal, we can extend this representation to the group $V$.

    For many choices of \cat{D}, Jones' technology mechanically assigns each morphism $\alpha: a \to a \otimes a$ or $\omega: a \otimes a \to a$ in \cat{D} an isomorphism $R: K \to K \otimes K$ via a colimit/limit construction, respectively (Remark \ref{rmk:constructionjonesiso}). I.e., any trivalent vertex $\alpha$ or $\omega$ is replaced by a trivalent vertex satisfying relations (I) and (II), in a canonical way, obtaining representations of Thompson's groups from easy data.
\end{remark}
\begin{example}\label{eg:vcantoraction}
    Let $\cat{D} := \cat{Set},$ $\otimes := \sqcup,$ and $\cantor := \{0,1\}^\NN,$ the Cantor space of infinite sequences in $0$ and $1.$ Consider the bijection $R: \mathfrak{C} \to \mathfrak{C} \sqcup \mathfrak{C}$ defined by the formula 
        \[ 
        x_1 x_2 x_3 \cdots \mapsto(x_1, x_2 x_3 \cdots).
        \] 
    This yields a monoidal functor $\Phi: \frc{\cat{SF}} \to \cat{Set}$ defined by $\Phi(\caret) := R$ which affords the following visual description. We view $\cantor$ as the boundary $\partial t_\infty$ of the rooted infinite regular binary tree $t_\infty.$ Thus, an element of $\cantor$ is viewed as an upwards path in $t_\infty$ consisting of left turns and right turns, where we use the convention that a left turn corresponds to $0$ and a right turn corresponds to $1$. For example, 
    
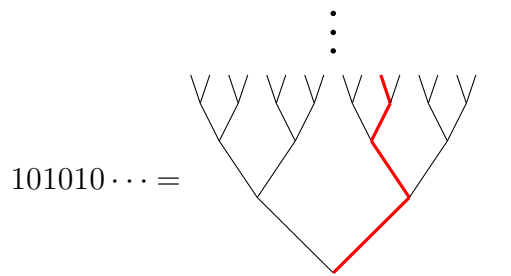
\[\begin{tikzpicture}
	\begin{pgfonlayer}{nodelayer}
		\node [style=none] (0) at (-2.25, 2.5) {$101010\cdots =$};
		\node [style=none] (1) at (4, 0) {};
		\node [style=none] (2) at (2, 2) {};
		\node [style=none] (3) at (6, 2) {};
		\node [style=none] (4) at (1, 3.5) {};
		\node [style=none] (5) at (3, 3.5) {};
		\node [style=none] (6) at (5, 3.5) {};
		\node [style=none] (7) at (7, 3.5) {};
		\node [style=none] (8) at (0.5, 4.5) {};
		\node [style=none] (9) at (1.5, 4.5) {};
		\node [style=none] (10) at (2.5, 4.5) {};
		\node [style=none] (11) at (3.5, 4.5) {};
		\node [style=none] (12) at (4.5, 4.5) {};
		\node [style=none] (13) at (5.5, 4.5) {};
		\node [style=none] (14) at (6.5, 4.5) {};
		\node [style=none] (15) at (7.5, 4.5) {};
		\node [style=none] (16) at (3, 3.5) {};
		\node [style=none] (17) at (0.25, 5.25) {};
		\node [style=none] (18) at (0.75, 5.25) {};
		\node [style=none] (19) at (1.25, 5.25) {};
		\node [style=none] (20) at (1.75, 5.25) {};
		\node [style=none] (21) at (2.25, 5.25) {};
		\node [style=none] (22) at (2.75, 5.25) {};
		\node [style=none] (23) at (3.25, 5.25) {};
		\node [style=none] (24) at (3.75, 5.25) {};
		\node [style=none] (25) at (4.25, 5.25) {};
		\node [style=none] (26) at (4.75, 5.25) {};
		\node [style=none] (27) at (5.25, 5.25) {};
		\node [style=none] (28) at (5.75, 5.25) {};
		\node [style=none] (29) at (6.75, 5.25) {};
		\node [style=none] (30) at (6.25, 5.25) {};
		\node [style=none] (31) at (7.25, 5.25) {};
		\node [style=none] (32) at (7.75, 5.25) {};
		\node [style=none] (33) at (5.25, 5.25) {};
		\node [style=none] (37) at (8.75, 2.5) {.};
		\node [style=none, scale=1.75] (40) at (4, 6.75) {$\vdots$};
	\end{pgfonlayer}
	\begin{pgfonlayer}{edgelayer}
		\draw (1.center) to (2.center);
		\draw [very thick, style=red] (1.center) to (3.center);
		\draw (2.center) to (4.center);
		\draw (2.center) to (16.center);
		\draw [very thick, style=red] (3.center) to (6.center);
		\draw (3.center) to (7.center);
		\draw (4.center) to (8.center);
		\draw (4.center) to (9.center);
		\draw (16.center) to (10.center);
		\draw (16.center) to (11.center);
		\draw (6.center) to (12.center);
		\draw [very thick, style=red] (6.center) to (13.center);
		\draw (7.center) to (14.center);
		\draw (7.center) to (15.center);
		\draw (8.center) to (17.center);
		\draw (8.center) to (18.center);
		\draw (9.center) to (19.center);
		\draw (9.center) to (20.center);
		\draw (10.center) to (21.center);
		\draw (10.center) to (22.center);
		\draw (11.center) to (23.center);
		\draw (11.center) to (24.center);
		\draw (12.center) to (25.center);
		\draw (12.center) to (26.center);
		\draw [very thick, style=red] (13.center) to (33.center);
		\draw (13.center) to (28.center);
		\draw (14.center) to (30.center);
		\draw (14.center) to (29.center);
		\draw (15.center) to (31.center);
		\draw (15.center) to (32.center);
	\end{pgfonlayer}
\end{tikzpicture}\]
    The bijection $R$ is then given by deleting a caret $\caret$ at the bottom of diagrams in $\cantor$ 
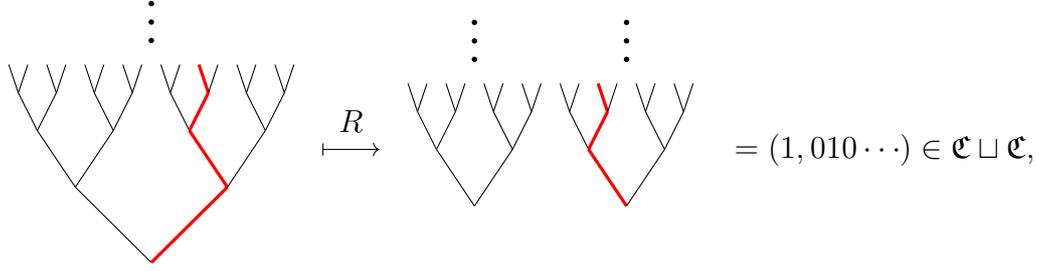
\[ \begin{tikzpicture}
	\begin{pgfonlayer}{nodelayer}
		\node [style=none] (1) at (6.75, 0) {};
		\node [style=none] (2) at (4.75, 2) {};
		\node [style=none] (3) at (8.75, 2) {};
		\node [style=none] (4) at (3.75, 3.5) {};
		\node [style=none] (5) at (5.75, 3.5) {};
		\node [style=none] (6) at (7.75, 3.5) {};
		\node [style=none] (7) at (9.75, 3.5) {};
		\node [style=none] (8) at (3.25, 4.5) {};
		\node [style=none] (9) at (4.25, 4.5) {};
		\node [style=none] (10) at (5.25, 4.5) {};
		\node [style=none] (11) at (6.25, 4.5) {};
		\node [style=none] (12) at (7.25, 4.5) {};
		\node [style=none] (13) at (8.25, 4.5) {};
		\node [style=none] (14) at (9.25, 4.5) {};
		\node [style=none] (15) at (10.25, 4.5) {};
		\node [style=none] (16) at (5.75, 3.5) {};
		\node [style=none] (17) at (3, 5.25) {};
		\node [style=none] (18) at (3.5, 5.25) {};
		\node [style=none] (19) at (4, 5.25) {};
		\node [style=none] (20) at (4.5, 5.25) {};
		\node [style=none] (21) at (5, 5.25) {};
		\node [style=none] (22) at (5.5, 5.25) {};
		\node [style=none] (23) at (6, 5.25) {};
		\node [style=none] (24) at (6.5, 5.25) {};
		\node [style=none] (25) at (7, 5.25) {};
		\node [style=none] (26) at (7.5, 5.25) {};
		\node [style=none] (27) at (8, 5.25) {};
		\node [style=none] (28) at (8.5, 5.25) {};
		\node [style=none] (29) at (9.5, 5.25) {};
		\node [style=none] (30) at (9, 5.25) {};
		\node [style=none] (31) at (10, 5.25) {};
		\node [style=none] (32) at (10.5, 5.25) {};
		\node [style=none] (33) at (8, 5.25) {};
		\node [style=none, scale=1.75] (35) at (6.75, 6.75) {$\vdots$};
		\node [style=none] (37) at (12, 3.75) {$R$};
		\node [style=none] (39) at (15.25, 1.5) {};
		\node [style=none] (40) at (19.25, 1.5) {};
		\node [style=none] (41) at (14.25, 3) {};
		\node [style=none] (42) at (16.25, 3) {};
		\node [style=none] (43) at (18.25, 3) {};
		\node [style=none] (44) at (20.25, 3) {};
		\node [style=none] (45) at (13.75, 4) {};
		\node [style=none] (46) at (14.75, 4) {};
		\node [style=none] (47) at (15.75, 4) {};
		\node [style=none] (48) at (16.75, 4) {};
		\node [style=none] (49) at (17.75, 4) {};
		\node [style=none] (50) at (18.75, 4) {};
		\node [style=none] (51) at (19.75, 4) {};
		\node [style=none] (52) at (20.75, 4) {};
		\node [style=none] (53) at (16.25, 3) {};
		\node [style=none] (54) at (13.5, 4.75) {};
		\node [style=none] (55) at (14, 4.75) {};
		\node [style=none] (56) at (14.5, 4.75) {};
		\node [style=none] (57) at (15, 4.75) {};
		\node [style=none] (58) at (15.5, 4.75) {};
		\node [style=none] (59) at (16, 4.75) {};
		\node [style=none] (60) at (16.5, 4.75) {};
		\node [style=none] (61) at (17, 4.75) {};
		\node [style=none] (62) at (17.5, 4.75) {};
		\node [style=none] (63) at (18, 4.75) {};
		\node [style=none] (64) at (18.5, 4.75) {};
		\node [style=none] (65) at (19, 4.75) {};
		\node [style=none] (66) at (20, 4.75) {};
		\node [style=none] (67) at (19.5, 4.75) {};
		\node [style=none] (68) at (20.5, 4.75) {};
		\node [style=none] (69) at (21, 4.75) {};
		\node [style=none] (70) at (18.5, 4.75) {};
		\node [style=none, label={right:$= (1, 010\cdots) \in \cantor \sqcup \cantor$,}] (71) at (21.75, 3) {};
		\node [style=none] (72) at (11.25, 3) {};
		\node [style=none] (73) at (12.75, 3) {};
		\node [style=none, scale=1.75] (74) at (15.25, 6.25) {$\vdots$};
		\node [style=none, scale=1.75] (75) at (19.25, 6.25) {$\vdots$};
	\end{pgfonlayer}
	\begin{pgfonlayer}{edgelayer}
		\draw (1.center) to (2.center);
		\draw [very thick, style=red] (1.center) to (3.center);
		\draw (2.center) to (4.center);
		\draw (2.center) to (16.center);
		\draw [very thick, style=red] (3.center) to (6.center);
		\draw (3.center) to (7.center);
		\draw (4.center) to (8.center);
		\draw (4.center) to (9.center);
		\draw (16.center) to (10.center);
		\draw (16.center) to (11.center);
		\draw (6.center) to (12.center);
		\draw [very thick, style=red] (6.center) to (13.center);
		\draw (7.center) to (14.center);
		\draw (7.center) to (15.center);
		\draw (8.center) to (17.center);
		\draw (8.center) to (18.center);
		\draw (9.center) to (19.center);
		\draw (9.center) to (20.center);
		\draw (10.center) to (21.center);
		\draw (10.center) to (22.center);
		\draw (11.center) to (23.center);
		\draw (11.center) to (24.center);
		\draw (12.center) to (25.center);
		\draw (12.center) to (26.center);
		\draw [very thick, style=red] (13.center) to (33.center);
		\draw (13.center) to (28.center);
		\draw (14.center) to (30.center);
		\draw (14.center) to (29.center);
		\draw (15.center) to (31.center);
		\draw (15.center) to (32.center);
		\draw (39.center) to (41.center);
		\draw (39.center) to (53.center);
		\draw [very thick, style=red] (40.center) to (43.center);
		\draw (40.center) to (44.center);
		\draw (41.center) to (45.center);
		\draw (41.center) to (46.center);
		\draw (53.center) to (47.center);
		\draw (53.center) to (48.center);
		\draw (43.center) to (49.center);
		\draw [very thick, style=red] (43.center) to (50.center);
		\draw (44.center) to (51.center);
		\draw (44.center) to (52.center);
		\draw (45.center) to (54.center);
		\draw (45.center) to (55.center);
		\draw (46.center) to (56.center);
		\draw (46.center) to (57.center);
		\draw (47.center) to (58.center);
		\draw (47.center) to (59.center);
		\draw (48.center) to (60.center);
		\draw (48.center) to (61.center);
		\draw (49.center) to (62.center);
		\draw (49.center) to (63.center);
		\draw [very thick, style=red] (50.center) to (70.center);
		\draw (50.center) to (65.center);
		\draw (51.center) to (67.center);
		\draw (51.center) to (66.center);
		\draw (52.center) to (68.center);
		\draw (52.center) to (69.center);
		\draw [{|->}] (72.center) to (73.center);
	\end{pgfonlayer}
\end{tikzpicture}\] with $R^{-1}$ given by placing a caret at the bottom of a diagram in $\cantor \sqcup \cantor$ and joining the path to the root. Since $\Phi$ is a monoidal functor, given a tree $t \in \mathfrak{T},$ we have that $\Phi(t)$ is given by removing a copy of $t$ at the bottom of diagrams in $\cantor$. If $\sigma \in S_n$ is a permutation, then $\Phi(\sigma): \cantor^{\sqcup n} \to \cantor^{\sqcup n}$ is given by permuting copies of $\cantor$.
\end{example}
\begin{definition}\label{def:jonesiso}
Let \cat{D} be a monoidal category, and suppose that $\alpha: a \to a \otimes a$ is a morphism in \cat{D}. The \textit{Jones isomorphism} of $\alpha$ consists of an object $\joneso{\alpha} \in \cat{D},$ a morphism $\iota(\alpha): a \to \joneso{\alpha},$ and an isomorphism $\jonesm{\alpha}: \joneso{\alpha} \to \joneso{\alpha} \otimes \joneso{\alpha}$ making the diagram
\[
    \begin{tikzcd}
    K(\alpha) \arrow[r, "R(\alpha)"]                  & K(\alpha) \otimes K(\alpha)                                   \\
    a \arrow[r, "\alpha"'] \arrow[u, "\iota(\alpha)"] & a \otimes a \arrow[u, "\iota(\alpha) \otimes \iota(\alpha)"']
    \end{tikzcd}
\]
commute, while being universal with respect to this property: if $\wt{R}: \wt{K} \to \wt{K} \otimes \wt{K}$ is an isomorphism and $j: a \to \wt{K}$ are such that the diagram
\[\begin{tikzcd}
	{\wt{K}} & {\wt{K} \otimes \wt{K}} \\
	a & {a \otimes a}
	\arrow["{\wt{R}}", from=1-1, to=1-2]
	\arrow["\alpha"', from=2-1, to=2-2]
	\arrow["j", from=2-1, to=1-1]
	\arrow["{j \otimes j}"', from=2-2, to=1-2]
\end{tikzcd}\]
commutes, then there is a unique morphism $\psi: K(\alpha) \to \wt{K}$ making the diagram
\[\begin{tikzcd}
	{\wt{K}} && {\wt{K} \otimes \wt{K}} \\
	\\
	{K(\alpha)} && {K(\alpha) \otimes K(\alpha)} \\
	\\
	a && {a \otimes a}
	\arrow["{\iota(\alpha)}", from=5-1, to=3-1]
	\arrow["{\iota(\alpha) \otimes \iota(\alpha)}", from=5-3, to=3-3]
	\arrow["\alpha"', from=5-1, to=5-3]
	\arrow["{R(\alpha)}", from=3-1, to=3-3]
	\arrow["\psi", from=3-1, to=1-1, dashed]
	\arrow["{\psi \otimes \psi}"', from=3-3, to=1-3]
	\arrow["{\wt{R}}", from=1-1, to=1-3]
	\arrow["j", bend left=60, from=5-1, to=1-1]
	\arrow["{j \otimes j}"', bend right=60, from=5-3, to=1-3]
\end{tikzcd}
\] commute. Let $\pi(\alpha): \frc{\cat{F}} \to \cat{D}$ be the monoidal functor defined by $\pi(\alpha)(\caret) := R(\alpha)$, extending $\pi(\alpha)$ to a functor $\frc{\cat{SF}} \to \cat{D}$ if \cat{D} is symmetric monoidal.

The Jones isomorphism of a morphism $\omega: a \otimes a \to a$ in \cat{D} consists of an object $\joneso{\omega} \in \cat{D},$ a morphism $\prj{\omega}: \joneso{\omega} \to a$, and an isomorphism $\jonesm{\omega}: \joneso{\omega} \xrightarrow{\sim} \joneso{\omega} \otimes \joneso{\omega}$ making the diagram
\[
\begin{tikzcd}
K(\omega) \arrow[r, "R(\omega)"] \arrow[d, "p(\omega)"'] & K(\omega) \otimes K(\omega) \arrow[d, "p(\omega) \otimes p(\omega)"] \\
a                                                        & a \otimes a \arrow[l, "\omega"]                                               
\end{tikzcd}
\]
commute in a universal way: if $\wt{R}: \wt{K} \to \wt{K} \otimes \wt{K}$ and $q: \wt{K} \to a$ make the diagram 
\[\begin{tikzcd}
	{\wt{K}} & {\wt{K} \otimes \wt{K}} \\
	a & {a \otimes a}
	\arrow["q"', from=1-1, to=2-1]
	\arrow["{q \otimes q}", from=1-2, to=2-2]
	\arrow["{\wt{R}}", from=1-1, to=1-2]
	\arrow["\omega", from=2-2, to=2-1]
\end{tikzcd}\]
commute, then there is a unique morphism $\psi: \wt{K} \to K(\omega)$ making the diagram 
\[\begin{tikzcd}
	{\wt{K}} && {\wt{K} \otimes \wt{K}} \\
	\\
	{K(\omega)} && {K(\omega) \otimes K(\omega)} \\
	\\
	a && {a \otimes a}
	\arrow["{p(\omega)}"', from=3-1, to=5-1]
	\arrow["{p(\omega) \otimes p(\omega)}"', from=3-3, to=5-3]
	\arrow["{R(\omega)}", from=3-1, to=3-3]
	\arrow["\omega", from=5-3, to=5-1]
	\arrow["{\wt{R}}", from=1-1, to=1-3]
	\arrow["\psi"', dashed, from=1-1, to=3-1]
	\arrow["{\psi \otimes \psi}", from=1-3, to=3-3]
	\arrow["q"', bend right =60, from=1-1, to=5-1]
	\arrow["{q \otimes q}", bend left = 60, from=1-3, to=5-3]
\end{tikzcd}\]
commute. 

Denote by $\pi(\omega)$ the monoidal functor $\frc{\cat{F}} \to \cat{D}$ defined by $\pi(\omega)(\caret) := \jonesm{\omega}$, again extending to a functor $\frc{\cat{SF}} \to \cat{D}$ if \cat{D} is symmetric monoidal. We omit the dependence on morphisms in the notation for Jones isomorphisms, for example writing $K, R, p, \pi$ instead of $K(\omega), R(\omega), p(\omega), \pi(\omega)$, when possible.
\end{definition}
\begin{remark}
If \cat{E} is a category and $e$ is an object of $\cat{E}$, denote by $e / \cat{E}$ the category whose objects are morphisms $e \to x$ with $x$ an object of $\cat{E}$, a morphism $(e \to x) \to (e \to y)$ being a morphism $x \to y$ making the relevant diagram commute. The category $\cat{E}/e$ is defined similarly, with objects being morphisms $x \to e$.

Let \cat{D} be a monoidal category, and \cat{M} be the category whose objects are morphisms in \cat{D} of the form $\alpha: a \to a \otimes a$ or $\omega: a \otimes a \to a$, and whose morphisms are commutative diagrams such as
\[\begin{tikzcd}
	a & {a \otimes a} \\
	{\wt{a}} & {\wt{a} \otimes \wt{a}}
	\arrow["\alpha", from=1-1, to=1-2]
	\arrow["{\wt{\omega}}", from=2-2, to=2-1]
	\arrow["\psi"', from=1-1, to=2-1]
	\arrow["{\psi \otimes \psi}", from=1-2, to=2-2]
\end{tikzcd} \ . \tag{$*$}\]
The diagram $(*)$ is a morphism $\alpha \to \wt{\omega}$ in $\cat{M},$ and the remaining morphisms in \cat{M} are obtained via changing the directions of the horizontal arrows in $(*).$ Let \cat{I} be the full subcategory of \cat{M} whose objects are the isomorphisms of the form $K \xrightarrow{\sim} K \otimes K$. We may then rephrase the universal properties of Definition \ref{def:jonesiso} as follows: given an object $\alpha: a \to a \otimes a$ in \cat{M}, a Jones isomorphism of $\alpha$ if it exists is an initial object of $\alpha / \cat{I},$ and is thus defined up to unique isomorphism in $\alpha / \cat{I}$. Similarly, if $\omega: a \otimes a \to a \in \cat{M},$ a Jones isomorphism of $\omega$ is a terminal object in $\cat{I} / \omega.$ This describes Jones isomorphisms as Kan extensions of their underlying monoidal functors $\cat{F} \to \cat{D}$, along the localisation functor $\cat{F} \to \frc{\cat{F}}$.
\end{remark}
\begin{example}\label{eg:cantordyadicjonesiso}
    Let $\cat{D} := \cat{Set}$ and $\otimes := \sqcup$, the disjoint union of sets. Consider a singleton set $\{\bullet\}$, the unique map $\omega : \{\bullet\} \sqcup \{\bullet\} \to \{\bullet\}$, and let $\alpha: \{\bullet\} \to \{\bullet\} \sqcup \{\bullet\}, \bullet \mapsto (0, \bullet)$. Then the bijection $R: \cantor \to \cantor \sqcup \cantor$ of Example \ref{eg:vcantoraction} along with the unique map $\cantor \to \{\bullet\}$ is the Jones isomorphism for $\omega$. Restricting $R$ to $ \QQ_2 := \{0,1\}^{(\NN)}$, the finitely supported elements of $\cantor$, recovers the Jones isomorphism of $\alpha$ (the required map $\{\bullet\} \to \QQ_2$ maps $\bullet \mapsto 00\cdots$).
\end{example}
\begin{remark}\label{rmk:constructionjonesiso}
    When a monoidal category \cat{D} has enough colimits and limits (and under a couple of technical assumptions), Jones isomorphisms for morphisms $\alpha: a \to a \otimes a$ and $\omega: a \otimes a \to a$ in \cat{D} (i.e. monoidal functors $\cat{F} \to \cat{D}$) are certain direct limits and inverse limits in \cat{D} over the directed set $\mathfrak{T}$ of binary trees, respectively. This is true for all of the cases we are interested in, including Example \ref{eg:cantordyadicjonesiso}. These are special cases of colimit/limit constructions for arbitrary functors $\Phi: \cat{C} \to \cat{D}$, that lead to actions of $\frc{\cat{C}}$ when $\cat{C}$ is a category with a calculus of fractions. See \cite{nogo} for the colimit construction and \cite{bsgauge} for the limit construction when \cat{D} is a concrete category. As such, the group $K(\omega)$ (Definition \ref{def:komega}) assigned to a group morphism $\omega: \Gamma \times \Gamma \to \Gamma$ is secretely an inverse limit over the set of trees $\mathfrak{T}$.
\end{remark}
\subsection{Thompson's groups}
 Recall the bijection $R: \cantor \to \cantor \sqcup \cantor$ of Example \ref{eg:vcantoraction}, and the resulting functor $\Phi: \frc{\cat{SF}} \to \cat{Set}$ defined by $\Phi(\caret) := R$. We then obtain an action $V \curvearrowright \cantor$ by restricting $\Phi$ to $V$ (Definition \ref{def:thompsonsgroups}).
\begin{lemma}\label{lem:prefixswap}
    Suppose that $v = s^{-1} \sigma t \in V, \ 1 \leq i \leq \ell(t),$ and $x \in \cantor.$ Then \[ \Phi(v)( l_t^i x) = l_s^{\sigma(i)} x.\] I.e., $V$ acts on $\cantor$ via prefix replacement.
\end{lemma}
\begin{proof}
    This follows from the visual description of the functor $\Phi$ given in \ref{eg:cantordyadicjonesiso}. We have that \[\Phi(t)(l_t^i x) = (i, x) \in \cantor^{\sqcup \ell(t)}\] since $\Phi(t)(l_t^i x)$ is given by removing a copy of $t$ from the bottom of the diagram corresponding to $l_t^i x.$ Thus, \[\Phi(\sigma t)(l_t^i x) = (\sigma(i), x), \] and so \[ \Phi(s^{-1} \sigma t)(l_t^i) = l_s^{\sigma(i)} x,\] since $\Phi(s)^{-1}$ is given by placing a copy of $s$ at the bottom of diagrams in $\cantor^{\sqcup \ell(s)}.$
\end{proof} An immediate consequence of Lemma \ref{lem:prefixswap} is that the action $V \curvearrowright \cantor$ is faithful. If $v = s^{-1} \sigma t \in V$ acts trivially on $\cantor,$ then for all $1 \leq i \leq \ell(t)$ we have that \[\Phi(v)(l_t^i 00 \cdots) = l_s^{\sigma(i)} 00 \cdots = l_t^{i} 00 \cdots, \] and so 
\[ l_t^i = l_{s}^{\sigma(i)}. \tag{$*$} \] This implies that $\Leaf(t) = \Leaf(s),$ and so $t = s.$ But then $(*)$ implies that $\sigma = \mathrm{id}$, so $v = e_V.$

We only needed the fact that $\Phi(v)$ acts trivially on $\QQ_2 := \{0,1\}^{(\NN)}$, so we also get that the action $V \curvearrowright \QQ_2$ is faithful.
\begin{definition}
    Given a word $u \in \{0,1\}^*,$ define \[ \cantor_u := \{u \cdot x \ | \ x \in \cantor\}.\] Such a subset of $\cantor$ is called a \textit{standard dyadic interval}, or sdi for short. The $\cantor_u, \ u \in \{0,1\}^*$ are clopen subsets of $\cantor$ forming a basis for the product topology on $\cantor$ after equipping $\{0,1\}$ with the discrete topology. A standard dyadic partition, or sdp, is a finite partition of $\cantor$ into standard dyadic intervals. Given an sdi $I = \cantor_u,$ we denote by $I_0$ and $I_1$ the first and second halves of $I,$ i.e. \[I_0 := \cantor_{u0}, \ I_1 := \cantor_{u1},\] so that $I = I_0 \sqcup I_1$.
\end{definition}
Every tree $t \in \mathfrak{T}$ yields an sdp $\{\cantor_{l}\}_{l \in \Leaf(t)}$ of $\cantor,$ and so the formula of Lemma \ref{lem:prefixswap} completely specifies the map $\Phi(v)$ for each $v \in V.$ Indeed, $\Phi(v)$ is a homeomorphism of $\cantor$ for each $v \in V,$ and since the action $V \curvearrowright \cantor$ is faithful, we may view $V$ as a subgroup of $\text{Homeo}(\cantor)$, the group of homeomorphisms of $\cantor$. As such, we write $v(x)$ instead of $\Phi(v)(x)$ for all $v \in V$ and $x \in \cantor$.
\begin{definition}\label{def:slope}
    Suppose that $v = s^{-1} \sigma t \in V$ and that $x \in \cantor,$ and write $x = l_t^i y$ for some unique $1 \leq i \leq \ell(t)$ and $y \in \cantor.$ Then $v(x) = l_s^{\sigma(i)} y,$ and we define the slope of $v$ at $x,$ denoted by $v'(x)$ to be \[ v'(x) := \frac{2^{|l_t^i|}}{2^{|l_s^{\sigma(i)}|}}. \]
\end{definition}
\begin{remark}
    Consider the binary expansion map $S: \cantor \to [0,1]$ defined by the formula \[S(x) = \sum_{k=1}^\infty \frac{x_k}{2^k}, \ x \in \cantor.\] Given an sdi $I = \cantor_u,$ a measure of the length of $I$ is the quantity $S(u11\cdots) - S(u00\cdots) = 2^{-|u|}.$ Thus, if $v \in V,$ the slope of $v$ at $x \in \cantor$ is the ratio of the length of the sdi $\cantor_l, \ l \in \Leaf(s)$ containing $v(x)$, to the length of the sdi $\cantor_{\sigma^{-1}(l)}$ containing $x$.

    The map $S$ restricts to a bijection between the elements of $\cantor$ with infinitely many zeros and the unit interval $[0,1) \subseteq \RR$. If $v \in V$ and $x \in \cantor$ has infinitely many zeros, so does $v(x)$, so the map $S$ establishes an isomorphism between $V$ and a subgroup of the piecewise linear right-continuous bijections of the unit interval $[0,1),$ which is essentially the definition of $V$ given in \cite{cfp}. For example, the element
    
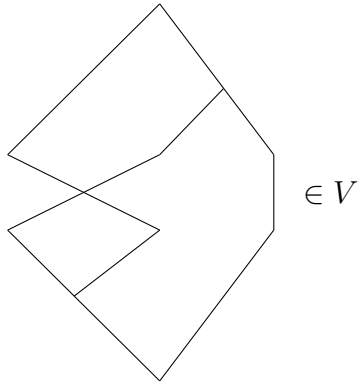
\[
    \begin{tikzpicture}
	\begin{pgfonlayer}{nodelayer}
		\node [style=none] (0) at (0, 0) {};
		\node [style=none] (1) at (-4, 4) {};
		\node [style=none] (2) at (3, 4) {};
		\node [style=none] (3) at (-2.25, 2.25) {};
		\node [style=none] (4) at (0, 4) {};
		\node [style=none] (5) at (0, 10) {};
		\node [style=none] (6) at (3, 6) {};
		\node [style=none] (7) at (-4, 6) {};
		\node [style=none] (8) at (1.675, 7.75) {};
		\node [style=none] (9) at (0, 6) {};
		\node [style=none] (10) at (4.5, 5) {$\in V$};
	\end{pgfonlayer}
	\begin{pgfonlayer}{edgelayer}
		\draw (0.center) to (1.center);
		\draw (0.center) to (2.center);
		\draw (3.center) to (4.center);
		\draw (5.center) to (6.center);
		\draw (5.center) to (7.center);
		\draw (8.center) to (9.center);
		\draw (1.center) to (9.center);
		\draw (4.center) to (7.center);
		\draw (2.center) to (6.center);
	\end{pgfonlayer}
\end{tikzpicture}
    \] corresponds to the bijection
    
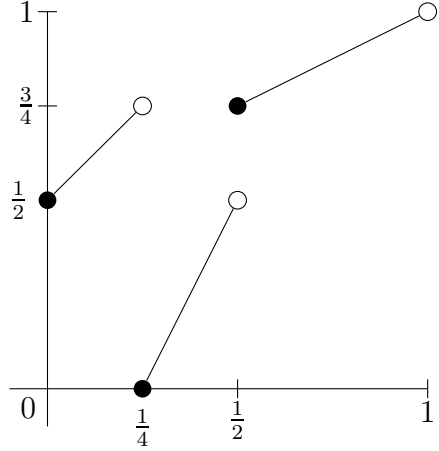
\[
    \begin{tikzpicture}
	\begin{pgfonlayer}{nodelayer}
		\node [style=none, label={left:$1$}] (0) at (9, 10) {};
		\node [style=none] (1) at (9, 0) {};
		\node [style=none, label={below:$1$}] (2) at (19, 0) {};
		\node [style=none] (3) at (8, 0) {};
		\node [style=none] (4) at (9, -1) {};
		\node [style=none, label={below:$\frac12$}] (5) at (14, 0) {};
		\node [circle, fill=black, scale=.6, label={below:$\frac14$}] (6) at (11.5, 0) {};
		\node [circle, fill=black, scale=.6, label={left:$\frac12$}] (7) at (9, 5) {};
		\node [style=none, label={left:$\frac34$}] (8) at (9, 7.5) {};
		\node [circle, fill=white, draw=black, scale=.6] (9) at (11.5, 7.5) {};
		\node [circle, fill=white, draw=black, scale=.6] (10) at (14, 5) {};
		\node [circle, fill=black, scale=.6] (11) at (14, 7.5) {};
		\node [circle, fill=white, draw=black, scale=.6] (12) at (19, 10) {};
		\node [style=none] (13) at (8.5, -0.5) {$0$};
		\node [style=none] (14) at (8.75, 7.5) {};
		\node [style=none] (15) at (9.25, 7.5) {};
		\node [style=none] (16) at (8.75, 10) {};
		\node [style=none] (17) at (9.25, 10) {};
		\node [style=none] (18) at (14, -0.25) {};
		\node [style=none] (19) at (14, 0.25) {};
		\node [style=none] (20) at (19, -0.25) {};
		\node [style=none] (21) at (19, 0.25) {};
	\end{pgfonlayer}
	\begin{pgfonlayer}{edgelayer}
		\draw (0.center) to (1.center);
		\draw (1.center) to (2.center);
		\draw (1.center) to (3.center);
		\draw (1.center) to (4.center);
		\draw (7) to (9);
		\draw (6) to (10);
		\draw (11) to (12);
		\draw (14.center) to (15.center);
		\draw (16.center) to (17.center);
		\draw (19.center) to (18.center);
		\draw (21.center) to (20.center);
	\end{pgfonlayer}
\end{tikzpicture}
    \]
    of $[0,1)$, with a black dots emphasising that a point is on the graph, and white dot emphasising that a point is not on the graph. Note that the tree at the bottom of the diagram of the element of $V$ corresponds to the domain (horizontal axis), and the tree on the top corresponds to the range (vertical axis).
    
    The slope $v'(x)$ with $v \in V$ and $x \in \cantor \setminus \{11\cdots\}$ is then equal to the right-slope of the graph of $v$ at $S(x),$ and $v'(11\cdots)$ is the limit of the slope of the graph of $v$ at $x \in [0,1)$ as $x \to 1$.
\end{remark}
Given $x \in \cantor,$ one can use the definition of slopes to show that the map $V \to 2^{\ZZ}, v \mapsto v'(x)$ satisfies the chain rule, and we will use this fact freely. 
    \begin{definition}\label{def:nvcstabn}
        Viewing Thompson's group $V$ as a subgroup of $\mathrm{Homeo}(\cantor)$, and identifying the dyadic rationals $\QQ_2$ with the finitely supported sequences in $\cantor$, define 
            \[
            \nvc := \{ \varphi \in \text{Homeo}(\cantor) \ | \ \ad_{\varphi}(v) := \varphi v \varphi^{-1} \in V\}
            \]
            and
            \[
            \stabn := \{\varphi \in \nvc \ | \ \varphi(\QQ_2) = \QQ_2\}.
            \]
    \end{definition}
The following proposition is crucial for our main results, and follows from a theorem of Rubin \cite{rub96}, the details of which can be found in \cite[Section 3]{chameleon}.
    \begin{proposition}\label{prop:autv}
        The map \[\nvc \to \aut(V), \varphi \mapsto \ad_\varphi: v \mapsto \varphi v \varphi^{-1}\] is an isomorphism of groups.
    \end{proposition}
We will also need the following facts about elements of $\nvc$ and $\stabn$, both of which are proved in \cite[Section 1]{jonestechI}.
\begin{lemma}\label{lem:phisdi}
    Suppose that $\varphi \in \nvc$ and that $I$ is an sdi. Then $\varphi(I)$ is a finite union of sdis.
\end{lemma}
\begin{proposition}\label{prop:homeoslope}
    Let $v \in V, \ x \in \cantor$ with $v(x) = x,$ and suppose that $\varphi \in \stabn.$ Then \[(ad_\varphi(v))'(\varphi(x)) = v'(x).\]
\end{proposition}
\begin{remark}
    The formula of Proposition \ref{prop:homeoslope} follows immediately from the chain rule in the case that $\varphi \in V$. This formula no longer holds in general when $\varphi \notin \stabn$ \cite[Remark 1.6]{jonestechI}, but can be generalised to accomodate any homeomorphism $\varphi \in \nvc$ \cite[Proposition 1.2]{jonestechII}.
\end{remark}

\section{Groups constructed from a Jones action}\label{sec:characteristic}
In this section we study the properties of groups $G = K \rtimes V$ coming from groups $K$ with an isomorphism $R: K \to K \times K$, focusing on isomorphisms between such groups $G$.
\begin{notation}\label{not:Rsemidirect}
If $R: K \to K \times K$ is an isomorphism of groups, let $\pi(R): \frc{\cat{SF}} \to \cat{Grp}$ the resulting functor defined by $\pi(R)(\caret) := R$. Restricting $\pi(R)$ to Thompson's group $V$ yields an action $V \curvearrowright K$, and thus a semidirect product $G(R) := K \rtimes V$. We will often drop the dependence on $R$ in our notations, writing $G$ instead of $G(R)$ and $\pi$ instead of $\pi(R)$ if the context allows.

When considering two group isomorphisms $R: K \to K^2$ and $\wt{R}: \wt{K} \to \wt{K}^2,$ we will often write $\wt{G}$ and $\wt{\pi}$ instead of $G(\wt{R})$ and $\pi(\wt{R})$.
\end{notation}
\begin{definition}\label{def:gensupp}
Suppose that $\Gamma$ is a group, and that $\alpha: \Gamma \to \Gamma^2, g \mapsto (\alpha_0(g), \alpha_1(g))$ is a group morphism. For each word $u = u_0 u_1 \cdots u_m \in \{0,1\}^*,$ define $\alpha_u := \alpha_{u_m} \circ \cdots \circ \alpha_{u_0}$ (notice the reversal of the letters of $u$). We use the convention that $\alpha_\epsilon := \text{id}_\Gamma.$

For each $g \in \Gamma,$ we define 
\[
\supp_{\alpha}(g) :=  \{x = x_0x_1\cdots \in \cantor \ | \ \alpha_{x_0\cdots x_m}(g) \neq e_\Gamma \text{ for all } m \geq 0\},
\] writing $\supp(g)$ instead when the context is clear.
\end{definition}
Fix a group isomorphism $R: K \to K^2.$ We regard each $a \in K$ as a map $\{0,1\}^* \to K$ by defining $a(u) := R_u(a)$ for all $u \in \{0,1\}^*$. In particular, $a(\epsilon) = a$. The infinite regular binary tree $t_\infty$ has vertex set $\{0,1\}^*$, so the elements $a \in K$ may be visualised as decorations of $t_\infty$ as follows: 
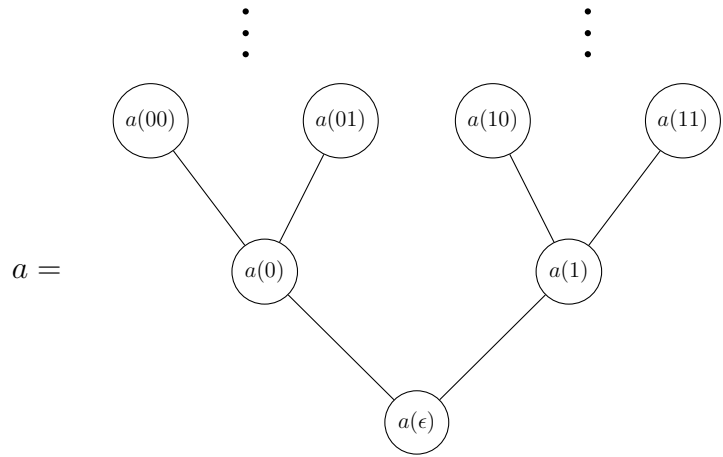
\[
\begin{tikzpicture}
	\begin{pgfonlayer}{nodelayer}
		\node [style=white circle] (0) at (0, 0) {$a(\epsilon)$};
		\node [style=white circle] (1) at (-4, 4) {$a(0)$};
		\node [style=white circle] (2) at (4, 4) {$a(1)$};
		\node [style=white circle] (3) at (-7, 8) {$a(00)$};
		\node [style=white circle] (4) at (-2, 8) {$a(01)$};
		\node [style=white circle] (5) at (2, 8) {$a(10)$};
		\node [style=white circle] (6) at (7, 8) {$a(11)$};
		\node [style=none, scale=2] (7) at (-4.5, 10.75) {$\vdots$};
		\node [style=none, scale=2] (8) at (4.5, 10.75) {$\vdots$};
		\node [style=none, scale=1] (9) at (-10, 4) {$a = $};
	\end{pgfonlayer}
	\begin{pgfonlayer}{edgelayer}
		\draw (0) to (1);
		\draw (1) to (3);
		\draw (1) to (4);
		\draw (0) to (2);
		\draw (2) to (5);
		\draw (2) to (6);
	\end{pgfonlayer}
\end{tikzpicture}
\]
The isomorphism $R: K \to K^2$ is then given by deleting a caret at the bottom the diagram of an element of $K$. I.e., for each $a \in K$, the diagram of $R(a) = (R_0(a), R_1(a))$ is
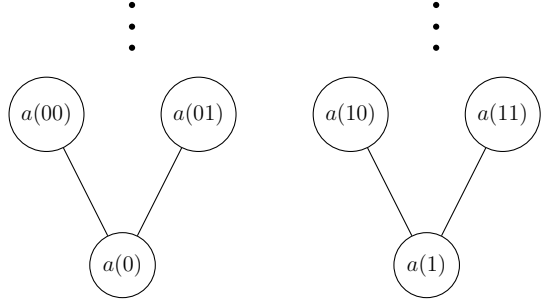
\[
\begin{tikzpicture}
	\begin{pgfonlayer}{nodelayer}
		\node [style=white circle] (12) at (16.75, 2.25) {$a(0)$};
		\node [style=white circle] (13) at (24.75, 2.25) {$a(1)$};
		\node [style=white circle] (14) at (14.75, 6.25) {$a(00)$};
		\node [style=white circle] (15) at (18.75, 6.25) {$a(01)$};
		\node [style=white circle] (16) at (22.75, 6.25) {$a(10)$};
		\node [style=white circle] (17) at (26.75, 6.25) {$a(11)$};
		\node [style=none, scale=2] (18) at (17, 9) {$\vdots$};
		\node [style=none, scale=2] (19) at (25, 9) {$\vdots$};
		\node [style=none] (20) at (28.25, 4) {.};
	\end{pgfonlayer}
	\begin{pgfonlayer}{edgelayer}
		\draw (12) to (14);
		\draw (12) to (15);
		\draw (13) to (16);
		\draw (13) to (17);
	\end{pgfonlayer}
\end{tikzpicture}
\]
Similarly, for each $n \geq 1,$ we identify the elements of $K^n$ with diagrams consisting of $n$ decorations of $t_\infty$ corresponding to elements of $K$ arranged next to each other. If $t \in \mathfrak{T}$ is a tree, since $R = \pi(\caret)$ and $\pi$ is a monoidal functor, we have that $\pi(t) : K \to K^{\ell(t)}$ is given by removing a copy of $t$ from the bottom of a diagram in $K.$ For example, if 
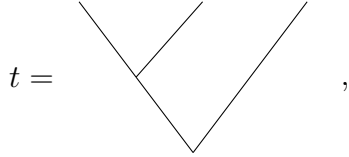
\[
\begin{tikzpicture}
	\begin{pgfonlayer}{nodelayer}
		\node [style=none] (0) at (0, 0) {};
		\node [style=none] (1) at (-1.5, 2) {};
		\node [style=none] (2) at (3, 4) {};
		\node [style=none] (3) at (-3, 4) {};
		\node [style=none] (4) at (0.25, 4) {};
		\node [style=none] (5) at (-4.25, 2) {$t = $};
		\node [style=none] (6) at (4, 1.75) {,};
	\end{pgfonlayer}
	\begin{pgfonlayer}{edgelayer}
		\draw (0.center) to (1.center);
		\draw (1.center) to (3.center);
		\draw (1.center) to (4.center);
		\draw (0.center) to (2.center);
	\end{pgfonlayer}
\end{tikzpicture}
\]
then 
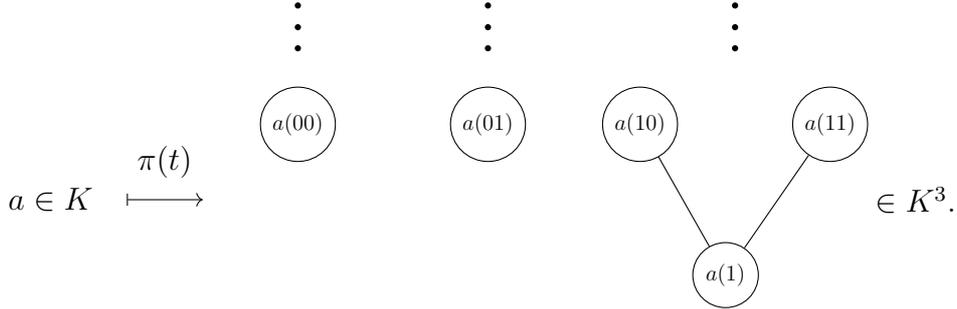
\[
\begin{tikzpicture}
	\begin{pgfonlayer}{nodelayer}
		\node [style=white circle] (2) at (4.25, 4) {$a(1)$};
		\node [style=white circle] (3) at (-7, 8) {$a(00)$};
		\node [style=white circle] (4) at (-2, 8) {$a(01)$};
		\node [style=white circle] (5) at (2, 8) {$a(10)$};
		\node [style=white circle] (6) at (7, 8) {$a(11)$};
		\node [style=none, scale=2] (7) at (-7, 11) {$\vdots$};
		\node [style=none, scale=2] (8) at (4.5, 11) {$\vdots$};
		\node [style=none] (9) at (-13.5, 6) {$a \in K$};
		\node [style=none] (10) at (-11.5, 6) {};
		\node [style=none] (11) at (-9.5, 6) {};
		\node [style=none, scale=2] (12) at (-2, 11) {$\vdots$};
		\node [style=none] (13) at (9.25, 6) {$\in K^3$.};
		\node [style=none] (14) at (-10.5, 7) {$\pi(t)$};
	\end{pgfonlayer}
	\begin{pgfonlayer}{edgelayer}
		\draw (2) to (5);
		\draw (2) to (6);
		\draw [{|->}] (10.center) to (11.center);
	\end{pgfonlayer}
\end{tikzpicture}
\]
In general, if $a \in K$, $t$ is a tree and $l_1,\dots,l_n$ are the leaves of $t$, we will have that 
\[
\pi(t)(a) = (a(l_1),\dots,a(l_n)).
\]
If $\sigma \in S_n$ for some $n \in \NN,$ then $\pi(\sigma)$ permutes the positions of elements of tuples $(a_i) \in K^n,$ so that $\pi(\sigma)(a_i) = (a_{\sigma^{-1}(i)}).$ This results in a formula for the action $V \curvearrowright K$.
\begin{lemma} \label{lem:Kactionformula}
Suppose that $v = s^{-1} \sigma t \in V,$ and that $a \in K.$ Then \[\pi(v)(a)(l_s^i u) = a(l_t^{\sigma^{-1}(i)}u), \ u \in \{0,1\}^*, \ 1 \leq i \leq \ell(t). \]
\end{lemma}

Let $\Phi: \frc{\cat{SF}} \to \cat{Set}$ be the functor of Example \ref{eg:vcantoraction}, where we recall that the action $V \curvearrowright \cantor$ is obtained by restricting $\Phi$ to $V$. If $a_1,\dots,a_n \in K,$ we define \[\supp(a_1,\dots,a_n) := \bigsqcup_{i=1}^n \supp(a_i). \]
\begin{lemma}\label{lem:functorsupp}
    Suppose that $v: m \to n$ is a morphism of $\frc{\cat{SF}},$ and that $a \in K^m.$ Then \[\supp(\pi(v)(a)) = \Phi(v)(\supp(a)).\]
\end{lemma}
\begin{proof}
     The fact that \[\supp(\pi(t)(a)) = \Phi(t)(\supp(a)) \text{ for all } t \in \mathfrak{T}, \ a \in K\] and \[\supp(\pi(\sigma)(a_1,\dots,a_n)) = \Phi(\sigma)(\supp(a_1,\dots,a_n)), \ \sigma \in S_n, \ a_1,\dots,a_n \in K\] follows from the visual description of the functor $\Phi$ given in Example \ref{eg:vcantoraction}, and the visual description of $\pi$ given above. Since the functors $\Phi$ and $\pi$ are monoidal and every forest $f: m \to n$ in $\cat{F}$ can be written as a tensor product of trees, we have that \[\supp(\pi(f)(a)) = \Phi(f)(\supp(a)), \ a \in K^{m}, \tag{$*$}\] and \[\supp(\pi(f)^{-1}(a)) = \Phi(f)^{-1}(\supp(a)), \ a \in K^{\ell(f)}. \tag{$**$}\] The proof is completed by combining $(*)$ and $(**)$ with the fact that every morphism of $\frc{\cat{SF}}$ can be written in the form $f^{-1} \sigma g$ with $f,g$ forests and $\sigma$ a permutation.
\end{proof}
\begin{lemma}\label{lem:suppproperties}
For all $a,b \in K$ and $v \in V$ we have the following:
\begin{enumerate}
    \item[(1)] $a = e_K$ if and only if $\supp(a) = \emptyset$,
    \item[(2)] $\supp(a) = \supp(a^{-1})$,
    \item[(3)] $\supp(\ad_{a}(b)) = \supp(b)$,
    \item[(4)] $\supp(ab) \subseteq \supp(a) \cup \supp(b)$,
    \item[(5)] If $\supp(a)$ and $\supp(b)$ are disjoint, then $ab = ba$,
    \item[(6)] $\supp(\pi(v)(a)) = v(\supp(a))$.
\end{enumerate}
\end{lemma}
\begin{proof}
     If $a \in K$ is non-trivial, then $R_i(a) \neq e_\Gamma$ for some $i \in \{0,1\},$ since $R$ is an isomorphism. But then $R_{ij}(a) = R_j \circ R_i (a) \neq e_\Gamma$ for some $j \in \{0,1\}.$ Continuing this process indefinitely produces an element $x = x_1 x_2 \cdots \in \cantor$ with the property that $R_u(a) \neq e_K$ for all finite truncations $u = x_1 \cdots x_m$ of $x,$ and so $x \in \supp(a).$ Of course if $a = e_K$ then $\supp(a) = \emptyset,$ so we have proven $(1)$.
     
     $(2)$ through $(4)$ follow immediately from Definition \ref{def:gensupp}. If $a,b \in K,$ then $\supp([a,b]) \subseteq \supp(a) \cap \supp(b)$ from $(2),$ $(3)$ and $(4),$ so $(5)$ then follows immediately from $(1).$ $(6)$ is a special case of Lemma \ref{lem:functorsupp}.
\end{proof}
\begin{remark}\label{rmk:covariantmonoidalexample}
    Jones' technology provides many examples of semidirect products $K \rtimes V$ arising from group isomorphisms from a group $K$ to its direct square $K^2$. In \cite{jonestechI, jonestechII}, each group morphism of the form $\alpha: \Gamma \to \Gamma^2, g \mapsto (\alpha_0(g), \alpha_1(g))$ is assigned a group $K(\alpha)$ and an action $ \rho: V \curvearrowright K(\alpha)$. The elements of $K(\alpha)$ are equivalence classes $[t,g]$ of pairs $(t,g)$ with $t \in \mathfrak{T}$ a binary tree, and $g \in \Gamma^{\ell(t)}$ a tuple of $\ell(t)$ elements of $\Gamma$. Such a pair $(t,g)$ can be interpreted as a diagram consisting of the elements of the tuple $g$ decorating the leaves of $t$, for example
    
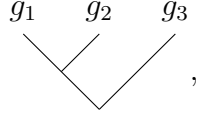
\[
    \begin{tikzpicture}
	\begin{pgfonlayer}{nodelayer}
		\node [style=none, label={above:{$g_1$}}] (0) at (-1, 1) {};
		\node [style=none] (1) at (0, 0) {};
		\node [style=none] (2) at (1, -1) {};
		\node [style=none, label={above:{$g_2$}}] (3) at (1, 1) {};
		\node [style=none, label={above:{$g_3$}}] (4) at (3, 1) {};
		\node [style=none] (5) at (3.5, -0.25) {,};
	\end{pgfonlayer}
	\begin{pgfonlayer}{edgelayer}
		\draw (0.center) to (2.center);
		\draw (1.center) to (3.center);
		\draw (2.center) to (4.center);
	\end{pgfonlayer}
\end{tikzpicture}

    \] where each $g_i \in \Gamma$. The equivalence relation on the pairs $(t,g)$ with $t$ a binary tree and $g \in \Gamma^{\ell(t)}$ is generated by $(t,g) \sim (t', g')$, where $t'$ is a tree obtained by placing a caret on top of $t$, say on the $k^{\text{th}}$ leaf, and $g'$ is obtained by bifurcating the $k^{\text{th}}$ entry of $g$ into the pair $(\alpha_0(g_k), \alpha_1(g_k))$. We may thus interpret elements $[t,g] \in K(\alpha)$ as equivalence classes of diagrams, for example having the equality
    
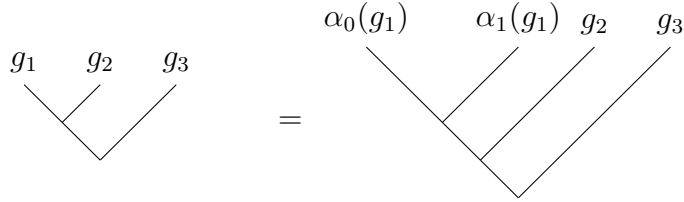
\[
    \begin{tikzpicture}
	\begin{pgfonlayer}{nodelayer}
		\node [style=none, label={above:{$g_1$}}] (0) at (-7, 1) {};
		\node [style=none] (1) at (-6, 0) {};
		\node [style=none] (2) at (-5, -1) {};
		\node [style=none, label={above:{$g_2$}}] (3) at (-5, 1) {};
		\node [style=none, label={above:{$g_3$}}] (4) at (-3, 1) {};
		\node [style=none] (5) at (0, 0) {$=$};
		\node [style=none, label={above:{$\alpha_0(g_1)$}}] (6) at (2, 2) {};
		\node [style=none] (7) at (4, 0) {};
		\node [style=none] (8) at (5, -1) {};
		\node [style=none, label={above:{$\alpha_1(g_1)$}}] (9) at (6, 2) {};
		\node [style=none, label={above:{$g_2$}}] (10) at (8, 2) {};
		\node [style=none] (11) at (6, -2) {};
		\node [style=none, label={above:{$g_3$}}] (12) at (10, 2) {};
	\end{pgfonlayer}
	\begin{pgfonlayer}{edgelayer}
		\draw (0.center) to (2.center);
		\draw (1.center) to (3.center);
		\draw (2.center) to (4.center);
		\draw (6.center) to (8.center);
		\draw (7.center) to (9.center);
		\draw (8.center) to (10.center);
		\draw (8.center) to (11.center);
		\draw (11.center) to (12.center);
	\end{pgfonlayer}
\end{tikzpicture}

    \] as elements of $K(\alpha)$, where the $g_i$ are elements of $\Gamma$. If $[t,g], [s,h] \in K(\alpha)$, then the product $[t,g] \cdot [s,h] = [t,gh]$ is given by multiplying $g$ and $h$ pointwise, after ensuring that $t = s$ using the equivalence relation $\sim$.
    
    We then obtain a group isomorphism $R(\alpha): K(\alpha) \to K(\alpha)^2$ given by removing a caret at the bottom of diagrams in $K(\alpha)$. For example, $R(\alpha)$ maps
        
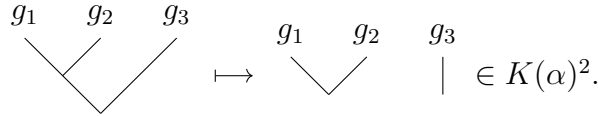
\[
        \begin{tikzpicture}
	\begin{pgfonlayer}{nodelayer}
		\node [style=none, label={above:{$g_1$}}] (0) at (-7, 1) {};
		\node [style=none] (1) at (-6, 0) {};
		\node [style=none] (2) at (-5, -1) {};
		\node [style=none, label={above:{$g_2$}}] (3) at (-5, 1) {};
		\node [style=none, label={above:{$g_3$}}] (4) at (-3, 1) {};
		\node [style=none] (5) at (-2, 0) {};
		\node [style=none] (6) at (-1, 0) {};
		\node [style=none, label={above:{$g_1$}}] (7) at (0, 0.5) {};
		\node [style=none] (8) at (1, -0.5) {};
		\node [style=none, label={above:{$g_2$}}] (10) at (2, 0.5) {};
		\node [style=none, label={above:{$g_3$}}] (11) at (4, 0.5) {};
		\node [style=none] (12) at (4, -0.5) {};
		\node [style=none, label={right:{$\in K(\alpha)^2$.}}] (13) at (4.5, 0) {};
	\end{pgfonlayer}
	\begin{pgfonlayer}{edgelayer}
		\draw (0.center) to (2.center);
		\draw (1.center) to (3.center);
		\draw (2.center) to (4.center);
		\draw [{|->}] (5.center) to (6.center);
		\draw (8.center) to (10.center);
		\draw (8.center) to (7.center);
		\draw (11.center) to (12.center);
	\end{pgfonlayer}
\end{tikzpicture}

        \]
     The group morphism $R(\alpha)$ along with the inclusion $\Gamma \to K(\alpha), g \mapsto [ \ | \ ,g]$ can be shown to yield a Jones isomorphism for $\alpha$ (Definition \ref{def:jonesiso}).
     
     We obtain a functor $\pi(\alpha) : \frc{\cat{SF}} \to \cat{Grp}$ defined by $\pi(\alpha)(\caret) := R(\alpha)$, which restricts to an action $V \curvearrowright K(\alpha)$. This is precisely the action $\rho: V \curvearrowright K(\alpha)$ defined in \cite{jonestechI, jonestechII}. We denote the resulting semidirect product by $G(\alpha) := K(\alpha) \rtimes V.$ It was shown in \cite{jonestechI} that the groups $G(\alpha)$ are precisely those considered by Tanushevski in \cite{tan16}, in which the lattice of normal subgroup s of $K(\alpha) \rtimes F \subseteq G(\alpha)$ was described. These groups $G(\alpha)$ are also examples of cloning systems, introduced by Witzel-Zaremsky in \cite{wz}, implying that they have good finiteness properties. The focus of \cite{jonestechI, jonestechII} and the present article is on the structure of isomorphisms between these groups.
     
     If $\alpha_0 \in \aut(\Gamma)$ and $\alpha_1(g) = e_\Gamma$, the resulting group $G(\alpha)$ is isomorphic to a restricted twisted permutational wreath product $\oplus_{\QQ_2} \Gamma \rtimes V$ \cite{jonestechI}. In \cite{jonestechI}, every isomorphism between two of these wreath products was shown to be \textit{spatial}, i.e. support-preserving up to a homeomorphism $\varphi$ of the Cantor space $\cantor$, leading to a complete classification of these wreath products up to isomorphism in terms of the input data $(\Gamma, \alpha_0)$.
     
     If $\alpha_0, \alpha_1 \in \aut(\Gamma)$, then $G(\alpha)$ is a semidirect product $L \Gamma \rtimes V$, where $L\Gamma$ denotes the set of continuous maps $\cantor \to \Gamma$, the latter equipped with the discrete topology \cite{jonestechII}. The action $V \curvearrowright L\Gamma$ is induced by the action $V \curvearrowright \cantor$, and is twisted using the automorphisms $\alpha_0, \alpha_1$. The isomorphisms $\theta: G \to \wt{G}$ between these semidirect products $G = L \Gamma \rtimes V$, $\wt{G} = L \wt{\Gamma} \rtimes V$ are also spatial, up to a multiple of a group morphism $\zeta: G \to Z \wt{G}$ \cite{jonestechII}. 
     
     In Theorem \ref{thm:isodecomp}, we extend these results and show that any isomorphism $\theta: G \to \wt{G}$ between any groups $G = K \rtimes V$, $\wt{G} = \wt{K} \rtimes V$ coming from isomorphisms $K \xrightarrow{\sim} K^2$, $\wt{K} \xrightarrow{\sim} \wt{K}^2$ is spatial, up to a multiple of a group morphism $\zeta: G \to Z \wt{G}$, using the notion of support in Definition \ref{def:gensupp}. This is indeed a generalisation; the class of groups $G(\alpha)$ coming from group morphisms $\alpha: \Gamma \to \Gamma^2, g \mapsto (\alpha_0(g), \alpha_1(g))$ with $\alpha_0 \in \aut(\Gamma)$ and $\alpha_1 \in \aut(\Gamma)$ or $\alpha_1 : g \mapsto e_\Gamma$ does not recover the full class of groups $G(R) = K \rtimes V$, with $R: K \xrightarrow{\sim} K^2$. For such $\alpha$, the group $K(\alpha)$ is never finitely generated, however there exist isomorphisms $\wt{R}: \wt{K} \to \wt{K}^2$ with $\wt{K}$ finitely generated \cite{directhopf, fgk, fgp}. An isomorphism $G(\alpha) \cong \wt{K} \rtimes V$ would restrict to an isomorphism $K(\alpha) \cong \wt{K}$ by Theorem \ref{thm:characteristic}, a contradiction.
     
     In \cite[Section 1.2.2]{jonestechII}, Brothier introduced a notion of support for elements $a \in K(\alpha)$ ($\alpha$ can be arbitrary) which we will denote here by $\wt{\supp}(a)$. Given a tree $t \in \mathfrak{T}$ and a tuple $g = (g_l)_{l \in \Leaf(t)},$ we may define a map $\kappa_t : \cantor \to \Gamma$ mapping $l x$ to $g_l$ for each leaf $l \in \Leaf(t)$ and $x \in \cantor.$ Denote by $\supp(\kappa_t(g))$ the support of the map $\kappa_t: \cantor \to \Gamma.$ For each $a = [t,g] \in K(\alpha)$, \[\wt{\supp}(a) := \bigcap_{s,h} \supp(\kappa_s(h)),\] where the intersection runs over all pairs $(s,h)$ with $s$ a tree and $h : \Leaf(s) \to \Gamma$ a map such that $a = [s,h].$ One may then wonder if $\wt{\supp}(a)$ has anything to do with Definition \ref{def:gensupp}. Indeed, one can verify that $\wt{\supp}(a) = \supp_{R(\alpha)}(a)$.
\end{remark}
\begin{definition}
Let $\mathcal{I}$ be the set of all finite unions of sdis. For each $I \in \mathcal{I}$ and subgroup $H \leq K,$ we let \[H_I := \{a \in H: \ | \ \supp(a) \subseteq I\}\] and \[H^R := \{a \in H : R_0(a) = R_1(a) = a\}.\] We refer to elements of $K^R$ as $R$-invariant. We also let
\[ \fixv(I) := \{v \in V \ | \ v(x) = x \text{ for all } x \in I\} \] and \[ \stabv(I) := \{ v \in V \ | \ v(I) = I\}.\]
\end{definition}
\begin{lemma}\label{lem:aIwelldef}
    Let $a \in K$ and $ \ u \in \{0,1\}^*$. If $t$ is a tree with $u \in \Leaf(t)$, define an element $a_{u,t} \in K^{\ell(t)}$ by $a_{u,t}(u) := a$, and $ a_{u,t}(l) = e_K$ for all $l \in \Leaf(t)$ not equal to $u$. Then the element \[\pi(t)^{-1}(a_{u,t}) \in K\] is independent of the choice of $t$, and is an element of $K_I.$
\end{lemma} 
\begin{proof}
   Suppose that $t$ is a tree having $u$ as a leaf. The fact that $\supp(\pi(t)^{-1}(a_{u,t})) \subseteq I$ follows immediately from Lemma \ref{lem:functorsupp}. Let $f := f_{k,\ell(t)}$ be an elementary forest with $k\neq i,$ where $l_t^i = u.$ Thus, $u$ is a leaf of $ft$, and
    \begin{align*}
        \pi(ft)\pi(t)^{-1}(a_{u,t}) = \pi(f)(a_{u,t}).
    \end{align*} Since $\pi(f)(a_{u,t})(u) = a$ and $\pi(f)(a_{u,t})(l) = e_K$ for all leaves $l$ of $ft$ not equal to $u,$ so we have that \[\pi(f)(a_{u,t}) = a_{ft},\] yielding that $\pi(ft)^{-1}(a_{ft}) = \pi(t)^{-1}(a_{u,t}).$ Thus, for every forest $f$ with $\ell(t)$ roots not supported at the leaf $u$ of $t,$ we have that $\pi(ft)^{-1}(a_{ft}) = \pi(t)^{-1}(a_{u,t}).$

     We finish with the observation that for each $u \in \{0,1\}^*,$ there is a smallest tree $t_u$ among trees $t$ having $u$ as a leaf. More precisely, if $t$ is a tree having $u$ as a leaf, then there is a forest $f$ not supported at the leaf $u$ of $t_u$ such that $t = f t_u.$
\end{proof} 
Lemma \ref{lem:aIwelldef} allows us to make the following definition.
\begin{definition}
    Suppose that $I = \cantor_u$ is an sdi, and that $a \in K.$ Using the notation of Lemma \ref{lem:aIwelldef}, we define \[a_I := \pi(t)^{-1}(a_{u,t})\] where $t$ is any tree having $u$ as a leaf. Furthermore, define \[a^{(u)} := R_u(a)_{I}.\]
\end{definition}
\begin{example}
    Suppose that $a \in K,$ and let $I := \cantor_{01}$. The diagram of $a_I$ is
    
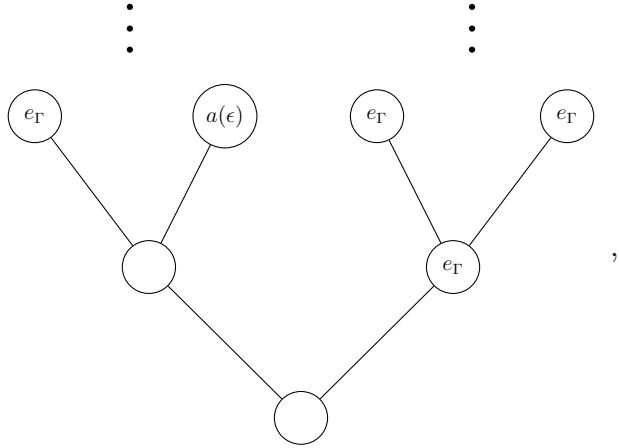
\[
     \begin{tikzpicture}
	\begin{pgfonlayer}{nodelayer}
		\node [style=white circle] (0) at (7.5, 0) {};
		\node [style=white circle] (1) at (3.5, 4) {};
		\node [style=white circle] (2) at (11.5, 4) {$e_\Gamma$};
		\node [style=white circle] (3) at (0.5, 8) {$e_\Gamma$};
		\node [style=white circle] (4) at (5.5, 8) {$a(\epsilon)$};
		\node [style=white circle] (5) at (9.5, 8) {$e_\Gamma$};
		\node [style=white circle] (6) at (14.5, 8) {$e_\Gamma$};
		\node [style=none, scale=2] (7) at (3, 10.75) {$\vdots$};
		\node [style=none, scale=2] (8) at (12, 10.75) {$\vdots$};
		\node [style=none] (9) at (15.75, 4.25) {,};
	\end{pgfonlayer}
	\begin{pgfonlayer}{edgelayer}
		\draw (0) to (1);
		\draw (1) to (3);
		\draw (1) to (4);
		\draw (0) to (2);
		\draw (2) to (5);
		\draw (2) to (6);
	\end{pgfonlayer}
\end{tikzpicture}
    \]
    and the diagram of $a^{(01)}$ is
    
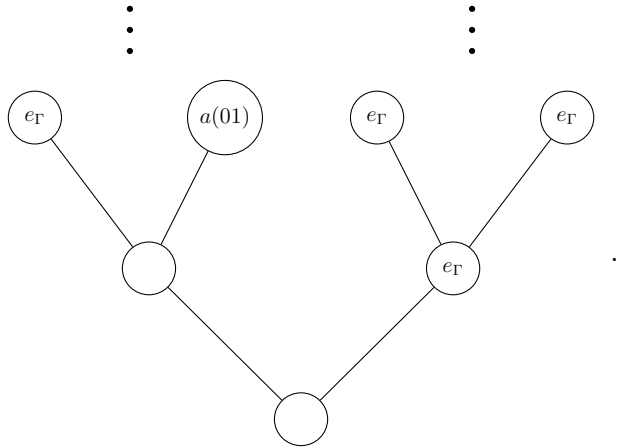
\[
    \begin{tikzpicture}
	\begin{pgfonlayer}{nodelayer}
		\node [style=white circle] (0) at (0, 0) {};
		\node [style=white circle] (1) at (-4, 4) {};
		\node [style=white circle] (2) at (4, 4) {$e_\Gamma$};
		\node [style=white circle] (3) at (-7, 8) {$e_\Gamma$};
		\node [style=white circle] (4) at (-2, 8) {$a(01)$};
		\node [style=white circle] (5) at (2, 8) {$e_\Gamma$};
		\node [style=white circle] (6) at (7, 8) {$e_\Gamma$};
		\node [style=none, scale=2] (7) at (-4.5, 10.75) {$\vdots$};
		\node [style=none, scale=2] (8) at (4.5, 10.75) {$\vdots$};
		\node [style=none] (9) at (8.25, 4.25) {.};
	\end{pgfonlayer}
	\begin{pgfonlayer}{edgelayer}
		\draw (0) to (1);
		\draw (1) to (3);
		\draw (1) to (4);
		\draw (0) to (2);
		\draw (2) to (5);
		\draw (2) to (6);
	\end{pgfonlayer}
\end{tikzpicture}
    \] Here we have left some vertices unlabelled for convenience.
\end{example}
\begin{lemma}\label{lem:vaI}
    Suppose that $I = \cantor_u$ and $J = \cantor_m$ are sdis, and that $v \in V$ satisfies $v(ux) = mx$ for all $x \in \cantor.$ Then $v(a_I) = a_J.$
\end{lemma}
\begin{proof}
    The hypothesis $v(ux) = mx$ for all $x \in \cantor$ is equivalent to the fact that there exist trees $t,s \in \mathfrak{T}$ and a permutation $\sigma$ such that $u = l_t^i,$ $m = l_s^{\sigma(i)},$ and $v = s^{-1} \sigma t.$ Thus, using the notation of Lemma \ref{lem:aIwelldef},
    \begin{align*}
        \pi(v)(a_I) &= \pi(s^{-1}\sigma t)(\pi(t^{-1})(a_{u,t})) \\ &= \pi(s^{-1} \sigma) (a_{u,t}) \\
        &= \pi(s^{-1})(a_s) \\
        &= a_J.
    \end{align*}
\end{proof}
\begin{lemma}\label{lem:aIhalves}
    If $I$ is an sdi, then $a_I = R_{0}(a)_{I_0} \cdot R_1(a)_{I_1}.$
\end{lemma}
\begin{proof}
Let $t$ be a tree having $u$ as a leaf, and let $i \in \{1,\dots,\ell(t)\}$ be such that $u = l_t^i.$ Let $f:= f_{i,\ell(t)}.$ Using the notation of Lemma \ref{lem:aIwelldef}, we have that 
\begin{align*}
    a_I &= \pi(t)^{-1}(a_{u,t}) \\
        &= \pi(ft)^{-1} \pi(f) (a_{u,t}) \\
        &= \pi(ft)^{-1} (R_0(a)_{u0, ft} \cdot R_1(a)_{u1,ft}) \\
        &= R_{0}(a)_{I_0} \cdot R_1(a)_{I_1}.
\end{align*}
\end{proof}
\begin{lemma}\label{lem:suppdecomp}
    If $a \in K$ and $t \in \mathfrak{T}$ is a tree, then \[a = \prod_{l \in \Leaf(t)} a^{(l)}, \tag{$*$}\] and the product can be taken in any order. In particular, if $(I^j)_{j=1}^n$ is an sdp, then the element $a \in K$ can be written uniquely as a product of elements $a_1,\dots, a_n \in K$ with $\supp(a_j) \subseteq I^j$ for all $1 \leq j \leq n.$
\end{lemma}
\begin{proof}
   We start by noting that the product in $(*)$ is well-defined and invariant under permuting terms since the $a^{(l)}$ have disjoint supports. Let $a \in K,$ and suppose that $u \in \{0,1\}^*.$ Then 
    \begin{align*}
        a^{(u0)}\cdot a^{(u1)} &:= R_{u0}(a)_{\cantor_{u0}} \cdot R_{u1}(a)_{\cantor_{u1}} \\ &= R_0\left(R_u(a)\right)_{\cantor_{u0}} \cdot R_1\left(R_u(a)\right)_{\cantor_{u1}} \\ &= R_u(a)_{\cantor_u} \\ &=: a^{(u)},
    \end{align*} where the second last line follows from Lemma \ref{lem:aIhalves}.
    
    Clearly $a = \prod_{l \in \Leaf(t)} a^{(l)}$ if $t = |$ is the trivial tree with one leaf. Suppose that $(*)$ holds for all trees $t$ with $n$ leaves. If $t$ is a tree with $n+1$ leaves, then $t = f_{k,n} s$ for some tree $s$ with $n$ leaves, and $1 \leq k \leq n.$ Letting $f := f_{k,n}$ we have that 
    \begin{align*}
        \prod_{l \in \Leaf(t)} a^{(l)} &=  \prod_{l \in \Leaf(fs) }a^{(l)} \\ &=  \left(\prod_{l \in \Leaf(s) \setminus \{l_s^k\}}  a^{(l)} \right) \cdot a^{(l_s^k \cdot 0)} \cdot a^{(l_s^k \cdot 1)} \\ &= \prod_{l \in \Leaf(s)} a^{(l)} \\ &= a.
    \end{align*} Thus, $(*)$ holds for every $a \in K$ and $t \in \mathfrak{T}$ by induction. 
    
    The remainder of the lemma follows from Lemma \ref{lem:suppproperties}, and the fact that if $a \in K, \
     t \in \mathfrak{T},$ and $l \in \Leaf(t),$ then $\supp(a^{(l)}) \subseteq \cantor_l$.  
\end{proof} 
\begin{corollary}
    If $u \in \{0,1\}^*$ and $a \in K_I,$ where $I := \cantor_u,$ then $a = a^{(u)}.$
\end{corollary}

\subsection{Isomorphism structure}
In this subsection we fix two group isomorphisms $R: K \to K^2, \ \wt{R}: \wt{K} \to \wt{K}^2$ and study the structure of isomorphisms between the resulting semidirect products $G = K \rtimes V, \wt{G} = \wt{K} \rtimes V.$ The following theorem has essentially the same proof as \cite[Theorem 4.1]{fscII}, which concerns isomorphisms between certain restricted permutational wreath products $\Gamma \wr V.$ We reproduce the proof using our terminology.
\begin{theorem}\label{thm:characteristic}
Every isomorphism $\theta: G \to \wt{G}$ restricts to an isomorphism $K \to \wt{K}.$ In particular, the normal subgroup $K \nsg G$ is characteristic.
\end{theorem}
\begin{proof}
For each subset $S \subseteq G,$ let $\nc(S)$ denote the normal closure of $S$ in $G.$ If a normal subgroup $L \nsg G$ can be written as a product $L = L_0 \times L_1,$ with $\nc(L_0) = L = \nc(L_1),$ we say that $L$ has the decomposability property, using a similar definition for normal subgroups $\wt{L} \nsg \wt{G}$. The approach is to show that $K \nsg G$ is the largest normal subgroup of $G$ having the decomposable property.

Firstly, $K$ has the decomposability property since $K = K_0 \times K_1$ via the isomorphism $K \cong K \times K,$ and $K_0$ and $K_1$ are conjugate in $G$ via the element $v_{(12)} := (\caret)^{-1} \circ (12) \circ \caret \in V.$ 

Suppose then that $L = L_0 \times L_1 \nsg G$ has the decomposability property, and assume that $L \subsetneq K,$ in search of a contradiction. Let $q: G \surj G/K = V, av \mapsto v$ be the canonical projection, and note that $q(L) \nsg V,$ since $L \nsg G$ and $q$ is surjective. Since $L \subsetneq K,$ there is a non-trivial $v \in V$ such that $av \in L$ for some $a \in K,$ and so $q(av)$ is a non-trivial element of $q(L).$ Since $V$ is simple, it must be the case that $q(L) = V.$

Let $p: L \surj V$ be the restriction of $q$ to $L,$ and note that $p(L_0), p(L_1) \nsg V.$ For $i \in \{0,1\},$ if $L_i$ were contained in $K,$ then so would its normal closure, implying that $L \subseteq K,$ contradicting our assumption. Thus, both $L_0$ and $L_1$ are not contained in $K,$ yielding that $p(L_0) = p(L_1) = V.$

But then if $v, w \in V,$ we can write $v = p(l_0)$ and $w = p(l_1)$ for some $l_0 \in L_0, \ l_1 \in L_1$, from which we obtain $vw = wv.$ We conclude that $V$ is abelian, obtaining our desired contradiction. Thus, $L \subseteq K.$

The fact that $\theta(K) = \wt{K}$ follows from the fact that $\theta(K)$ and $\wt{K}$ are both the largest normal subgroup of $\wt{G}$ having the decomposability property.

Taking $\wt{K} = K$ yields that $K \nsg G$ is characteristic.
\end{proof}

\begin{lemma}\label{lem:weakiso}
Suppose that $\theta: G \to \wt{G}$ is an isomorphism. Then there is a unique triple $(\kappa, c, \varphi),$ with $\kappa: K \to \wt{K}$ an isomorphism, $c: V \to \wt{K}$ a map, and $\varphi \in N_\cantor(V),$ such that
\[ 
\theta(av) = \kappa(a) \cdot c_v \cdot \ad_{\varphi(v)} \text{ for all } av\in G
\] and \[ c_{vw} = c_v \cdot  \pi(\ad_{\varphi}(v))(c_w) \text{ for all } v,w \in V.\]
\end{lemma}
\begin{proof}
For each $v \in V$, we can write $\theta(v)$ uniquely as a product $\theta(v) = c_v \cdot \phi_v$ with $c_v \in K$ and $\phi_v \in V$. The map $\phi: v \mapsto \phi_v$ then defines an automorphism of $V$, so by Proposition \ref{prop:autv}, there is a unique $\varphi \in N_\cantor(V)$ satisfying $\phi = \ad_{\varphi}: v \mapsto \varphi v \varphi^{-1}$. Furthermore,
\begin{align*}
    \theta(vw) &= \theta(v) \theta(w) = c_{vw} \phi_{vw} \\ &= c_v \phi_v c_w \phi_w \\
    &= c_v \pi(\phi_v)(c_w) \phi_v \phi_w.
\end{align*} Letting $\kappa$ be the restriction of $\theta$ to $K,$ by Theorem \ref{thm:characteristic}, we have proved the existence of the required tuple $(\kappa, c, \varphi)$. Uniqueness is clear.
\end{proof}
\begin{lemma}\label{lem:gencentre}
The centre $ZG$ of $G$ is equal to $(ZK)^R,$ the group of elements of $ZK$ that are $R$-invariant, and similarly $Z\wt{G} = (Z \wt{K})^{\wt{R}}.$
\end{lemma}
\begin{proof}
Suppose that $g = av \in ZG.$ Then for each $w \in V$ we have that \[g w g^{-1} = w,\] and so \[a v w v^{-1} a^{-1} = a \cdot \pi(vwv^{-1})(a^{-1}) \cdot vwv^{-1} = w,\] and so $vwv^{-1} = w.$ This implies that $v \in ZV,$ and since $V$ is a simple non-abelian group we get that $v = e.$

Thus, $ g = a \in ZK,$ so it remains to show that $a$ is $R$-invariant. Let $u$ be a non-empty word in $0$ and $1$ and let $v \in V$ be such that $v(0x) = ux$ for all $x \in \cantor,$ so that by Lemma \ref{lem:Kactionformula} we have that \[\pi(v)(a)(0) = a(u).\] But since $a \in ZG$ we have that $\pi(v)(a) = a,$ and thus $a(u) = a(0).$ Notice then that $a(0)$ is $R$-invariant, since
\begin{align*}
    a(0) &= R^{-1}(a(00), a(01)) \\
         &= R^{-1}(a(0), a(0)).
\end{align*} We finish by noting that
\begin{align*}
    a &= a(\epsilon) \\
      &= R^{-1}(a(0), a(1)) \\
      &= R^{-1}(a(00), a(01)) \\
      &= a(0).
\end{align*}

Conversely, if $a \in (ZK)^R,$ then $\pi(v)(a) = a$ for all $v \in V$ by Lemma \ref{lem:Kactionformula}. Since $a$ commutes with every $b \in K,$ we have that $a \in ZG.$
\end{proof}
\begin{lemma}\label{lem:thetasupp}
Let $\theta: G \to \wt{G}, av \mapsto \kappa(a) c_v \ad_{\varphi}(v)$ be an isomorphism, and $I \in \mathcal{I}$ be a finite union of sdis. Then \[ \theta(K_I) \subseteq \wt{K}_{\varphi(I)} \cdot Z \wt{G}.  \]
\end{lemma}
\begin{proof}
We start with the case in which $I$ is an sdi. The case $I = \cantor$ is clear. Suppose then that $I \subsetneq \cantor.$ Let $J \subseteq \varphi(\cantor \setminus I)$ be an sdi, and suppose that $a \in K_I, \ \wt{b} \in \wt{K}_{J}.$

Write $\wt{b} = \wt{b}_0 \cdot \wt{b}_1$ where each $\supp(\wt{b}_i) \subseteq J_i,$ and let $d$ be the element of $K$ satisfying $[\theta(a), \wt{b}_0] = \theta(d).$ We then have that $\supp(d) \subseteq I$ and $\supp(\theta(d)) \subseteq J_0.$ Since $\varphi(I)$ and $J$ are disjoint, we may consider an element $v \in \fixv(\varphi(I))$ interchanging $J_0$ and $J_1.$ Letting $\phi := \ad_{\varphi},$ we have that $\phi^{-1} v \in \fixv(I),$ and so $\pi(\phi^{-1} v)(d) = d.$ Thus,
\begin{align*}
    \theta(d) &= \theta(\pi(\phi^{-1} v)(d)) \\ &= \ad(c_{\phi^{-1}v})(\wt{\pi}(v)(\theta(d))),
\end{align*} yielding that 
\[
\supp(\theta(d)) = v(\supp(\theta(d)).
\] Since $\supp(\theta(d)) \subseteq J_0$ and $v(\supp(\theta(d)) \subseteq J_1,$ we obtain that $\supp(\theta(d)) = \emptyset,$ allowing us to conclude that $[\theta(a), \wt{b}_0] = e.$ An identical argument shows that $[\theta(a), \wt{b}_1] = e,$ and we immediately obtain that $[\theta(a), \wt{b}] = e.$

Since $\varphi(I)^c$ is a finite union of sdis, each $\wt{b} \in \wt{K}_{\varphi(I^c)} $ can be written as a finite product $\wt{b}^1 \cdots \wt{b}^m$ with each $\wt{b}^i$ supported in an sdi $J^i \subseteq \varphi(I^c),$ and so $[\theta(a), \wt{b}] = e_\Gamma.$ 

This shows that $\theta(a)(u) \in Z \wt{K}$ for all words $u \in \{0,1\}^*$ corresponding to an sdi contained in $\varphi(I^c).$ Let $J, J'$ be sdis contained in $\varphi(I)^c$ with corresponding words $u$ and $u'$, and let $v \in \fixv(\varphi(I))$ with $v(J) = J'.$ Then 
\begin{align*}
    \theta(a) &= \theta(\pi(\phi^{-1}v)(a))\\
              &= \ad(c_{\phi^{-1}v})\left(\wt{\pi}(v)(\theta(a))\right).
\end{align*} Evaluating the above at the word $u',$ since $\theta(a)(u') \in Z \wt{K},$ we have that \[\theta(a)(u') = \wt{\pi}(v)(\theta(a))(u') = \theta(a)(u).\] We can now let $\wt{c} := \theta(a)(u)$, where $u \in \{0,1\}^*$ is any word corresponding to an sdi contained in $\varphi(I^c)$. If $u$ is such a word, then
\begin{align*}
    \wt{c} &= \theta(a)(u) \\
           &= \wt{R}^{-1}(\theta(a)(u0), \theta(a)(u1)) \\
           &= \wt{R}^{-1}(\wt{c}, \wt{c}),
\end{align*} showing that $\wt{c}$ is $R$-invariant. We thus have $\wt{c} \in Z\wt{G}$, obtaining the decomposition \[\theta(a) = \left( \theta(a) \cdot \wt{c}^{-1} \right) \cdot \wt{c} \in \wt{K}_{\varphi(I)} \cdot Z \wt{G}. \] Thus, $\theta(K_I) \subseteq \wt{K}_{\varphi(I)} \cdot Z\wt{G}$. 

If $I$ is a finite union of sdis, we can write $I = \bigsqcup_{j = 1}^n I^j,$ with the $I^j$ sdis, so that $K_I = K_{I^1} \times \cdots \times K_{I^n}.$ But then if $a \in K_I,$ we can write  \[\theta(a) = \wt{a}^1 \cdot \wt{b}^1 \cdots \wt{a}^n \cdot \wt{b}^n = \wt{a}^1 \cdots \wt{a}^n \cdot \wt{b}^1 \cdots \wt{b}^n\] with each $\wt{a}^j \in K_{\varphi(I^j)}$ and $\wt{b}^j \in Z \wt{G},$ so that $\theta(a) \in \wt{K}_{\varphi(I)} \cdot Z \wt{G}.$
\end{proof} 
\begin{lemma}\label{lem:centremorphism}
Suppose that $\theta: G \to \wt{G}$ is an isomorphism, and $\zeta : G \to Z \wt{G}$ is a group morphism. Then the map $\theta \cdot \zeta : G \to \wt{G}, g \mapsto \theta(g) \cdot \zeta(g)$ is an isomorphism of groups.

Moreover,
\[ ZG \subseteq G' \subseteq \ker \zeta,\] where $G'$ denotes the derived subgroup $[G,G]$ of $G$.
\end{lemma} 
\begin{proof}
The proof is essentially the same as that of \cite[Proposition 2.7]{jonestechII}. Multiplicativity of $\theta \cdot \zeta$ is clear. We have that
\begin{align*}
    \ker(\theta \cdot \zeta) &= \{ av \in G: \theta(av) = \zeta(av)^{-1} \} \\ &= \{a \in ZG: \theta(a) = \zeta(a)^{-1}\},
\end{align*} which follows from the fact that $\theta$ is an isomorphism and $ZG \subseteq K.$ Therefore, if we can show that $ZG \subseteq \ker \zeta,$ we may conclude that $\theta \cdot \zeta$ is injective.

The inclusion $G' \subseteq \ker \zeta$ follows from the fact that $Z\wt{G}$ is abelian. Let $a \in ZG,$ and $I \subsetneq \cantor$ be an sdi. Let $v \in V$ and $J$ be an sdi disjoint from $I,$ such that $v(J_0) = J$ and $v(J_1) = I,$ so that $v(J) = I \sqcup J.$ Then 
\begin{align*}
    [v, a_J ] &= \pi(v)(a_J) \cdot a_J^{-1} \\ 
              &= \pi(v)(a_J^{(0)} \cdot a_J^{(1)}) \cdot a_J^{-1} \\
              &= R_0(a)_J \cdot R_1(a)_{I} \cdot a_J^{-1} \\
              &= a_J \cdot a_I \cdot a_J^{-1} \\
              &= a_I.
\end{align*} Thus, $a_I \in G'.$ Writing $a = a_{\cantor_0} \cdot a_{\cantor_1},$ we obtain that $a \in G'.$ Thus, $\theta \cdot \zeta$ is injective.

Surjectivity of $\theta \cdot \zeta$ follows from the fact that the image of the element $\theta^{-1}(\wt{g}) \cdot \theta^{-1} \circ \zeta \circ \theta^{-1}(\wt{g}^{-1}) \in G$ under $\theta \cdot \zeta$ is $\wt{g},$ for each $\wt{g} \in \wt{G}.$ Here one needs that $ZG \subseteq \ker \zeta.$
\end{proof}

\begin{theorem}\label{thm:isodecomp}
If $\theta: G \to \wt{G}$ is an isomorphism, then there is a unique tuple $(\kappa^0, \zeta, c, \varphi)$ with $\kappa^0: K \to \wt{K}$ an isomorphism, $\zeta: G \to Z\wt{G}$ a group morphism, $c: V \to \wt{K}$ a map, and $\varphi \in N_\cantor(V)$ a homeomorphism, satisfying the following properties:
\begin{enumerate}
\item[(1)] $\theta(av) = \kappa^0(a) \cdot \zeta(av) \cdot c_v \cdot \ad_{\varphi}(v)$ for all $av \in G.$
    \item [(2)]  $c_{vw} = c_v \cdot \wt{\pi}(\ad_{\varphi}(v))(c_w) $ for all $v,w \in V.$
    \item [(3)] $\supp(\kappa^0(a)) = \varphi(\supp(a))$ for all $a \in K.$
\end{enumerate} 
\end{theorem}
\begin{proof}
Let $\theta: G \to \wt{G}$ be an isomorphism. Using the notation of Lemma \ref{lem:weakiso}, write $\theta(av) = \kappa(a) c_v \ad_{\varphi}(v)$ for all $av \in G.$

For each $I \in \mathcal{I}$ properly contained in $\cantor,$ the subgroups $\wt{K}_{\varphi(I)} , Z \wt{G} \leq \wt{K}$ intersect trivially and commute, so by Lemma \ref{lem:thetasupp}, there are unique group morphisms $\kappa^I: K_I \to \wt{K}_{\varphi(I)}$ and $\zeta^I: K_I \to Z \wt{G}$ satisfying $\kappa(a) = \kappa^I(a) \cdot \zeta^I(a)$ for all $a \in K_I.$ 

If $u$ is a non-empty word in $0$ and $1,$ then for each $a \in K_{\cantor_u},$ we have that
\begin{align*}
    \kappa(a) &= \kappa^{\cantor_u}(a) \cdot \zeta^{\cantor_u}(a) = \kappa(a^{(u0)}) \kappa(a^{(u1)}) \\
              &= \kappa^{\cantor_{u0}}(a^{(u0)}) \cdot \kappa^{\cantor_{u1}}(a^{(u1)}) \cdot \zeta^{\cantor_{u0}}(a^{(u0)}) \cdot \zeta^{\cantor_{u1}}(a^{(u1)}).
\end{align*} Here we have written $\kappa(a)$ as an element of $\wt{K}_{\varphi(I)} \cdot Z \wt{G}$ in two different ways, and so \[\zeta^{\cantor_u}(a) = \zeta^{\cantor_u}(a^{(u)}) = \zeta^{\cantor_u}(a^{(u0)}) \cdot \zeta^{\cantor_u}(a^{(u1)}).\] This implies that for each tree $t \in \mathfrak{T}$ with $|\Leaf(t)| \geq 2,$ we have that the map
\[
\zeta^t(a) := \prod_{l \in \Leaf(t)} \zeta^{\cantor_l}(a^{(l)}), \ a \in K
\] is independent of the choice of $t,$
and we denote the map $\zeta^t$ by $\zeta^\cantor.$

\textbf{Claim 1:} If $I \in \mathcal{I}$ is properly contained in $\cantor,$ then for all $a \in K_I$ we have that $\zeta^I(a) = \zeta^\cantor(a).$

If $I \subsetneq \cantor$ is a finite union of sdis, we can find sdis $I^1, \dots, I^n$ and a non-trivial tree $t$ such that $I = \bigsqcup_{j=1}^n I^j,$ and $I^j = \cantor_{u^j}$ for some leaf $u^j$ of $t$, for each $\ 1 \leq j \leq n.$ Then for all $a \in K_I$
\begin{align*}
    \zeta^t(a) &= \prod_{u' \in \Leaf(t)} \zeta^{\cantor_{u'}}(a^{(u')}) \\
     &= \prod_{j=1}^n \zeta^{\cantor_{u^j}}(a^{(u^j)}) \\
     &= \prod_{j=1}^n \zeta^{I^j}(a^{(u^j)}) \\
     &= \zeta^I(a),
\end{align*} where the second line follows from the fact that $a$ is supported in $I,$ and the final line can be seen from writing $a = \prod_{i = 1}^n a^{(u^j)}$ and writing $\theta(a)$ as an element of $\wt{K}_{\varphi(I)} \times Z\wt{G}$ in two different ways:
\begin{align*}
   \theta(a) &= \kappa^I(a) \cdot \zeta^I(a) \\ 
             &= \prod_{j=1}^n \kappa^{I^j}(a^{(u^j)}) \cdot \prod_{j = 1}^n  \zeta^{I^j}(a^{(u^j)}).
\end{align*}

\textbf{Claim 2:} For all $a \in K$ and $v \in V$ we have that $\zeta^\cantor(\pi(v)(a)) = \zeta^\cantor(a).$

Let $a \in K, v \in V.$ Then 
\begin{align*}
   \kappa\left(\pi(v)\left(a^{(0)}\right)\right) &= \ad(c_v) \wt{\pi}(\varphi v \varphi^{-1}) \left(\kappa^{I_0}\left(a^{(0)}\right) \cdot \zeta^{I_0}\left(a^{(0)}\right)\right) \\
                   &= \ad(c_v) \wt{\pi}(\varphi v \varphi^{-1}) \left(\kappa^{I_0}\left(a^{(0)}\right) \cdot \zeta^{I_0}\left(a^{(0)}\right)\right) \\
                   &= \kappa^{v(I_0)}\left(\pi(v)\left(a^{(0)}\right)\right) \cdot \zeta^{v(I_0)}\left(\pi(v)\left(a^{(0)}\right)\right), 
\end{align*} and so $\zeta^{I_0}\left(a^{(0)}\right) = \zeta^{v(I_0)}\left(\pi(v)\left(a^{(0)}\right)\right).$ From Claim 1, we have that $\zeta^{I_0} \left(a^{(0)}\right) = \zeta^{\cantor}\left(a^{(0)}\right)$ and $\zeta^{v(I_0)}\left(\pi(v)\left(a^{(0)}\right)\right) = \zeta^\cantor\left(\pi(v)\left(a^{(0)}\right)\right),$ yielding $\zeta^{\cantor}\left(a^{(0)}\right) = \zeta^{\cantor}\left(\pi(v)\left(a^{(0)}\right)\right).$ An identical argument shows that $\zeta^{\cantor}\left(a^{(1)}\right) = \zeta^{\cantor}\left(\pi(v)\left(a^{(1)}\right)\right),$ immediately obtaining the desired equality $\zeta^{\cantor}(a) = \zeta^{\cantor}(\pi(v)(a)).$

Claim 2 allows us to extend $\zeta^{\cantor}$ to a morphism $\zeta: G \to Z \wt{G}$ via the formula $\zeta(av) := \zeta^{\cantor}(a), av \in G.$ Since $\zeta$ is valued in $Z \wt{G},$ the map $\zeta^\dagger: G \to Z \wt{G}, g \mapsto \zeta(g)^{-1}$ is also a group morphism, and by Lemma \ref{lem:centremorphism}, the map $\theta^0 := \theta \cdot \zeta^\dagger$ is an isomorphism. Denote by $\kappa^0 : K \to \wt{K}$ the restriction of $\theta^0$ to $K.$ 

\textbf{Claim 3:} We have that $\supp(\kappa^0(a)) = \varphi(\supp(a))$ for all $a \in K.$

We start by proving the inclusion $\supp(\kappa^0(a)) \subseteq \varphi(\supp(a)).$ The case in which $\supp(a) = \cantor$ is clear. Suppose then that $a \in K$ with $\supp(a) \subsetneq \cantor.$ If $x \notin \supp(a),$ then we may consider an sdi $I \neq \cantor$ such that $x \in I \subseteq \cantor \setminus \supp(a).$ But then $a \in K_{\cantor \setminus I},$ so $\theta(a) = \kappa^{\cantor \setminus I} (a) \cdot \zeta(a)$ by Claim 1, yielding $\kappa^0(a) = \kappa^{\cantor \setminus I}(a).$ Thus, $\varphi(x) \notin \supp(\kappa^0(a)).$ This establishes the inclusion $\supp(\kappa^0(a)) \subseteq \varphi(\supp(a))$ for all $a \in K.$ The reverse inclusion is obtained by applying the same considerations to $\theta^{-1},$ obtaining for all $\wt{a} \in \wt{K}$ that $\supp((\kappa^0)^{-1}(\wt{a})) \subseteq \varphi^{-1}(\supp(\wt{a})).$ In particular, we have that $\supp(a) \subseteq \varphi^{-1}(\supp(\kappa^0(a)))$ for all $a \in K.$

So far, we have found a tuple $(\kappa^0, \zeta, c, \varphi)$ satisfying properties $(1)$ through $(3).$ Uniqueness of $c$ and $\varphi$ follows from Lemma \ref{lem:weakiso}. If $\kappa': K \xrightarrow{\sim} \wt{K}, \zeta' : G \to \wt{G}$ satisfy $\kappa = \kappa' \cdot \zeta'$ with $\supp(\kappa')(a) = \varphi(\supp(a))$ for all $a \in K,$ then if $a$ is an element of $K$ with support strictly contained in $\cantor$ we have that 
\begin{align*}
    \kappa(a) &= \kappa^0(a) \cdot \zeta(a) \\ 
              &= \kappa'(a) \cdot \zeta'(a).
\end{align*} Since $\supp(\kappa^0(a)) = \supp(\kappa'(a)) \subsetneq \cantor,$ we obtain that $\zeta(a) = \zeta'(a),$ and thus $\kappa^0(a) = \kappa'(a).$ One then obtains that $\zeta(a) = \zeta'(a)$ and $\kappa^0(a) = \kappa'(a)$ for all $a \in K$ via decomposing $a = a^{(0)} \cdot a^{(1)},$ noting that $\supp(a^{(i)}) \subseteq \cantor_i.$ Thus, $\kappa^0 = \kappa',$ and $\zeta = \zeta'$ since $\zeta(v) = \zeta'(v) = e$ for all $v \in V.$
\end{proof}

\subsection{Obvious isomorphisms}\label{subsec:obviousiso}
In this subsection we will see some sufficient conditions for two semidirect products $G = K \rtimes V, \ \wt{G} = \wt{K} \rtimes V$ to be isomorphic, where $G$ and $\wt{G}$ are obtained from isomorphisms $K \xrightarrow{\sim} K^2$ and $ \wt{K} \xrightarrow{\sim} \wt{K}^2$ respectively.
\begin{remark}
Retaining the notation of Remark \ref{rmk:covariantmonoidalexample}, in \cite[Proposition 2.4]{jonestechII}, Brothier gives some sufficient conditions for groups $G(\alpha) = K(\alpha) \rtimes V$ with $\alpha_0$, $\alpha_1 \in \aut(\Gamma)$ to be isomorphic, in terms of the automorphisms $\alpha_0$, $\alpha_1$. One such sufficient condition is that if $\wt{\alpha} : \Gamma \to \Gamma^2, g \mapsto (\alpha_1(g), \alpha_0(g))$, then $G(\alpha) \cong G(\wt{\alpha})$. 

Apart from operating within a larger class of groups (Remark \ref{rmk:covariantmonoidalexample}), we are more focused on determining the isomorphism class of a group $G(R) = K \rtimes V$ in terms of the isomorphism $R: K \xrightarrow{\sim} K^2$. For groups $G(\alpha)$ with $\alpha: \Gamma \to \Gamma^2$ and $\alpha_0, \alpha_1 \in \aut(\Gamma)$, we are concerned with how the isomorphism class of $G(\alpha)$ varies with the ``large-scale" isomorphism $R(\alpha)$, rather than the ``small-scale" morphism $\alpha$. Interestingly, changes to $\alpha$ such as swapping $\alpha_0$ and $\alpha_1$ can be detected at the level of the isomorphism $R(\alpha)$, so Proposition \ref{prop:sufficientiso} extends \cite[Proposition 2.4]{jonestechII}. We will not prove this here, since the same principle applies to groups $G(\omega)$ with $\omega: \Gamma^2 \to \Gamma$ (Definition \ref{def:komega}), and is easily adapted to the groups $G(\alpha)$. More precisely, Proposition \ref{prop:sufficientomegaiso} follows from Proposition \ref{prop:sufficientiso}.
\end{remark}
\begin{lemma}\label{lem:functoriso}
    Let $R : K \to K^2$ and $\wt{R}: \wt{K} \to \wt{K}^2$ be group isomorphisms, and let $\kappa: K \to \wt{K}$ be a group morphism such that the diagram 
    \[ 
    \begin{tikzcd} 
    K \arrow[r, "R"] \arrow[d, swap, "\kappa"] & K \times K \arrow[d, "\kappa \times \kappa"] \\
    \wt{K} \arrow[r, swap, "\wt{R}"] & \wt{K} \times \wt{K}
    \end{tikzcd}
    \] commutes. Then $\kappa$ is equivariant with respect to the actions $V \curvearrowright K, \wt{K},$ and thus extends to a group morphism $\theta: G \to \wt{G}, av \mapsto \kappa(a) \cdot v.$ If $\kappa$ is an isomorphism, so is $\theta$.
\end{lemma}
\begin{proof}
We have that $\kappa R_i = \wt{R}_i \kappa, \ i \in \{0,1\},$ which implies that $\kappa R_u = \wt{R}_u \kappa$ for all $u \in \{0,1\}^*.$ Thus, if $t \in \mathfrak{T}$ is a tree, then for all $a \in K,$
\begin{align*}
\wt{\pi}(t) \circ \kappa (a) &= (\wt{R}_l \circ \kappa (a))_{l \in \Leaf(t)} \\ &= (\kappa \circ R_l(a))_{l \in \Leaf(t)} \\ &= \kappa^{\ell(t)} \circ \pi(t)(a).
\end{align*} Hence, we have that \[\wt{\pi}(t) \circ \kappa = \kappa^{\ell(t)} \circ \pi(t)\] for all trees $t \in \mathfrak{T},$ which quickly implies that \[\wt{\pi}(v) \circ \kappa = \kappa \circ \pi(v)\] for all $v \in V.$ The fact that $\theta$ is an isomorphism if $\kappa$ is an isomorphism follows immediately from the formula of $\theta$.
\end{proof}
\begin{lemma}\label{lem:inneriso}
    We have that $G(R \circ \ad(b)) \cong G(R)$ for all $b \in K$.
\end{lemma}
\begin{proof}
    Let $\wt{R} := R \circ \ad(b)$, and $\wt{\pi} := \pi(\wt{R})$. For each word $u = u_0 u_1 \cdots u_n  \in \{0,1\}^*$, define $b_u^R := R_{u_n}(b) \cdot R_{u_{n-1} u_n}(b) \cdots R_u(b)$.

    \textbf{Claim:} For all trees $t \in \mathfrak{T}$ we have that $\wt{\pi}(t) = \ad\left(b_t\right) \circ \pi(t),$ where $b_t := (b_u^R)_{u \in \Leaf(t)}.$

    It is quickly verified that the claim is true when $t = \caret$. Suppose that $t \in \mathfrak{T}$ satisfies the claim, and let $f$ be an elementary forest, with the caret in the $i^{\text{th}}$ position. Then 
    \begin{align*}
    \wt{\pi}(ft) &= \wt{\pi}(f) \wt{\pi}(t) \\ &= \wt{\pi}(f) \ad(b_t) \pi(t) \\
            &= \left(\text{id}^{i - 1} \times \wt{\pi}(\caret) \times \text{id}^{\ell(t) - i}\right) \circ \ad(b_t) \pi(t) \\
            &= \ad(d) \circ \pi(ft)
    \end{align*} where $d \in K^{\ell(t)},$ with for all $1 \leq j \leq \ell(t)$
    \[d_j = \begin{cases}
    b_{u_j}^R & j < i, \\
    b_0^R \cdot R_0(b_{u_i}^R) & j = i, \\ 
    b_1^R \cdot R_1 (b_{u_i}^R) & j = i+1, \\
    b_{u_{j-1}}^R & j > i.
    \end{cases}\] Thus, $d = b_{ft},$ proving the claim.

    This claim implies the identity $\wt{\pi}(f) = \ad(b_f) \pi(f)$ for all forests $f,$ where $b_f := (b_{f_i})_{i=1}^n$ after writing $f$ as a tensor product of trees $f_1, \dots, f_n.$

    Let $v = s^{-1} \sigma t \in V.$ Then for all $a \in K,$
    \begin{align*}
        \wt{\pi}(v)(a) &= \wt{\pi}(s)^{-1} \wt{\pi}(\sigma) \wt{\pi}(t) (a) \\                        &= \pi(s)^{-1} \ad(b_s^{-1}) \pi(\sigma) \ad(b_t) \pi(t)(a) \\
                       &= \pi(s)^{-1}\left(b_s^{-1} \cdot \pi(\sigma)(b_t) \cdot \pi(\sigma t) (a) \cdot \pi(\sigma)( b_t)^{-1} \cdot b_s\right) \\
                       &= \ad\left(\pi(s)^{-1}(b_s^{-1} \cdot \pi(\sigma)(b_t))\right) \circ \pi(v)(a). \tag{$*$}
    \end{align*} 
    The value of $\pi(s)^{-1}(b_s^{-1} \cdot \pi(\sigma)(b_t))$ does not depend on the representative $t,\sigma,s$ of $v,$ and so we may define 
    \begin{align*}
        c_v &:= \left(\pi(s)^{-1}(b_s^{-1} \cdot \pi(\sigma)(b_t))\right)^{-1} \\ &= \pi(s)^{-1}(\pi(\sigma)(b_t^{-1}) \cdot b_s).
    \end{align*} From $(*),$ we have that $\pi(v) = \ad(c_v) \wt{\pi}(v)$ for all $v \in V,$ and a direct computation shows that for all $v, w \in V$ we have that $c_{vw} = \pi(v)(c_w) \cdot c_v.$ Thus, from $(*),$ we have that $c: V \to K$ is a cocycle in the sense that the identity $c_{vw} = c_v \wt{\pi}(v) (c_w)$ holds for all $v, w \in V.$ 
    
    Therefore, the map $\theta: G(R) \to G(\wt{R}), av \mapsto a \cdot c_v \cdot v$ is a group morphism, and is easily shown to have the map $av \mapsto a \cdot c_v^{-1} \cdot v$ as an inverse.
\end{proof}
\begin{lemma}\label{lem:morphismembedding}
Suppose that $f: n \to n \otimes n = 2n$ is a morphism in $\frc{\cat{SF}},$ so that $\pi(f): K^n \to K^n \times K^n$ is an isomorphism. Then there is an embedding $G(\pi(f)) \hookrightarrow G(R).$
\end{lemma}
\begin{proof}
For each morphism $t: 1 \to n \in \frc{\cat{SF}},$ define a map $\theta = \theta(f,t): G(\pi(f)) \to G(R)$ defined by the formula 
\[
\theta(av) := \pi(t^{-1})(a) \cdot t^{-1} \Psi(v) t, \ av \in G(\pi(f)),
\] where $\Psi: \frc{\cat{SF}} \to \frc{\cat{SF}}$ is the functor defined by mapping $1 \mapsto n$ and $\caret \mapsto f.$ A quick computation shows that $\theta$ is a group morphism, and the injectivity of $
\theta$ follows from the fact that the functor $\Psi$ is faithful.
\end{proof}
\begin{remark}
The embedding $\theta(f,t)$ in the proof of Lemma \ref{lem:morphismembedding} is an isomorphism if and only if the functor $\Psi: \frc{\cat{SF}} \to \frc{\cat{SF}}, 1 \mapsto n, \caret \mapsto f$ restricts to an isomorphism $V \to V_n := \frc{\cat{SF}}(n,n).$ There are cases in which the functor $\Psi$ restricted to $V$ is not a surjection onto $V_p.$

If
\[
\begin{tikzpicture}
	\begin{pgfonlayer}{nodelayer}
		\node [style=none] (0) at (0, 0) {};
		\node [style=none] (1) at (-2, 2) {};
		\node [style=none] (2) at (1, 1) {};
		\node [style=none] (3) at (0, 2) {};
		\node [style=none] (4) at (2, 2) {};
		\node [style=none] (5) at (1, 3) {};
		\node [style=none] (6) at (0.5, 1.5) {};
		\node [style=none] (7) at (1.5, 2.5) {};
		\node [style=none] (8) at (-3, 1.25) {$f =$};
		\node [style=none] (9) at (5, 1.25) {$\in \frc{\cat{SF}}(1,2)$,};
	\end{pgfonlayer}
	\begin{pgfonlayer}{edgelayer}
		\draw (2.center) to (6.center);
		\draw (2.center) to (4.center);
		\draw (4.center) to (7.center);
		\draw (7.center) to (5.center);
		\draw (6.center) to (3.center);
		\draw (3.center) to (5.center);
		\draw (6.center) to (7.center);
		\draw (0.center) to (2.center);
		\draw (0.center) to (1.center);
	\end{pgfonlayer}
\end{tikzpicture}
\]
then for all $v \in V,$ as noted by Zaremsky \cite{zaremskyoverflow}, one can show that $\Psi(v) \in V$ maps elements of the cantor space with prefix $11$ to an element of the cantor space with $11$ as a subword. Thus, not every element of $V$ is of the form $\Psi(v)$ with $v \in V$, so the embedding $\theta: G(\pi(f)) \hookrightarrow G(R)$ may be proper.
\end{remark}
\begin{corollary}\label{cor:isoperm}
Consider the isomorphism $\wt{R}: K \to K^2, a \mapsto (R_1(a), R_0(a)).$ Then $G(\wt{R}) \cong G(R).$
\end{corollary}
\begin{proof}
Let $f := \sigma \caret \in \frc{\cat{SF}},$ where $\sigma$ is the transposition of $S_2.$ Then the embedding $\theta: G(\pi(f)) = G(\wt{R}) \hookrightarrow G(R)$ defined in the proof of Lemma \ref{lem:morphismembedding} is surjective, since the functor $\Psi: \frc{\cat{SF}} \to \frc{\cat{SF}}$ mapping $1 \mapsto 1$ and $\caret \mapsto f$ satisfies $\Psi \circ \Psi = \text{id}.$
\end{proof}
Lemmas \ref{lem:functoriso}, \ref{lem:inneriso}, and Corollary \ref{cor:isoperm} yield the following.
\begin{proposition}\label{prop:sufficientiso}
Suppose that $R: K \to K^2$ and $\wt{R}: \wt{K} \to \wt{K}^2$ are group isomorphisms. If there is an isomorphism $\kappa: K \xrightarrow{\sim} \wt{K}$, a permutation $\sigma \in S_2$ and an element $\wt{b} \in \wt{K}$ such that 
\[
\wt{R}_{\sigma(i)} =  \kappa R_i \kappa^{-1} \circ  \ad(\wt{b}) \text{ for $i = 0$ and $i = 1$,}
\] then $G \cong \wt{G}$.
\end{proposition}

\section{Contravariant Jones' technology}\label{sec:wreath}
We will now assign a group isomorphism $R(\omega): K(\omega) \to K(\omega)^2$ to each group morphism $\omega: \Gamma^2 \to \Gamma$, which is the best approximation to $\omega$ by an isomorphism (Definition \ref{def:jonesiso}). Thus, a group morphism $\omega: \Gamma^2 \to \Gamma$ will produce a group $G(\omega) := K(\omega) \rtimes V$.

In Subsection \ref{subsec:unipropfunc}, we verify that $R(\omega): K(\omega)^2 \to K(\omega)^2$ satisfies the required universal property, show that the construction of $G(\omega)$ is functorial, and construct obvious isomorphisms between these groups. Next, in Subsection \ref{subsec:wreathproducts}, we restrict our focus to group morphisms of the form $\omega: \Gamma^2 \to \Gamma, (g,h) \mapsto \alpha(g)$ with $\alpha \in \aut(\Gamma)$ and show that these groups $G(\omega)$ are unrestricted twisted permutational wreath products. We then classify these wreath products up to isomorphism in Subsection \ref{subsec:thinclassification}. We finish with Subsection \ref{subsec:endtoaut}, showing that every group $G(\omega)$ coming from a group morphism of the form $\omega: \Gamma^2 \to \Gamma, (g,h) \mapsto \alpha(g)$ with $\alpha \in \text{End}(\Gamma)$ is isomorphic to one of these wreath products, resulting in a weaker classifcation of the groups in this broader class.
\subsection{Functoriality and obvious isomorphisms}\label{subsec:unipropfunc}
Fix a group $\Gamma$ and a group morphism $\omega: \Gamma^2 \to \Gamma.$ 
\begin{definition}\label{def:komega}
Define $K(\omega)$ to be the subgroup of $\prod_{\{0,1\}^*} \Gamma$ consisting of the maps $a: \{0,1\}^* \to \Gamma$ satisfying the following condition:
\[ 
a(u) = \omega(a(u0), a(u1)) \text{ for all } u \in \{0,1\}^*.
\]
\end{definition}
As noted in Section \ref{sec:characteristic}, the set $\{0,1\}^*$ is in bijection with the vertices of the infinite binary tree $t_\infty$, so the elements of $K(\omega)$ can be thought of as decorations of $t_\infty$. We identify each $a \in K(\omega)$ with the diagram
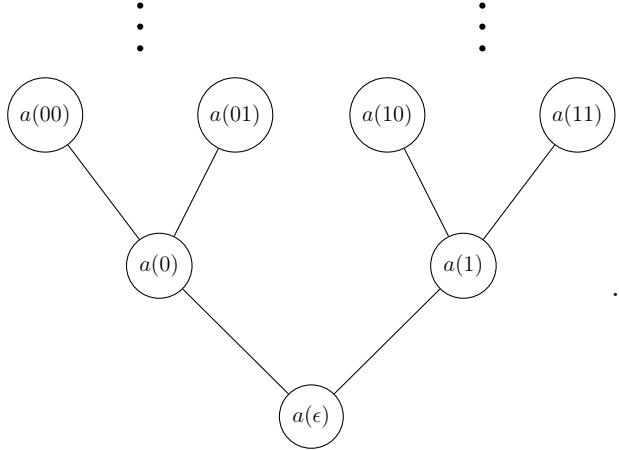
\[
\begin{tikzpicture}
	\begin{pgfonlayer}{nodelayer}
		\node [style=white circle] (0) at (0, 0) {$a(\epsilon)$};
		\node [style=white circle] (1) at (-4, 4) {$a(0)$};
		\node [style=white circle] (2) at (4, 4) {$a(1)$};
		\node [style=white circle] (3) at (-7, 8) {$a(00)$};
		\node [style=white circle] (4) at (-2, 8) {$a(01)$};
		\node [style=white circle] (5) at (2, 8) {$a(10)$};
		\node [style=white circle] (6) at (7, 8) {$a(11)$};
		\node [style=none, scale=2] (7) at (-4.5, 10.75) {$\vdots$};
		\node [style=none, scale=2] (8) at (4.5, 10.75) {$\vdots$};
		\node [style=none] (9) at (8, 3.25) {.};
	\end{pgfonlayer}
	\begin{pgfonlayer}{edgelayer}
		\draw (0) to (1);
		\draw (1) to (3);
		\draw (1) to (4);
		\draw (0) to (2);
		\draw (2) to (5);
		\draw (2) to (6);
	\end{pgfonlayer}
\end{tikzpicture}

\]
We obtain a map $R(\omega): K(\omega) \to K(\omega)^2$ given by deleting a caret at the bottom of diagrams of elements of $K(\omega)$. Visually, if $a \in K$, then $R(\omega)(a)$ has diagram
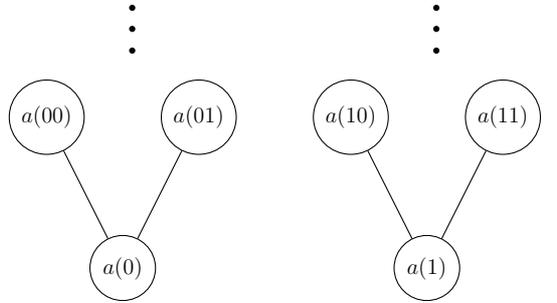
\[
\begin{tikzpicture}
	\begin{pgfonlayer}{nodelayer}
		\node [style=white circle] (0) at (0, 0) {$a(0)$};
		\node [style=white circle] (1) at (8, 0) {$a(1)$};
		\node [style=white circle] (2) at (-2, 4) {$a(00)$};
		\node [style=white circle] (3) at (2, 4) {$a(01)$};
		\node [style=white circle] (4) at (6, 4) {$a(10)$};
		\node [style=white circle] (5) at (10, 4) {$a(11)$};
		\node [style=none, scale=2] (6) at (0.25, 6.75) {$\vdots$};
		\node [style=none, scale=2] (7) at (8.25, 6.75) {$\vdots$};
		\node [style=none] (8) at (11.5, 1.75) {,};
	\end{pgfonlayer}
	\begin{pgfonlayer}{edgelayer}
		\draw (0) to (2);
		\draw (0) to (3);
		\draw (1) to (4);
		\draw (1) to (5);
	\end{pgfonlayer}
\end{tikzpicture}

\] which we interpret as as an element of $K(\omega)^2$. The inverse of $R(\omega)$ is given by placing a caret at the bottom of a pair $(a,a') \in K(\omega)^2$, and decorating the bottom of this caret with the element $\omega(a(0),a'(0)) \in \Gamma$.
\begin{lemma}\label{lem:uniprop}
The isomorphism $R(\omega) : K(\omega) \to K(\omega)^2$ along with the map $K(\omega) \to \Gamma$ given by evaluation at $\epsilon \in \{0,1\}^*$ forms the Jones isomorphism for $\omega: \Gamma \to \Gamma^2$ (Definition \ref{def:jonesiso}).
\end{lemma}
\begin{proof}
Suppose that $\wt{R} : \wt{K} \xrightarrow{\sim} \wt{K}^2$ is an isomorphism of groups, and that $p: \wt{K} \to \Gamma$ is group morphism making the diagram
    \[
    \begin{tikzcd}
    \wt{K} \arrow[d, "p"'] \arrow[r, "\wt{R}"] & \wt{K}^2 \arrow[d, "p \times p"] \\
    \Gamma                                     & \Gamma^2 \arrow[l, "\omega"]    
    \end{tikzcd}
    \]
commute. Define a map $\psi: \wt{K} \to K(\omega)$ sending each $\wt{a} \in K$ to the map $\psi(\wt{a}): \{0,1\}^* \to \Gamma, u \mapsto p \circ \wt{R}_u(\wt{a}).$ It is quick to check that $\psi$ is the unique group morphism making the required diagram commute.
\end{proof}
If $\omega: \Gamma^2 \to \Gamma$ is already an isomorphism, Lemma \ref{lem:uniprop} ensures that $K(\omega)$ is canonically isomorphic to $\Gamma$. Following the notation of Definition \ref{def:jonesiso} and Section \ref{sec:characteristic}, we denote by $\pi(\omega): \frc{\cat{SF}} \to \cat{Grp}$ the functor defined by $\pi(\omega)(\caret) := R(\omega)$ and define $G(\omega) := K(\omega) \rtimes V.$ By definition, we have that $G(\omega) = G(R(\omega))$. 

Since $\pi(\omega): \frc{\cat{SF}} \to \cat{Grp}$ is monoidal, for each $t \in \mathfrak{T}$ and $a \in K$, the isomorphism of groups $\pi(\omega)(t): K(\omega) \to K(\omega)^{\ell(t)}$ is given by deleting a copy of $t$ at the bottom of each diagram in $K(\omega)$. As a result letting $l_1,\dots,l_n$ be the leaves of $t$, we have that
\[ \pi(\omega)(t)(a) = (a_1,\dots,a_n),\] where for each $1 \leq i \leq n$, \[a_i(u) = a(l_i u) \text{ for all $u \in \{0,1\}^*$.}\]
We immediately obtain a formula for the action $V \curvearrowright K(\omega).$
\begin{lemma}\label{lem:contrajonesformula}
Suppose that $a \in K(\omega), \ v = s^{-1} \sigma t \in V,$ and $1 \leq i \leq \ell(t).$ Let $l_t^i$ and $l_s^{\sigma^{-1}(i)}$ be the $i^{\text{th}}$ leaf of $t$ and the $\sigma^{-1}(i)^{\text{th}}$ leaf of $s,$ respectively. Then \[\pi(v)(a)(l_s^i u) = a(l_t^{\sigma^{-1}(i)} u), \ u \in \{0,1\}^*.\]
\end{lemma}
\begin{lemma}\label{lem:contrafunctoriality}
   Suppose that $\wt{\omega}: \wt{\Gamma}^2 \to \wt{\Gamma}$ is a group morphism and that $\beta: \Gamma \to \wt{\Gamma}$ is a group morphism making the diagram
  \[\begin{tikzcd}
	{\Gamma^2} & \Gamma \\
	{\wt{\Gamma}^2} & {\wt{\Gamma}}
	\arrow["\omega", from=1-1, to=1-2]
	\arrow["{\wt{\omega}}"', from=2-1, to=2-2]
	\arrow["{\beta^2}"', from=1-1, to=2-1]
	\arrow["\beta", from=1-2, to=2-2]
\end{tikzcd}\]
   commute. Then the map $\kappa: K \to \wt{K}$ defined by the formula \[\kappa(a)(u) := \beta(a(u)), \ u \in \{0,1\}^*\] is an equivariant group morphism, and thus extends to a group morphism $\theta: G \to \wt{G}, av \mapsto \kappa(a) v.$
\end{lemma}
\begin{proof}
    By Lemma \ref{lem:uniprop}, there is a unique group morphism $K(\omega) \to K(\wt{\omega})$ making the diagram 
        \[
        \begin{tikzcd}
        \Gamma \arrow[ddd, "\beta"', bend right]                     & \Gamma^2 \arrow[l, "\omega"'] \arrow[ddd, "\beta \times \beta", bend left] \\
        K(\omega) \arrow[d, dashed] \arrow[r, "R(\omega)"] \arrow[u] & K(\omega)^2 \arrow[d] \arrow[u]                                            \\
        K(\wt{\omega}) \arrow[r, "R(\wt{\omega})"'] \arrow[d]        & K(\wt{\omega})^2 \arrow[d]                                                 \\
        \wt{\Gamma}                                                  & \wt{\Gamma}^2 \arrow[l, "\wt{\omega}"]                                    
        \end{tikzcd}
        \]
    commute. Here the maps $K(\omega) \to \Gamma$ and $K(\wt{\omega}) \to \wt{\Gamma}$ are the evaluation maps at $\epsilon \in \{0,1\}^*$, and the maps $K(\omega)^2 \to \Gamma^2$, $K(\wt{\omega})^2 \to \wt{\Gamma}^2$ are the squares of these maps.

    The group morphism $K(\omega) \to K(\wt{\omega})$ corresponding to the dashed line in the above diagram is exactly $\kappa$. This map is precisely $\kappa: K(\omega) \to K(\wt{\omega})$. By Lemma \ref{lem:functoriso}, $\kappa$ is equivariant with respect to the actions $V \curvearrowright K, \wt{K}$, and thus extends to a group morphism $G \to \wt{G}$ which acts as the identity on $V$.
\end{proof}
\begin{corollary}\label{cor:omegafunctoriso}
    If $\wt{\Gamma}$ is a group and $\beta: \Gamma \to \wt{\Gamma}$ is a group isomorphism, then $G(\omega) \cong G(\wt{\omega}),$ where $\wt{\omega} := \beta^{-1} \omega (\beta \times \beta).$
\end{corollary}
\begin{remark}\label{rmk:functoriality}
    Suppose that we have a commutative diagram of groups
        \[
        D = \begin{tikzcd}
        \Gamma \arrow[r, "\alpha"] \arrow[d, swap,  "\psi"] & \Gamma^2 \arrow[d, "\psi \times \psi"]                     \\
        \wt{\Gamma}                          & \wt{\Gamma}^2 \arrow[l, "\wt{\omega}"]
        \end{tikzcd} \tag{I}
        \]
 Retaining the notation of Remark \ref{rmk:covariantmonoidalexample}, let $\pi, \wt{\pi}: \frc{\cat{SF}} \to \cat{Grp}$ be the functors obtained from the isomorphisms $R(\alpha): K(\alpha) \to K(\alpha)^2$ and $R(\wt{\omega}): K(\wt{\omega}) \to K(\wt{\omega})^2$ respectively. For each $[t,g] \in K(\alpha),$ we may consider the element  \[K(D)([t,g]) := \wt{\pi}(t)^{-1}(a_1, \dots, a_{\ell(t)}) \in K(\wt{\omega}),\] where the $a_i \in K(\wt{\omega})$ are defined by the formula \[ a_i(u) := \psi(\alpha_u(g_i)), \ u \in \{0,1\}^*,  1 \leq i \leq \ell(t).\] One can check that the assignment $K(D): K(\alpha) \to K(\wt{\omega}), [t,g] \mapsto K(D)([t,g])$ is a well-defined group morphism that is equivariant with respect to the actions $V \curvearrowright K(\alpha), K(\wt{\omega}).$ Thus, we may extend $K(D)$ to a group morphism $G(D): G(\alpha) \to G(\wt{\omega}).$

Similarly, a diagram 
\[
D = \begin{tikzcd}
\Gamma \arrow[d, swap, "\psi"]                     & \Gamma^2 \arrow[l, swap, "\omega"] \arrow[d, "\psi \times \psi"] \\
\wt{\Gamma} \arrow[r, swap, "\wt{\alpha}"] & \wt{\Gamma}^2                         
\end{tikzcd} \tag{II}
\]
induces a group morphism $G(D): G(\omega) \to G(\wt{\alpha})$ via the assigment \[K(\omega) \to K(\wt{\alpha}), a \mapsto [\ |,  \psi(a(\epsilon))], \] where we note that \[ \left[t, \psi^{\ell(t)} \left( a_l \right)_{l \in \text{Leaf}(t)}\right] = [ \ | , \psi(a(\epsilon))]\] for all $t \in \mathfrak{T}.$

In Lemma \ref{lem:contrafunctoriality}, we saw that a commutative diagram 
\[
D = 
\begin{tikzcd}
\Gamma \arrow[d, swap, "\psi"] & \Gamma^2 \arrow[l, "\omega"'] \arrow[d, "\psi \times \psi"] \\
\wt{\Gamma}              & \wt{\Gamma}^2 \arrow[l, "\wt{\omega}"]                     
\end{tikzcd} \tag{III}
\] yields a map $G(\omega) \to G(\wt{\omega}),$ which we denote by $G(D).$ Finally, it was observed in \cite{tan16} that a commutative diagram 
\[
D =
\begin{tikzcd}
\Gamma \arrow[d, swap, "\psi"] \arrow[r, swap, "\alpha"'] & \Gamma^2  \arrow[d, "\psi \times \psi"] \\
\wt{\Gamma} \arrow[r, swap, "\wt{\alpha}"]               & \wt{\Gamma}^2                     
\end{tikzcd} \tag{IV}
\] yields a map $G(D): G(\alpha) \to G(\wt{\alpha})$ defined by the formula \[ G(D)[t,g] = [t, \psi^{\ell(t)}(g)].\] Consider now the category $\cat{M}$ whose objects are either morphisms of the form $\alpha: \Gamma \to \Gamma^2$ or $\omega: \Gamma^2 \to \Gamma$ in $\cat{Grp}$, and whose morphisms are commutative squares such as (I) through (IV). Tanushevski proves that the map of categories $G: \cat{M} \to \cat{Grp}$ restricts to a functor on the subcategory of $\cat{M}$ given by diagrams $D$ such as (IV); similar computations verify that $G$ is indeed a functor on all of $\cat{M}$. I.e., the assignment $\alpha: \Gamma \to \Gamma^2 \mapsto G(\alpha)$ of Tanushevski and Brothier along with the map $\omega: \Gamma^2 \to \Gamma \mapsto G(\omega)$ assemble into a functor $\cat{M} \to \cat{Grp}$, with $G(R) \cong G(R^{-1}) \cong K \rtimes V$ when $R: K \to K^2$ is a group isomorphism.

If $\beta$ is an automorphism of a group $\Gamma$, and $\alpha: \Gamma \to \Gamma^2, g \mapsto (\beta(g), e)$,  $\omega: \Gamma^2 \to \Gamma$, $(g,h) \mapsto \beta^{-1}(g)$, then the diagram (I) commutes after letting $\wt{\omega} = \omega$ and taking $\psi = \text{id}_\Gamma$. By applying the functor $G$, we obtain an embedding $G(\alpha) \hookrightarrow G(\omega)$, which identifies $G(\alpha)$ with the restricted twisted permutational wreath product $\oplus_{\QQ_2} \Gamma \rtimes V$ after identifying $G(\omega)$ with the corresponding unrestricted wreath product (Corollary \ref{cor:wreathiso}).

If instead $\alpha(g) = (g,\beta(g))$ and $\omega(g,h) = g$ for all $g,h \in \Gamma$, the functor $G$ yields an embedding $G(\alpha) \hookrightarrow G(\omega)$. In \cite{jonestechII}, Brothier showed that $G(\alpha)$ is isomorphic to a semidirect product $L \Gamma \rtimes V$, where $L \Gamma$ are the continouous maps $\cantor \to \Gamma$, the latter equipped with the discrete topology. The action $V \curvearrowright L \Gamma$ is given by translation, and twisting by the automorphism $\beta$. However, Corollary \ref{cor:wreathiso} identifies $G(\omega)$ with an untwisted wreath product $\prod_{\QQ_2} \Gamma \rtimes V$, so this embedding is somewhat surprising. This embedding ``untwists" the action $V \curvearrowright L \Gamma$, at the cost of the image of a map of $L \Gamma$ in $\prod_{\QQ_2} \Gamma$ no longer being continuous in general.
\end{remark}
\begin{lemma}\label{lem:omegainneriso}
    Suppose that $h \in \Gamma$ and that there is a map $b: \{0,1\}^* \to \Gamma$ satisfying the formula \[ b(u) = h \cdot \omega(b(u0), b(u1)), \text{ for all $ u \in \{0,1\}^*$.} \tag{$*$}\] Then $G(\omega) \cong G(\ad(h) \circ \omega).$
\end{lemma}
\begin{proof}
    Let $\wt{\omega} := \ad(h) \circ \omega.$ Consider the group isomorphism $\theta: K \to \wt{K}$ given by the formula \[ \theta(a)(u) = \ad(b(u))(a(u)), \ a \in K, \ u \in \{0,1\}^*.\] Here the condition $(*)$ ensures that $\theta$ takes values in $\wt{K}.$ Define maps $d^0, d^1: \{0,1\}^* \to \Gamma$ given by the formulae \[ d^i(u) := b(u) \cdot b(iu)^{-1}, \ u \in \{0,1\}^*, \ i \in \{0,1\}. \] Then $d^0, d^1 \in \wt{K},$ and a quick computation shows that \[ \theta R_i \theta^{-1} = \ad(d^i) \wt{R}_i, \ i \in \{0,1\}. \] Letting $R' := (\theta \times \theta) \circ R \circ \theta^{-1},$ we then have that $R' = \ad(d^0,d^1) \circ \wt{R},$ and so by Proposition \ref{prop:sufficientiso} it follows that $G(R) \cong G(R') \cong G(\wt{R}).$
\end{proof} If $\omega: \Gamma^2 \to \Gamma$ is surjective, we can recursively build a map $b: \{0,1\}^* \to \Gamma$ satisfying condition $(*)$ of Lemma \ref{lem:omegainneriso}, satisfying $b(\epsilon) := e_\Gamma$ and $\omega(b(u0), b(u1)) = h^{-1} \cdot b(u)$ for all $u \in \{0,1\}^*.$
\begin{corollary}\label{cor:surjinneriso}
    If $\omega: \Gamma^2 \to \Gamma$ is surjective, then $G(\ad(h) \circ \omega) \cong G(\omega)$ for all $h \in \Gamma.$
\end{corollary}
\begin{lemma}\label{lem:omegasubgroup}
    There is a subgroup $\Gamma_\omega \leq \Gamma$ such that $\omega: \Gamma \times \Gamma \to \Gamma$ restricts to a surjective group morphism $\omega_{\mathrm{res}}: \Gamma_\omega \times \Gamma_\omega \to \Gamma_\omega$, and $G(\omega) \cong G(\omega_{\mathrm{res}})$.
\end{lemma}
\begin{proof}
    For each $g \in \Gamma$, write $\omega_0(g) := \omega(g,e)$ and $\omega_1(g) := \omega(e,g)$, obtaining endomorphisms $\omega_0, \omega_1 \in \mathrm{End}(\Gamma)$ such that $\omega(g,h) = \omega_0(g) \cdot \omega_1(h)$ for all $g, h \in \Gamma$. If $u = u_1 \cdots u_n \in \{0,1\}^*$, define $\omega_u := \omega_{u_1} \circ \cdots \circ \omega_{u_n}$. Notice that the order of the letters of $u$ is preserved, unlike in Definition \ref{def:gensupp}. For each tree $t \in \mathfrak{T}$, define a group morphism $\omega_t : \Gamma^{\ell(t)} \to \Gamma$ defined by the formula 
    \[
    \omega_t (g) := \prod_{l \in \Leaf(t)} \omega_l(g_l), \ g: \Leaf(t) \to \Gamma,
    \]
    noting that the order in which the product is taken is immaterial, since the images of $\omega_0$ and $\omega_1$ commute. We note that if $\Phi: \cat{F} \to \cat{Grp}$ is the contravariant monoidal functor defined by $\Phi(\caret) := \omega$, then $\omega_t = \Phi(t)$ for all trees $t$.

    Let $\Gamma_\omega := \bigcap_{t \in \mathfrak{T}} \mathrm{im}(\omega_t)$. The fact that $\omega(g,h) \in \Gamma_\omega$ when $g, h \in \Gamma_\omega$ follows immediately from the equality of maps
    \[
    \omega \circ (\omega_t \times \omega_s) = \omega_{f} : \Gamma^{\ell(t) + \ell(s)} \to \Gamma,
    \] where $f := (t \otimes s) \circ \caret$. By the definition of $K(\omega)$ (Definition \ref{def:komega}), we have that $a(u) \in \Gamma_\omega$ for all $a \in K(\omega)$ and $u \in \{0,1\}^*$. Thus, $K(\omega) = K(\omega_{\mathrm{res}})$, where $\omega_{\mathrm{res}}$ is the restriction of $\omega$ to $\Gamma_\omega \times \Gamma_\omega$, and so $G(\omega) = G(\omega_{\mathrm{res}})$.
\end{proof}
\begin{remark}
    Lemma \ref{lem:omegasubgroup} says that by passing to a subgroup of $\Gamma$, we may assume that $\omega$ is surjective. This is dual to a result of Tanushevski (Proposition 3.2 of \cite{tan16}). Recalling the notation of Remark \ref{rmk:covariantmonoidalexample}, Tanushevski proves that for each group morphism $\alpha: \Gamma \to \Gamma \times \Gamma$, there is a normal subgroup $N \vartriangleleft \Gamma$ such that the map 
        \[
        \overline{\alpha}: \Gamma/N \to \Gamma/N \times \Gamma/N, gN \mapsto (\alpha_0(g)N, \alpha_1(g)N)
        \]
    is well-defined, injective, and satisfies $G(\alpha) \cong G(\overline{\alpha})$. I.e., by passing to a quotient of $\Gamma$, we can assume that $\alpha$ is injective.
\end{remark}
\begin{lemma}\label{lem:omegaperm}
Let $\wt{\omega}: \Gamma^2 \to \Gamma$ be defined by the formula $\wt{\omega}(g,h) := \omega(h,g)$ for all $g, h \in \Gamma$. Then $G(\wt{\omega}) \cong G(\omega)$.
\end{lemma}
\begin{proof}
    Let $\kappa: K(\omega) \to K(\wt{\omega})$ be the isomorphism of groups defined by the formula 
        \[
        \kappa(a)(u) := a(\neg \ u) \text{ for all  $u \in \{0,1\}^*$},
        \] where $\neg \ 0 := 1$, $\neg \ 1 := 0$, and $\neg \ u$ is the word in $\{0,1\}^*$ given by applying $\neg$ to each digit of $u$. It is quickly checked that 
        \[
        \wt{R}_{\sigma(i)} = \kappa R_i \kappa^{-1} \text{ for $i = 0, 1$.}
        \] Thus, $G \cong \wt{G}$ by Proposition \ref{prop:sufficientiso}.
\end{proof}
Corollaries \ref{cor:omegafunctoriso}, \ref{cor:surjinneriso} and Lemmas \ref{lem:omegainneriso}, \ref{lem:omegasubgroup}, \ref{lem:omegaperm} yield the following.
\begin{proposition}\label{prop:sufficientomegaiso}
    Suppose that $\wt{\omega}: \wt{\Gamma}^2 \to \wt{\Gamma}$ is a group morphism satisfying 
    \[
    \wt{\omega}\left(\widetilde{g}_1,\widetilde{g}_2\right) = \beta \circ \ad(g) \circ \omega \circ (\beta \times \beta)^{-1}\left(\wt{g}_{\sigma(1)},\wt{g}_{\sigma(2)}\right) \text{ for all $\wt{g}_1, \wt{g}_2 \in \wt{\Gamma}$,}
    \] for some $\sigma \in S_2$, $g \in \Gamma_\omega$ and an isomorphism $\beta: \Gamma \to \wt{\Gamma}$. Then $G \cong \wt{G}$.
\end{proposition}
\subsection{Wreath products}\label{subsec:wreathproducts}
In this subsection we will demonstrate that certain choices of group morphisms $\omega: \Gamma^2 \to \Gamma$ yield a description of $G(\omega)$ as an unrestricted twisted permutational wreath product $\Gamma \wr V.$

Fix a group $\Gamma$, an automorphism $\alpha \in \aut(\Gamma),$ and define a group morphism $\omega: \Gamma^2 \to \Gamma$ mapping $(g,h) \mapsto \alpha(g).$
\begin{proposition}\label{prop:wreath} The map $\kappa: \prod_{\QQ_2} \Gamma \to K(\omega)$ sending a map $f : \QQ_2 \to \Gamma$ to the map $\{0,1\}^* \to \Gamma, u \mapsto \alpha^{-|u|} (f(u00\cdots))$, is an isomorphism of groups.

Thus, we may view an element $a \in K(\omega)$ as a map $\QQ_2 \to \Gamma$, writing $a(x) \in \Gamma$ for the image of $x \in \QQ_2$ under $\kappa^{-1}(a)$.
\end{proposition}
\begin{proof}
It is quick to verify that the map $\kappa : \prod_{\QQ_2} \Gamma \to \prod_{\{0,1\}^*} \Gamma$ is a group morphism taking values in $K(\omega)$. The inverse of $\kappa$ maps an element $a \in K(\omega)$ to the map $\kappa^{-1}(a): \QQ_2 \to \Gamma$, given by the formula 
\[
\kappa^{-1}(a)(x) = \alpha^{|u|}(a(u)), \ x \in \QQ_2, \ u \in \{0,1\}^* \text{ with } x = u 00 \cdots.
\] If $x \in \QQ_2$ and $u, u'$ are two words in $\{0,1\}^*$ satisfying $x = u00\cdots = u'00\cdots$, then without loss $u' = u0^{n}$ for some $n \geq 0$, so
    \begin{align*}
        \alpha^{|u'|}(a(u')) &= \alpha^{|u| + n}(\alpha^{-n}(a(u))) \\
                             &= \alpha^{|u|}(a(u)).
    \end{align*} Thus, our formula for $\kappa^{-1}(a)$ indeed defines a map $\QQ_2 \to \Gamma$.
\end{proof}
\begin{remark}
The isomorphism $\kappa : \prod_{\QQ_2} \Gamma \xrightarrow{\sim} K(\omega)$ in Proposition \ref{prop:wreath} might not be the most obvious one. Instead, we could take an element $a \in K(\omega)$, and assign it the map $f: \QQ_2 \to \Gamma$ defined by the formula $f(x) = a(u_x),$ where $u_x$ is the longest prefix of $x$ not ending in $0$ ($u_{00\cdots} = \epsilon$ by convention). From $f$ we can recover $a: \{0,1\}^* \to \Gamma$ via the formula 
\[
    a(u) = \alpha^{-n}(f(u00\cdots)), u \in \{0,1\}^*,
\] where $u = u_{x} 0^n$, and $x = u00\cdots$.

We favour the isomorphism of Proposition \ref{prop:wreath} due to the convenient formula for the action $V \curvearrowright \prod_{\QQ_2} \Gamma$, induced by the Jones action $V \curvearrowright K(\omega)$ and the isomorphism $\kappa$ given in Proposition \ref{prop:wreathaction}.
\end{remark}
\begin{proposition}\label{prop:wreathaction} Let $v \in V$ and $a \in K(\omega).$ Viewing the elements of $K(\omega)$ as maps $\QQ_2 \to \Gamma,$ for all $x \in \QQ_2$ we have that
\[ \pi(v)(a)(x) = \alpha^{- \log_2(v'(v^{-1}x))}(a(v^{-1}x)). \]
\end{proposition}

\begin{proof}
Let $\kappa$ be the isomorphism $\prod_{\QQ_2} \Gamma \to K(\omega)$ of Proposition \ref{prop:wreath}, and write $v = s^{-1} \sigma t$, where $t$ and $s$ are trees with $\ell(t) = \ell(s)$, and $\sigma \in S_{\ell(t)}$. Suppose that $f: \QQ_2 \to \Gamma$ is a map, and that $x \in \QQ_2.$ Let $l_s^i$ be the $i^{\text{th}}$ leaf of $s$ which is also a prefix of $x$, so that $x = l_s^i u 00\cdots$ for some $u \in \{0,1\}^*$. Then
\begin{align*}
    \left(\kappa^{-1} \circ \pi(v) \circ \kappa(f)\right)(x) &= \alpha^{|l_s^i u|}(\pi(v) \circ  \kappa(f)(l_s^i u)) \\
    &= \alpha^{|l_s^i u|} \left( \alpha^{-|l_t^{\sigma^{-1}(i) }u|} \left(f(l_t^{\sigma^{-1}(i)} u 00\cdots\right)\right) \\
    &= \alpha^{-N}(f(v^{-1}x)),
\end{align*} where the first line follows from the formula for $\kappa^{-1}$ in the proof of Proposition \ref{prop:wreath}, the second from Lemma \ref{lem:contrajonesformula}, and $N := |l_t^{\sigma^{-1}(i)}| - |l_s^i|$. By Definition \ref{def:slope}, we indeed have that $N = \log_2(v'(v^{-1}x))$, so we are done.
\end{proof}
\begin{corollary}\label{cor:wreathiso}
Consider the action $\rho: V \curvearrowright \prod_{\QQ_2} \Gamma$ given by the formula 
\[ 
\rho(v)(a)(x) = \alpha^{-\log_2(v'(v^{-1}x))}(a(v^{-1}x)), \ a: \QQ_2 \to \Gamma, \ v \in V, \ x \in \QQ_2.
\] Then $G(\omega)$ is isomorphic to the unrestricted twisted permutational wreath product \\ $\prod_{\QQ_2} \Gamma \rtimes_\rho V.$
\end{corollary}
We will now reconcile the notion of support of an element $a \in K(\omega)$ given by $R(\omega)$, and the set of non-zero values of $a \in K(\omega)$, viewed as a map $\QQ_2 \to \Gamma$. Recall from Definition \ref{def:gensupp} that for each $ a \in K(\omega),$ we have that \[\supp(a) = \supp_{R(\omega)}(a) := \{x = x_0 x_1 \cdots \in \cantor \ | \ R(\omega)_{x_0 \cdots x_m}(a) \neq e_\Gamma \text{ for all } m \geq 0\}.\]
\begin{lemma}\label{lem:densesupp}
Let $a \in K(\omega).$ Then the set \[ \{x \in \QQ_2 \ | \ a(x) \neq e_\Gamma\}\] is dense in $\supp(a).$
\end{lemma}
\begin{proof}
    Let $S := \{x \in \QQ_2 \ | \ a(x) \neq e_\Gamma\}.$ Suppose that $x \in S,$ and let $u \in \{0,1\}^*$ be a prefix of $x$ such that $x = u00\cdots.$ Then
    \begin{align*}
         R_u(a)(\epsilon) = a(u) = \alpha^{-|u|}(a(u00\cdots))
    \end{align*}
    which yields that $R_u(a) \neq e_K.$ Hence, $R_{x_0 x_1 \cdots x_m}(a) \neq e_\Gamma$ for sufficiently large $m$, and so $x \in \supp(a).$

    Conversely, suppose that $x \in \supp(a).$ Then for all $m \geq 0$ we have that $R_{x_0 \cdots x_m}(a) \neq e_K,$ and thus there is a word $u_m \in \{0,1\}^*$ such that \[R_{x_0 \cdots x_m}(a)(u_m) = a(x_0 \cdots x_m u_m) \neq e_\Gamma. \tag{$*$} \] Letting $y_m := x_0 \cdots x_m u_m 0 0 \cdots$ for all $m \geq 0,$ the sequence $(y_m)_{m \geq 0}$ is contained in $S$ by $(*)$. Finally, for each $m \geq 0,$ the first $m$ digits of $y_m$ and $x$ agree, so $y_m \to x$ in $\cantor$.
\end{proof}
\subsection{Classification of wreath products}\label{subsec:thinclassification}
In this subsection, we give a necessary and sufficient condition for two wreath products considered in the previous subsection to be isomorphic.
\begin{definition}
    If $\Gamma$ is a group and $\alpha \in \text{End}(\Gamma),$ denote by $G(\alpha) = K(\alpha) \rtimes V$ the group $G(\omega)$ coming from the group morphism $\omega: \Gamma^2 \to \Gamma, (g,h) \mapsto \alpha(g)$. Thus, if $\alpha \in \aut(\Gamma)$, we identify $G(\alpha)$ with the wreath product of Corollary \ref{cor:wreathiso}. We will stick to our convention of writing $G, \wt{G}$ instead of $G(\alpha), G(\wt{\alpha})$, given endomorphisms $\alpha$ and $ \wt{\alpha}$ of groups $\Gamma$ and $\wt{\Gamma}$.

    If $\alpha$ is an endomorphism of a group $\Gamma$, we denote by $\Gamma^\alpha$ the fixed points of $\alpha$. Moreover, if $g \in \Gamma$ and $x \in \QQ_2$, denote by $g_x$ the map $\QQ_2 \to \Gamma$ supported at $x$ and equal to $g$. I.e., $g_x(x) = g$ and $g_x(y) = e_\Gamma$ for all $y \neq x$.
\end{definition}
Throughout the rest of the subsection, $\Gamma$ and $\wt{\Gamma}$ will be groups, and $\alpha, \wt{\alpha}$ will be \textit{automorphisms} of $\Gamma$ and $\wt{\Gamma}$, respectively.
\begin{lemma}\label{lem:wreathcentre}
The centre $Z(G(\alpha))$ consists of the constant maps $\QQ_2 \to \Gamma$ taking values in $Z\Gamma \cap \Gamma^\alpha.$
\end{lemma}
\begin{proof}
Suppose that $a \in Z(G(\alpha)).$ By Lemma \ref{lem:gencentre}, we have that $R_u(a) = a$ for all $u \in \{0,1\}^*.$ But then $a(u) = R_u(a)(\epsilon) = a(\epsilon)$ for all $u \in \{0,1\}^*,$ and so $a$ is constant when viewed as a map $\{0,1\}^* \to \Gamma.$ Moreover, $a : \{0,1\}^* \to \Gamma$ is valued in $ Z\Gamma \cap \Gamma^\alpha$ since $a \in ZK,$ and $\alpha(a(u0)) = a(u)$ for all $u \in \{0,1\}^*,$ and so the formula \[ a(x) = \alpha^{|u|}(a(u)) \ x \in \QQ_2, \ u \in \{0,1\}^*, \ x = u00\cdots  \] from the proof of Proposition \ref{prop:wreath} allows us to conclude that $a : \QQ_2 \to \Gamma$ is constant and valued in $Z \Gamma \cap \Gamma^\alpha.$

Conversely, if $a: \QQ_2 \to \Gamma$ is constant and valued in $Z \Gamma \cap \Gamma^\alpha,$ then $a \in ZK$ and $\pi(v)(a) = a$ for all $v \in V,$ implying that $a \in ZG.$
\end{proof}
    \begin{proposition}\label{prop:finitesuppiso}
    Suppose that $\Gamma$ is not the trivial group, and let $\theta: G \to \wt{G}$ be an isomorphism. Using the notation of Theorem \ref{thm:isodecomp}, write \[ \theta(av) = \kappa^0(a) \cdot \zeta(a) \cdot c_v \cdot \ad_{\varphi}(v), \ av \in G.\] Then $\varphi \in \stabn$ (Definition \ref{def:nvcstabn}), and there is a unique family of isomorphisms $(\kappa_x: \Gamma \to \wt{\Gamma})_{x \in \QQ_2}$ such that \[ \kappa(g_x) = \kappa_x(g)_{\varphi(x)} \mod Z \wt{G} \text{ for all } g \in \Gamma, \ x \in \QQ_2.\]
    
    Moreover, the map $\kappa^1: K \to \wt{K}$ defined by the formula \[\kappa^1(a)(\varphi(x)) := \kappa_x(a(x))\] induces an isomorphism $\theta^1: G \to \wt{G}, av \mapsto \kappa^1(a) \cdot c_v \cdot \ad_{\varphi}(v)$.
    \end{proposition}
    \begin{proof}
        We start by showing that $\varphi \in \stabn.$ Let $x \in \QQ_2$ and $g$ be a non-trivial element of $\Gamma.$ By Lemma \ref{lem:densesupp}, we have that $\supp(g_x) = \{x\},$ so Theorem \ref{thm:isodecomp} implies that $\supp(\kappa^0(g_x)) = \{\varphi(x)\}$. Again by Lemma \ref{lem:densesupp}, the set \[\{ y \in \QQ_2 \ | \kappa^0(g_x)(y) \neq e_\Gamma\}\] is non-empty and contained in $\{\varphi(x)\},$ and so $\varphi(x) \in \QQ_2.$ A similar argument applied to $\theta^{-1}$ shows that $\varphi^{-1}(x) \in \QQ_2$.
        
        Thus, we may write \[ \kappa(g_x) = \kappa_x(g)_{\varphi(x)} \cdot \zeta(g_x) \] for some non-trivial element $\kappa_x(g) \in \wt{\Gamma}$. However, Lemma \ref{lem:wreathcentre} tells us that $\wt{h}_{\varphi(x)} \notin Z\wt{G}$ for all non-trivial $\wt{h} \in \wt{\Gamma}$, so $\kappa_x(g)$ is the unique element of $\wt{\Gamma}$ satisfying 
            \[
            \kappa(g_x) = \kappa_x(g)_{\varphi(x)} \mod Z\wt{G}.
            \]
         So far we have shown that $\varphi \in \stabn,$ and that for each non-trivial $g \in \Gamma$ and $x \in \QQ_2,$ there is a unique $\kappa_x(g) \in \wt{\Gamma}$ satisfying $\kappa(g_x) = \kappa_x(g)_{\varphi(x)} \mod Z \wt{G}$. A similar argument shows that if $x \in \QQ_2$, then the only element $\wt{h} \in \wt{\Gamma}$ satisfying $\kappa(e_x) = \wt{h}_{\varphi(x)} \mod Z \wt{G}$ is $\wt{h} = \wt{e}$, the identity element of $\wt{\Gamma}$. We therefore define $\kappa_x(e) := \wt{e}$. It only remains to show that each map $\kappa_x : \Gamma \to \wt{\Gamma}, \ x \in \QQ_2$ is an isomorphism.
        
        Fix $x \in \QQ_2.$ The map $\kappa_x: \Gamma \to \wt{\Gamma}$ is multiplicative since $\kappa$ is a group morphism. We already saw that if $g \in \Gamma$ is not the identity $e$, then $\kappa_x(g) \neq \wt{e}$, so $\kappa_x$ is injective. Next, let $\wt{g} \in \wt{\Gamma}.$ If $\wt{g} = \wt{e},$ then $\wt{g} = \kappa_x(e).$ If $\wt{g} \neq \wt{e},$ then applying the same considerations to $\theta^{-1}$ as we did to $\theta$, we have that $\theta^{-1}(\wt{g}_{\varphi(x)}) = g_x \cdot a$ for some non-trivial $g \in \Gamma$ and $a \in ZG.$ But then 
        \begin{align*}
            \kappa(g_x \cdot a) &= \kappa_x(g)_{\varphi(x)} \cdot \wt{a} \\ &= \wt{g}_{\varphi(x)}
        \end{align*} for some $\wt{a} \in Z \wt{G}.$ Hence, $\kappa_x(g)_{\varphi(x)} = \wt{g}_{\varphi(x)}$, implying that $\kappa_x(g) = \wt{g}$, so $\kappa_x : \Gamma \to \wt{\Gamma}$ is an isomorphism. 
        
        Finally, consider the maps $\kappa^1: K \to \wt{K}, \theta^1: G \to \wt{G}$ defined by the formulae \[ \kappa^1(a)(\varphi(x)) := \kappa_x(a(x)) \] and \[ \theta^1: G \to \wt{G}, av \mapsto \kappa^1(a) c_v \ad_{\varphi(v)}  \] for all $a \in K, x \in \QQ_2$, and $v \in V$.
        
        Since $\kappa^0$ and $\kappa^1$ agree on $\oplus_{\QQ_2} \Gamma,$ we have that \[ \kappa^1 \circ \pi(v)(a) = \ad(c_v)\circ \wt{\pi}(\ad_{\varphi}(v))\circ \kappa^1(a) \tag{$*$} \] for all $a \in \oplus_{\QQ_2} \Gamma.$ Suppose now that $a \in K, \ v \in V$ and $x \in \QQ_2.$ Then 
        \begin{align*}
            \kappa^1 \circ \pi(v)(a)(\varphi(vx)) &= \kappa_{vx} \left(\pi(v)(a)(vx)\right) \\ 
            &= \kappa_{vx}(\alpha^{-\log_2(v'(x))}(a(x))) \\ 
            &= \kappa^1 \circ \pi(v)(a(x)_x)(\varphi(vx)).
        \end{align*} A similar computation shows that 
        \begin{align*}
            \ad(c_v) \circ \wt{\pi}(\ad_{\varphi}(v)) \circ \kappa^1(a)(\varphi(vx)) &= \ad(c_v) \circ \wt{\pi}(\ad_{\varphi}(v)) \circ \kappa^1(a(x)_x)(\varphi(vx)),
        \end{align*} and so condition $(*)$ holds for all $ a \in K.$ This allows us to conclude that $\theta^1$ is a group morphism, and the fact that $\theta^1$ is an isomorphism follows from the fact that $\kappa^1: K \to \wt{K}$ and $\theta: G \to \wt{G}$ are isomorphisms.
    \end{proof}
\begin{remark}\label{rmk:bitflip}
The assumption that at least one of the groups $\Gamma, \wt{\Gamma}$ is non-trivial in Proposition \ref{prop:finitesuppiso} is essential. If either $\Gamma$ or $\wt{\Gamma}$ were trivial, we would have $G = \wt{G} = V$, so $\varphi$ could be any homeomorphism of the Cantor space normalising $V$. The bit-flipping homeomorphism, which maps an element $x = x_0 x_1 \cdots \in \cantor$ to the element $(\neg \ x_0) (\neg \ x_1) \cdots \in \cantor$, where $\neg \ 0 := 1$ and $\neg \ 1 := 0$, normalises $V$ but doesn't stabilise $\QQ_2$. Conjugating an element of $V$ by this homeomorphism reflects its diagram about the vertical axis.
\end{remark}
\begin{remark}
A corollary of Proposition \ref{prop:finitesuppiso} is that if $G(\alpha) \cong G(\wt{\alpha})$ with $\alpha \in \aut(\Gamma)$ and $\wt{\alpha} \in \aut(\wt{\Gamma})$, then $\Gamma \cong \wt{\Gamma}$. Theorem \ref{thm:thinclassification} is stronger than this, and relates the automorphisms $\alpha$ and $\wt{\alpha}$.
\end{remark}
\begin{theorem}\label{thm:thinclassification}
The groups $G, \wt{G}$ are isomorphic if and only if $\wt{\alpha} = \ad(h) \circ \beta \alpha \beta^{-1}$ for some isomorphism $\beta: \Gamma \to \wt{\Gamma}$ and $h \in \wt{\Gamma}.$
\end{theorem}
\begin{proof}
The case in which either $\Gamma$ or $\wt{\Gamma}$ are the trivial group is clear, so suppose that $\Gamma$ is non-trivial. 

If $\wt{\alpha} = \ad(h) \circ \beta \alpha \beta^{-1}$ for some $h \in \wt{\Gamma}$ and $\beta: \Gamma \to \wt{\Gamma}$ an isomorphism, then $G(\alpha) \cong G(\wt{\alpha})$ via Corollaries \ref{cor:omegafunctoriso} and \ref{cor:surjinneriso}.

For the converse, suppose that $\theta: G \to \wt{G}$ is an isomorphism, and write \[\theta(av) = \kappa^0(a) \cdot \zeta(a) \cdot c_v \cdot \ad_{\varphi}(v), \ av \in G,\] using the notation of Theorem \ref{thm:isodecomp}. Let $(\kappa_x : \Gamma \to \wt{\Gamma})$ be the family of isomorphisms from Proposition \ref{prop:finitesuppiso}, and fix $g \in \Gamma, \ v \in V,$ and $x \in \QQ_2$. Then
\begin{align*}
    \theta^1(v \cdot g_x \cdot v^{-1}) &= \theta^1 \left(\left[\alpha^{-\log_2(v'(x))}\left( g\right)\right]_{vx} \right) \\ &= \left[\kappa_{vx} \circ \alpha^{-\log_2(v'(x))}(g)\right]_{\varphi(vx)}.
\end{align*} On the other hand,
\begin{align*}
   \theta^1(v \cdot g_x \cdot v^{-1}) &= \ad(c_v) \circ \wt{\pi}(\ad_{\varphi}(v)) \left( \kappa_x(g)_{\varphi(x)} \right)  \\ &= \ad(c_v) \left(\left[ \wt{\alpha}^{-\log_2((\varphi v \varphi^{-1})'(\varphi(x))} \circ \kappa_x(g) \right]_{\varphi(vx)} \right).
\end{align*} Evaluating the resulting equation
\begin{align*}
    \left[\kappa_{vx} \circ \alpha^{-\log_2(v'(x))}(g)\right]_{\varphi(vx)} = \ad(c_v) \left(\left[ \wt{\alpha}^{-\log_2((\varphi v \varphi^{-1})'(\varphi(x))} \circ \kappa_x(g) \right]_{\varphi(vx)} \right)
\end{align*} after fixing $x = 00\cdots \in \QQ_2$ and choosing $v \in V$ satisfying $v(x) = x$ and $v'(x) = \frac 1 2$ yields \[ \kappa_0 \circ \alpha (g) = \ad(c_v)(\varphi(x)) \circ \wt{\alpha} \circ \kappa_0(g).\] Here we have used Proposition \ref{prop:homeoslope} to establish that $(\ad_{\varphi}(v))'(\varphi(x)) = v'(x) = \frac 1 2$. We have arrived at the desired equality \[ \wt{\alpha} = \ad(\wt{h}) \circ \beta \alpha \beta^{-1}, \] where $\wt{h} := c_v(\varphi(x))^{-1}$ and $\beta := \kappa_0$.
\end{proof}
\begin{lemma}\label{lem:etalem}
Using the notation of Theorem \ref{thm:isodecomp}, suppose that
    \[
    \theta: G(\alpha) \to G(\wt{\alpha}), av \mapsto \kappa^0(a) \cdot \zeta(a) \cdot c_v \cdot \ad_{\varphi}(v)
    \]
is an isomorphism. Then $a(x) \in Z \Gamma$ imples that $\kappa^0(a)(\varphi(x)) \in Z \wt{\Gamma}$ for all $a \in K$ and $x \in \QQ_2$.

Recalling the family of isomorphisms $(\kappa_x: \Gamma \to \wt{\Gamma})_{x \in \QQ_2}$ and the isomorphism $\kappa^1: K \to \wt{K}$ of Proposition \ref{prop:finitesuppiso}, the map $\eta: K \to K, a \mapsto \kappa^1(a)^{-1} \cdot \kappa^0(a) $ is a group morphism that is valued in $ZK$, and satisfies the formula \[ \eta \circ \pi(v) = \ad(c_v) \circ \wt{\pi}(\ad_\varphi(v)) \circ \eta\] for all $v \in V$.
\end{lemma}
\begin{proof}
Suppose that $a \in K$ and $x \in \QQ_2$ satisfy $a(x) \in Z \Gamma$. Then if $\wt{g} \in \wt{\Gamma}$,
\begin{align*}
    [\kappa^0(a), \wt{g}_{\varphi(x)}] &= \kappa^0\left( \left[a, (\kappa^0)^{-1}\left( \wt{g}_{\varphi(x)}\right)\right]\right) \\
    &= \kappa^0\left([a, g_x] \right)
    \end{align*} for some $g \in \Gamma.$ Since $a(x) \in Z \Gamma,$ we obtain that $[a,g_x] = e_K,$ and so $[\kappa^0(a), \wt{g}_{\varphi(x)}] = e_{\wt{K}}.$ Evaluating at $\varphi(x)$ yields that $\kappa^0(a)(\varphi(x)) \in Z \wt{\Gamma},$ which proves the first part of the lemma.

    Now letting $a \in K$ and $x \in \QQ_2$ be arbitrary, and defining $a' := a(x)_x^{-1} \cdot a$, we have that
        \begin{align*}
            \kappa^1(a)^{-1} \cdot \kappa^0(a) &= \kappa^1(a)^{-1} \cdot \kappa^0(a(x)_x) \cdot \kappa^0(a').
        \end{align*}
    Evaluating the above at $\varphi(x),$ we obtain that 
    \begin{align*}
        \eta(a)(\varphi(x)) &= \kappa^1(a)(\varphi(x))^{-1} \cdot \kappa^0(a(x)_x)(\varphi(x)) \cdot \kappa^0(a')(\varphi(x)) \\
        &= \kappa^1(a(x)_x)(\varphi(x))^{-1} \cdot \kappa^1(a(x)_x)(\varphi(x)) \cdot \kappa^0(a')(\varphi(x)) \\
        &= \kappa^0(a')(\varphi(x)).
    \end{align*} Thus, $\eta(a)(\varphi(x)) \in Z \wt{\Gamma}$ by applying the first part of the lemma to the fact that $a'(x) = e_\Gamma$. If $b$ is another element of $K$, we have that \[\eta(ab) = \kappa^1(b)^{-1} \eta(a) \kappa^0(b) = \eta(a)\eta(b).\] Thus, $\eta$ is a group morphism $K \to ZK$.
    Finally, if $v \in V,$ then
    \begin{align*}
        \eta(\pi(v)(a)) &= \kappa^1(\pi(v)(a))^{-1} \cdot \kappa^0(\pi(v)(a)) \\
                        &= \ad(c_v) \circ \wt{\pi}(\ad_{\varphi}(v)) \left( \kappa^1(a)^{-1} \cdot \kappa^0(a)\right) \\ &= \ad(c_v) \circ \wt{\pi}(\ad_{\varphi}(v)) \circ \eta(a).
    \end{align*}
\end{proof}
\subsection{From endomorphism to automorphism}\label{subsec:endtoaut}
We finish Section \ref{sec:wreath} by demonstrating that for each group $\Gamma$ and endomorphism $\alpha: \Gamma \to \Gamma,$ there is a group $\limg$ and an automorphism $\lima$ of $\limg$ such that $G(\alpha) \cong G(\lima)$. Thus, $G(\alpha)$ is a wreath product $\varprojlim \Gamma \wr V$ (Corollary \ref{cor:wreathiso}). Using Theorem \ref{thm:thinclassification}, we obtain a (weaker) classification of these groups $G(\alpha)$ with $\alpha \in \mathrm{End}(\Gamma)$.

Let $\Gamma$ be a group, and $\alpha: \Gamma \to \Gamma$ be an endomorphism of $\Gamma$.
\begin{definition}
Define $\limg$ to be the inverse limit of the inverse system of groups
    \[ 
    \Gamma \xleftarrow{\alpha} \Gamma \xleftarrow{\alpha} \Gamma \xleftarrow{\alpha} \cdots,
    \] 
so that 
    \[
    \limg = \left\{(g_n) \in \prod_{\NN} \Gamma \ | \ g_n = \alpha(g_{n+1}) \text{ for all } n \in \NN\right\}.
    \] 
Let $\lima$ be the map $\limg \to \prod_{\NN} \Gamma$ defined by the formula 
    \[
    \left[\left(\lima\right)(g)\right]_n := \alpha(g_n), \ g \in \limg, n \in \NN.
    \] 
\end{definition}
\begin{lemma}
The map $\lima: \limg \to \prod_{\NN} \Gamma$ takes values in $\limg$, and the resulting map $\lima: \limg \to \limg$ is a group automorphism. Moreover, the inverse of $\lima$ is the left shift $ \limg \to \limg, (g_n) \mapsto (g_{n+1})$.
\end{lemma}
\begin{proof}
Suppose that $g \in \limg,$ and $n \in \NN.$ Then \begin{align*}
    \alpha\left(\left[\left(\lima\right)(g)\right]_{n+1}\right) &= \alpha^2(g_{n+1}) \\ &= \alpha(g_n) \\ &= \left[\left(\lima\right)(g)\right]_n,
\end{align*}  proving that $\lima$ takes values in $\limg.$

It is clear that $\lima \in \text{End}(\limg)$, and the left shift map $\limg \to \limg, (g_n) \mapsto (g_{n+1})$ is easily verified to be the inverse of $\lima$.
\end{proof}
\begin{lemma}\label{lem:autouniprop}
The group $\limg$ and the automorphism $\lima \in \aut(\limg)$ make the diagram
\[\begin{tikzcd}
	\limg & \limg \\
	\Gamma & \Gamma
	\arrow[from=1-1, to=2-1]
	\arrow["\lima", from=1-1, to=1-2]
	\arrow["\alpha"', from=2-1, to=2-2]
	\arrow[from=1-2, to=2-2]
\end{tikzcd} \tag{I}\]
commute (all arrows $\limg \to \Gamma$ denote the projection map $(g_n) \mapsto g_0$), and are universal with respect to this property, in the sense that every commutative diagram 
\[\begin{tikzcd}
	{\wt{\Gamma}} & {\wt{\Gamma}} \\
	\Gamma & \Gamma
	\arrow["p"', from=1-1, to=2-1]
	\arrow["{\wt{\alpha}}", from=1-1, to=1-2]
	\arrow["\alpha"', from=2-1, to=2-2]
	\arrow["p", from=1-2, to=2-2]
\end{tikzcd} \tag{II}\]
with $\wt{\Gamma}$ a group, $\wt{\alpha} \in \aut(\wt{\Gamma}),$ and $p: \wt{\Gamma} \to \Gamma$ a group morphism, factors uniquely through the diagram $(*)$ as illustrated below:
\[
\begin{tikzcd}
 \wt{\Gamma} \arrow[r, "\wt{\alpha}"] \arrow[d, "\psi", dashed] \arrow[dd, "p"', bend right=49] &  \wt{\Gamma} \arrow[d, "\psi"', dashed] \arrow[dd, "p", bend left=49] \\
\limg \arrow[r, "\lima"] \arrow[d]                                                              & \limg \arrow[d]                                                       \\
\Gamma \arrow[r, "\alpha"']                                                                     & \Gamma                                                               
\end{tikzcd} \tag{III}\]
\end{lemma} 
\begin{proof}
Commutativity of the diagram (I) is quickly verified. Let $\wt{\Gamma}$ be a group, $\wt{\alpha} \in \aut(\wt{\Gamma}),$ and $p: \wt{\Gamma} \to \Gamma$ be a group morphism making the diagram (II) commute. 

If $\psi: \wt{\Gamma} \to \limg$ makes (III) commute, immediately
    \begin{align*}
    \psi \circ \wt{\alpha}^{-n}(\wt{g}) &= \lima^{-1} \circ \psi \circ \wt{\alpha}^{-(n-1)}(\wt{g})  = \cdots = \lima^{-n} \circ \psi(\wt{g})
    \end{align*} for all $n \in \NN$. Evaluating the above at $0 \in \NN$, for all $n \in \NN$ we have that
    \begin{align*}
       \psi(\wt{g})_n &= \left[\psi \circ \wt{\alpha}^{-n}(\wt{g})\right]_0 \\ &= \left[\psi \circ \wt{\alpha}^{-1} \left( \wt{\alpha}^{1-n} (\wt{g}_0)\right)\right]_0 \\ &= p \circ \alpha^{-1} \left(\wt{\alpha}^{-(n-1)}(\wt{g})\right) \\ &= p \circ \wt{\alpha}^{-n}(\wt{g}).
    \end{align*} The above defines the required map $\psi: \wt{\Gamma} \to \limg$.
\end{proof}
    \begin{example}
    If $\alpha \in \aut(\Gamma),$ then the projection $\limg \to \Gamma$ is an isomorphism.
    \end{example}
    \begin{example}\label{eg:triviallimgrp}
    Suppose that $\Gamma = \ZZ,$ $q \in \mathbb{Z}$, and that $\alpha: n \mapsto qn.$ A sequence of integers in $\limg$ must be constant and equal to $0$, so $\limg$ is the trivial group.
    \end{example}
    \begin{example}\label{eg:pontduality}
    Suppose that $\Gamma$ is the compact multiplicative group $S^1$ of unit-length complex numbers. Let $q$ be a non-zero integer, and define $\alpha: z \mapsto z^q$, so that $\limg$ is the compact group consisting of sequences $(z_n)$ of points in $S^1$ with $z_{n+1}^q = z_n$ for all $n \in \NN$. The Pontryagin dual $\widehat{\limg}$ of $\limg$ can then be identified with the direct limit $\varinjlim \widehat{\Gamma}$ of the diagram
    \[\widehat{\Gamma} \xrightarrow{\widehat{\alpha}} \widehat{\Gamma} \xrightarrow{\widehat{\alpha}} \widehat{\Gamma} \xrightarrow{\widehat{\alpha}} \cdots,\] 
    where $\widehat{\Gamma}$ is the dual group of $\Gamma$, and $\widehat{\alpha} : \widehat{\Gamma} \to \widehat{\Gamma}$ is the dual morphism of $\alpha$ (precomposition of characters with $\alpha$).
    
    Identifying $\widehat{\Gamma}$ with $\ZZ$, the dual map $\widehat{\alpha}$ becomes the multiplication map $\ZZ \to \ZZ, n \mapsto qn$, and the dual automorphism $\widehat{\lima}: \varinjlim \widehat{\Gamma} \to \varinjlim \widehat{\Gamma}$ is given by the formula \[\widehat{\lima}([n,m]) = [n,qm] \text{ for all $m,n \in \NN$.}\] However, $\varinjlim \widehat{\Gamma} \cong \ZZ[\frac 1 q]$ via the map $[n,m] \mapsto \frac{m}{q^n}$, and under this identification, the automorphism $\widehat{\varprojlim \alpha} \in \aut(\ZZ[\frac 1 q])$ is the multiplication map by $q$.
    
    Thus, $\limg = \widehat{\ZZ[\frac{1}{q}]}$ and $\lima = \widehat{M_q}$, where $M_q: \ZZ[\frac 1 q] \to \ZZ[\frac 1 q], x \mapsto qx$.
    \end{example}
\begin{proposition}\label{prop:endtoaut}
We have that $G(\alpha) \cong G(\lima).$
\end{proposition}
\begin{proof}
Since the diagram $(I)$ of Lemma \ref{lem:autouniprop} commutes, by Lemma \ref{lem:contrafunctoriality} we obtain a $V$-equivariant group morphism $ \theta: K(\lima) \to K(\alpha)$, defined by the formula 
    \[ 
    \theta(a)(u) =  p(a(u)), \ a \in K(\lima), \ u \in \{0,1\}^*, 
    \] 
where $p: \limg \to \Gamma, (g_n) \mapsto g_0$ is the canonical projection. 

Suppose that $a \in K(\lima),$ and that $\theta(a) = e_{K(\alpha)},$ so that $p(a(u)) = e_\Gamma$ for all $u \in \{0,1\}^*$. If $u \in \{0,1\}^*$ and $n \in \NN$, since $\left(\lima\right)^{-1} :(g_n) \mapsto (g_{n+1})$, we have that  
\begin{align*}
   a(u)_n &=  p \circ \left(\lima\right)^{-n} (a(u)) \\ 
          &= p (a(u0^n)) \\ &= e_\Gamma.
\end{align*} Thus, $a$ is the identity element of $K(\lima)$, and so $\theta$ is injective. 

One can quickly verify that for each $a \in K(\alpha)$, the map $\widehat{a}: \{0,1\}^* \to \limg$ given by the formula \[ \widehat{a}(u)(n):= a(u0^n), \ u \in \{0,1\}^*, \ n \in \NN\] is an element of $K(\lima)$, and satisfies $\theta(\widehat{a}) = a$. Thus, $\theta : K(\lima) \to K(\alpha)$ is a $V$-equivariant isomorphism, yielding that $G(\lima) \cong G(\alpha)$.
\end{proof} 
We arrive at the following corollary of Theorem \ref{thm:thinclassification} and Proposition \ref{prop:endtoaut}.
\begin{corollary}
    Suppose that $\Gamma$ and $\wt{\Gamma}$ are groups with endomorphisms $\alpha \in \mathrm{End}(\Gamma)$, and $\wt{\alpha} \in \mathrm{End}(\wt{\Gamma})$. Then 
        \[
        G(\alpha) \cong G(\wt{\alpha}) \text{ if and only if $\varprojlim \wt{\alpha} = \ad(\widetilde{h}) \circ \beta \circ \lima \circ \beta^{-1}$}
        \] for some $\wt{h} \in \varprojlim \wt{\Gamma}$ and some isomorphism $\beta: \limg \to \varprojlim \wt{\Gamma}$.
\end{corollary}

\section{Automorphisms of untwisted wreath products}\label{sec:automorphisms}
In this section, we decompose the automorphism group of an unrestricted, untwisted wreath product $G = \prod_{\QQ_2} \Gamma \rtimes V$ into an iterated semidirect product \[(((( A_6\rtimes A_5 ) \rtimes A_4) \rtimes A_3) \rtimes A_2) \rtimes A_1\] of subgroups $A_i \leq \aut(G)$ (Theorem \ref{thm:autdecomp}).

Let $\Gamma$ be a non-trivial group, put $K := \prod_{\QQ_2} \Gamma$, and let $G := K \rtimes V$ be the unrestricted, untwisted permutational wreath product coming from the action $V \curvearrowright \QQ_2$. Recall that $K$ can be identified with the group $K(\omega)$, where $\omega: \Gamma^2 \to \Gamma, (g,h) \mapsto g$ (Proposition \ref{prop:wreathaction}). This is equivariant with respect to the actions $V \curvearrowright K, K(\omega)$, so that $G \cong G(\omega)$ (Corollary \ref{cor:wreathiso}). Since we denote the action $V \curvearrowright K(\omega)$ by $\pi$, we will do the same for the wreath action $V \curvearrowright K$. 

We also freely use the notations of Theorem \ref{thm:isodecomp} and Proposition \ref{prop:finitesuppiso}. In particular, for each $\theta \in \aut(G)$, $av \in G$, $x \in \QQ_2$, and $g \in \Gamma$, we write
\begin{itemize}
    \item $\theta(av) = \kappa^0(a) \cdot \zeta(a) \cdot c_v \cdot \ad_{\varphi}(v)$,
    \item $\theta(a) = \kappa(a) = \kappa^0(a) \cdot \zeta(a)$,
    \item $\kappa^0(a)(g_x) = \kappa_x(g(x))$,
    \item $\kappa^1(a)(\varphi(x)) = \kappa_x(a(x))$.
\end{itemize} The fact that Theorem \ref{thm:isodecomp} applies to automorphisms of $G$ can be seen in two ways: the action $V \curvearrowright K(\omega)$ is implemented by an isomorphism $R(\omega): K(\omega) \to K(\omega)^2$, and the action $V \curvearrowright K$ is implemented by the isomorphism $K \to K^2$ given by the bijection $\QQ_2 \to \QQ_2 \sqcup \QQ_2$ of Example \ref{eg:cantordyadicjonesiso}.
\begin{lemma}
The map $\epi_1: \aut(G) \to \stabn$ mapping an automorphism $\theta \in \aut(G)$ to the homeomorphism $\varphi \in \stabn$ obtained by writing $\theta(v) = c_v \cdot \ad_\varphi(v)$ for all $v \in V$ is a split group epimorphism.
\end{lemma}
\begin{proof}
Well-definition of $\epi_1$ follows from Theorem \ref{thm:isodecomp} and Proposition \ref{prop:finitesuppiso}. A quick computation shows that $\epi_1$ is a group morphism, and a section for $\epi_1$ is the map $\stabn \to \aut(G), \varphi \mapsto \left(av \mapsto (a \circ \varphi^{-1}) \cdot \ad_{\varphi}(v) \right)$.
\end{proof} Thus, defining $A_1 := \stabn,$ we have that \[ \aut(G) = \ker \epi_1 \rtimes A_1.\] 
\begin{lemma}\label{lem:autsplit}
    Suppose that $\theta \in \ker \chi_1,$ so that $\theta(av) = \kappa(a) \cdot c_v \cdot v$ for all $av \in G.$ Letting $\epi_2(\theta) := \kappa_{00\cdots} \in \aut(\Gamma)$, the resulting mapping $\epi_2: \ker \chi_1 \to \aut(\Gamma)$ is a split group epimorphism.
\end{lemma}
\begin{proof}
    A direct computation shows that $\epi_2$ is a group morphism, and a section for $\epi_2$ is the map $\aut(\Gamma) \to \aut(G), \beta \mapsto \left(av \mapsto \overline{\beta}(a) v\right),$ where for each $\beta \in \aut(\Gamma)$,
        \[ \overline{\beta}(a)(x) := \beta(a(x)), \ a \in K, \ x \in \QQ_2.\]
\end{proof}
We thus have a further decomposition \[ \ker \epi_1 = \ker \epi_2 \rtimes A_2,\] where $A_2 := \aut(\Gamma).$ 

Our next step is to construct a split epimorphism $\epi_3: \ker \epi_2 \to Z \Gamma$.
\begin{lemma}\label{lem:kappa1inner}
    If $\theta \in \ker \epi_2$, then there exists $h \in K$ such that $\kappa^1 = \ad(h)$.       
\end{lemma}
\begin{proof}
    If $\theta \in \ker \epi_2$, then since $\kappa_{00\cdots} = \text{id}_{\Gamma}$, for all $v \in V$, $x \in \QQ_2,$ and $g \in \Gamma$ we have that 
    \begin{align*}
    \theta^1(v \cdot g_x \cdot v^{-1}) &= \theta^1(g_{vx}) = \ad(c_v)\kappa_x(g)_{vx} \\ &= \kappa_{vx}(g)_{vx},
    \end{align*} yielding that $\ad(c_v(vx)) \circ \kappa_x = \kappa_{vx}.$ Since the action $V \curvearrowright \QQ_2$ is transitive and $\kappa_{00\cdots} = \text{id}_\Gamma,$ the automorphism $\kappa_x \in \aut(\Gamma)$ is inner for all $x \in \QQ_2$. For each $x \in \QQ_2,$ choose $v_x \in V$ such that $v_x(00\cdots) = x,$ and define $h(x) := c_{v_x}(x)$. It is immediate that $\kappa^1 = \ad(h).$
\end{proof}
Suppose that $\theta \in \ker \epi_2$, and that $h \in K$ satisfies $\kappa^1 = \ad(h)$. Then for all $v \in V$ and $x \in \QQ_2,$
\begin{align*}
\kappa^1(vav^{-1})(vx) = \ad(h(vx))(a(x)) = \ad(c_v(vx) \cdot h(x))(a(x)),
\end{align*} implying that for all $v \in V$
    \[c_v = h \cdot \pi(v)(h)^{-1} = [h,v] \mod ZK.\]
Let $d(h): V \to ZK$ be the unique map satisfying 
    \[c_v = d(h)_v \cdot [h,v], \ v \in V.\] 
A quick computation shows that $d(h)$ satisfies the cocycle identity 
    \[d(h)_{vw} = d(h)_v \cdot \pi(v)(d_w) \text{ for all $v, w \in V$.}\]
We will now extract an element $z(h) \in Z \Gamma$ from the cocycle $d(h): V \to ZK$. We will need the following definition and classification of cocycles $V \to \prod_{\QQ_2} Z \Gamma$ given in \cite[Proposition 4.7]{jonestechI} to do so.
    \begin{definition}\label{def:centrecocycle}
    If $z \in Z \Gamma,$ let $s(z) : V \to K$ be the map defined by the formula 
        \[s(z)_v (x) := z^{\log_2(v'(v^{-1}x))}, \ v \in V, \ x \in \QQ_2.\]
    The chain rule implies that $s(z)$ is a cocycle, in the sense that $s(z)_{vw} = s(z)_{v} \cdot \pi(v)\left(s(z)_w\right)$ for all $v, w \in V$.
    \end{definition}
    \begin{proposition}\label{prop:abeliancoc}
    Suppose that $\delta: V \to \prod_{\QQ_2} Z\Gamma$ is a cocycle, i.e. satisfies the identity 
        \[\delta_{vw} = \delta_v \cdot \pi(v) \left( \delta_w\right), \ v,w \in V.\]
    Then there are unique elements $z \in Z \Gamma$ and $f \in \left( \prod_{\QQ_2} Z\Gamma \right) / Z\Gamma$ such that \[ \delta_v = s(z)_v \cdot f \cdot (\pi(v) (f))^{-1} = s(z)_v \cdot [f,v]\] for all $v \in V.$ Here $Z \Gamma$ is identified with the constant maps in $\prod_{\QQ_2} Z \Gamma$.
    \end{proposition}
Now consider the unique elements $z(h) \in Z\Gamma$ and $f(h) \in ZK/Z\Gamma$ satisfying 
    \[d(h)_v = s(z(h))_v \cdot [f(h),v] \text{ for all $v \in V$.}\] 
\begin{lemma}\label{lem:epi3welldef}
    Suppose that $h, h' \in K$ satisfy $\kappa^1 = \ad(h) = \ad(h')$. Then $z(h) = z(h')$.
\end{lemma}
\begin{proof}
        Writing $z, f, z', f'$ instead of $z(h), f(h), z(h'), f(h')$ respectively, for all $v \in V$ we have that
                \[c_v = s(z)_v  \cdot [f,v] \cdot [h,v] = s(z')_v  \cdot [f',v] \cdot [h',v].
                \]
        However, since $\ad(h) = \ad(h'),$ we have that $h' = ha$ for some $a \in ZK$. This implies that for all $v \in V$
            \begin{align*}
                s(z')_v \cdot [f',v] \cdot [h',v] &= s(z')_v \cdot [f',v] \cdot [ha,v] \\ &= s(z')_v \cdot [f'a,v] \cdot [h,v] \\
                          &= s(z)_v  \cdot [f,v] \cdot [h,v],
            \end{align*} 
        so
            \[ s(z')_v \cdot [f'a,v] = s(z)_v \cdot [f,v].\]
        Proposition \ref{prop:abeliancoc} implies that $z = z'$.
    \end{proof}
Lemmas \ref{lem:kappa1inner} and \ref{lem:epi3welldef} allow us to define a map $\epi_3: \ker \epi_2 \to Z \Gamma$ that sends an automorphism $\theta \in \ker \epi_2$ to the element $z(h) \in Z \Gamma$ obtained by writing $\kappa^1 = \ad(h)$ with $h \in K$.
\begin{lemma}\label{lem:epi3}
    The map $\epi_3: \ker \epi_2 \to Z \Gamma$ is a split epimorphism.
\end{lemma}
\begin{proof}
    A quick computation shows that $\epi_3$ is a group morphism, and a section to $\epi_3$ is the map $Z \Gamma \to \ker \chi_2, z \mapsto (av \mapsto a \cdot s(z)_v \cdot v).$
\end{proof} Lemma \ref{lem:epi3} yields the decomposition \[\ker \chi_2 = \ker \chi_3 \rtimes A_3, \ A_3 := Z\Gamma.\] Next, we will define a group morphism $\epi_4: \ker \epi_3 \to K / Z\Gamma$ that splits onto its image. Suppose that $\theta \in \ker \epi_3$, and write $\kappa^1 = \ad(h)$ for some $h \in K.$ Since $z(h) = e_\Gamma$, there is a map $f: \QQ_2 \to Z \Gamma$ such that
    \[c_v = [hf,v], \ v \in V.\]
Thus, for all $av \in G$,
    \begin{align*}
        \theta^1(av) &= h a h^{-1} \cdot [hf, v] \cdot v \\ &= (hf) \cdot a \cdot (hf)^{-1} \cdot [hf,v] \cdot v
        \\ &= (hf) \cdot av \cdot (hf)^{-1}.
    \end{align*} 
We obtain a group morphism \[ \epi_4: \ker \chi_3 \to K / Z\Gamma\] that maps $\theta \in \ker \epi_3$ to the element $h \in K / Z \Gamma$ satisfying $\theta^1 = \ad(h)$. Here we identify $ZG \leq K$ with $Z \Gamma$ (Lemma \ref{lem:wreathcentre}).

If $\theta \in \ker \epi_3$, when we write $\theta^1 = \ad(h)$ with $h \in K/Z \Gamma$, we have that $h(00\cdots) \in Z \Gamma$ since $\kappa_{00\cdots} = \mathrm{id}_\Gamma$. Conversely, if $h \in K / Z \Gamma$ satisfies $h(00\cdots) \in Z \Gamma$, it is quick to verify that the inner automorphism $\ad(h) \in \aut(G)$ is an element of $\ker \epi_3$. Hence, $\epi_4$ is a split epimorphism $\ker \epi_3 \to A_4$, where $A_4 := \{h \in K / Z \Gamma \ | \ h(00\cdots) \in Z \Gamma\}$. Therefore,
\[\ker \chi_3 = \ker \chi_4 \rtimes A_4.\] 
\begin{lemma}
The map $\chi_5: \ker \chi_4 \to \aut(G), \theta \mapsto \left(av \mapsto \kappa^0(a) \cdot v\right)$ is a group morphism that is split onto its image.
\end{lemma}
\begin{proof}
The fact that $\chi_5$ takes values in $\aut(G)$ follows from Lemma \ref{lem:centremorphism}. Suppose that $\theta_1, \theta_2 \in \ker \chi_4,$ and write \[ \theta_i(av) = \kappa_i^0(a) \cdot \zeta_i(a) \cdot v, \ i \in \{1,2\},\] using the notation of Theorem \ref{thm:isodecomp}. Then for all $av \in G$ we have that
\begin{align*}
    \theta_1 \circ \theta_2(av) &= \theta_1(\kappa_2^0(a) \cdot \zeta_2(a) \cdot v) \\ &= \kappa_1^0 \circ \kappa_2^0(a) \cdot \kappa_1^0( \zeta_2(a)) \cdot v.
\end{align*} Letting $\theta := \theta_1 \circ \theta_2$ and writing \[\theta(av) = \kappa^0(a) \cdot \zeta(a) \cdot v, \ av \in G,\] from Theorem \ref{thm:isodecomp} we have that $\kappa^0 = \kappa_1^0 \circ \kappa_2^0,$ showing that $\epi_5$ is indeed a group morphism. Moreover, $\epi_5$ splits via the inclusion $\text{im}(\epi_5) \hookrightarrow \ker(\epi_4).$
\end{proof}
Letting $A_5 := \text{im}(\epi_5)$, we have that \[ \ker \chi_4 = \ker \chi_5 \rtimes A_5, \] where \[ \ker \chi_5 = \{\theta \in \aut(G): \theta(av) = a \cdot \zeta(a) \cdot v \text{ for some } \zeta :G \to ZG\}. \] A moment's calculation shows that the map \[\ker \chi_5 \to \text{Hom}(G,ZG) , (av \mapsto a \cdot \zeta(a) \cdot v) \mapsto \zeta \] is an isomorphism of groups ($\text{Hom}(G,ZG)$ having pointwise multiplication). Letting $A_6 := \mathrm{Hom}(G,ZG)$, our work so far leads to the following result.
\begin{theorem}\label{thm:autdecomp}
We have a decomposition \[\aut(G) =    ((((A_6\rtimes A_5) \rtimes A_4)\rtimes A_3)\rtimes A_2)\rtimes A_1, \] where
\begin{itemize}
    \item $A_1 = \stabn,$
    \item $A_2 = \aut(\Gamma),$
    \item $A_3 = Z \Gamma,$
    \item $A_4 = \{h \in K \ | \ h(00\cdots) \in Z \Gamma\} / Z \Gamma,$
    \item $A_6 = \text{Hom}(G,ZG).$
\end{itemize} The group $A_5$ consists of the automorphisms $\theta \in \aut(G)$ satisfying the following conditions:
\begin{itemize}
    \item $\supp(\theta(a)) = \supp(a)$ for all $a \in K$,
    \item $\theta$ restricts to the identity on $\oplus_{\QQ_2} \Gamma$ and $V$.
\end{itemize}
\end{theorem}
\begin{proof}
We only need to justify the last part of the theorem. It quickly follows from the definition $A_5 := \mathrm{im}(\epi_5)$ that an automorphism $\theta \in \aut(G)$ is an element of $A_5$ if and only if $\theta(av) = \kappa^0(a) \cdot v$ for all $av \in G$, and $\kappa^1 = \mathrm{id}_K$. By Theorem \ref{thm:isodecomp}, the condition $\theta(a) = \kappa^0(a)$ is equivalent to saying that $\theta$ is support-preserving, and by definition, the condition $\kappa^1 = \mathrm{id}_K$ is equivalent to saying that $\theta$ is the identity on $\oplus_{\QQ_2}\Gamma$.
\end{proof}
\begin{remark}\label{rmk:exoticauto}
The most mysterious group in the decomposition of $\aut(G)$ in Theorem \ref{thm:autdecomp} is $A_5$, which doesn't seem to have a description in terms of less complicated groups. By Lemma \ref{lem:etalem}, for each $\theta \in A_5$ we can write $\theta(av) = a \cdot \eta(a) \cdot v$ for some $V$-equivariant group morphism $\eta: K \to ZK$. So if $\Gamma$ has trivial centre or is a perfect group, we have that $A_5 = \{\text{id}_G\}$. It is unknown whether $A_5$ is non-trivial even when $\Gamma = \ZZ_2$ (Questions \ref{question:z2auto} and \ref{question:z2autorephrased}).
\end{remark}
\begin{remark}\label{rmk:autpermwreath}
It is interesting to compare Theorem \ref{thm:autdecomp} with Houghton's description of $\aut(W)$ \cite{autstandwreath} for \textit{standard} unrestricted wreath products $W = \prod_{B} \Gamma \rtimes B$, where $B$ is an arbitrary group. Here $B \curvearrowright B$ by left multiplication, and $B \curvearrowright \prod_{B} \Gamma $ by translation.

If $\theta \in \aut(W)$, then for each $b \in B$ we have that 
\[
\theta(b) = c_b \cdot \phi_b 
\] for some unique $c_b \in \prod_{B} \Gamma$ and $\phi_b \in B$. The mapping $\phi: b \mapsto \phi_b$ then defines an automorphism of $B$, resulting in a group morphism $\phi: \aut(W) \to \aut(B), \theta \mapsto \phi(\theta)$. A section for $\phi$ is obtained by mapping each $\beta \in \aut(B)$ to the automorphism $\theta \in \aut(W)$ defined by the formula 
\[
\theta(f \cdot b) = (f \circ \beta^{-1}) \cdot \beta(b), \ f \in \prod_{B} \Gamma , \ b \in B. \tag{$\clubsuit$}
\] This results in the decomposition
\[\aut(W) = \ker \phi \rtimes \aut(B).\]
Letting $G := \prod_{\QQ_2} \Gamma \rtimes V$, the analogous group morphism $\epi_1: \aut(G) \to \aut(V)$ used in the first step of the proof of Theorem \ref{thm:autdecomp} is obtained in the same way.

One difference between $\phi$ and $\epi_1$ is that the image of $\epi_1$ is not the full automorphism group of $V$. Every automorphism of $V$ is implemented via conjugation by a unique homeomorphism of the Cantor space (Proposition \ref{prop:autv}). In Proposition \ref{prop:finitesuppiso}, we saw that the homeomorphism $\varphi \in \mathrm{Homeo}(\cantor)$ implementing the automorphism $\epi_1(\theta) \in \aut(V)$ must stabilise the dyadic rationals $\QQ_2$. The bit-flipping homeomorphism $\varphi \in \cantor$ (Remark \ref{rmk:bitflip}) normalises $V$, but does not stabilise $\QQ_2$, so $\im {\epi_1} $ is properly contained in $\aut(V)$.

The group morphism $\epi_1 : \aut(G) \to \aut(V)$ is split onto its image for a different reason than $\phi$. A section for $\epi_1$ is given by mapping each $\varphi \in \stabn$ to the automorphism 
\[
av \mapsto (a \circ \varphi^{-1}) \cdot \varphi v \varphi^{-1}, \ a \in \prod_{\QQ_2} \Gamma, \ v \in V.
\]
This section to $\epi_1$ can be described as assigning each automorphism of $V$ in the image of $\epi_1$ a bijection of the base space $\QQ_2$, which can be used to define an automorphism of $G$. The section for $\phi$ in $(\clubsuit)$ is obtained in a similar but more straightforward way. So far, we have decompositions
\[\aut(W) = \ker \phi \rtimes \aut(B) \text{ and } \aut(G) = \ker \epi_1 \rtimes A_1, \] where $A_1 := \stabn$. The descriptions of $\aut(G)$ and $\aut(W)$ now diverge significantly. 

Every automorphism $\theta \in \ker \phi$ satisfies $\theta(b) = c_b \cdot b$ for some map $c: B \to \prod_{B} \Gamma$. Houghton then shows that there exists $g \in \prod_{B} \Gamma$ such that
\[ 
c_b = g b g^{-1} = [g,b] \cdot b, \ b \in B,
\]
so $\theta$ can be written as a (not necessarily unique) product $\theta_0 \circ \theta_1$, where $\theta_0$ fixes $B$ pointwise, and $\theta_1 = \ad(g)$ with $g \in \prod_{B} \Gamma$. This is not always true for elements of $\ker \epi_1$; if $z \in Z \Gamma \setminus \{e_\Gamma\}$, the automorphism $\theta \in A_3$ defined by the formula $\theta(av) := a \cdot s(z)_v \cdot v$ for each $av \in G$ (Definition \ref{def:centrecocycle}), does not map $V$ to a conjugate of $V$ by an element of $\prod_{\QQ_2} \Gamma$. This is the only obstruction: the elements $\theta \in \aut(G)$ which map $V$ to a conjugate of itself by an element of $\prod_{\QQ_2} \Gamma$ are precisely those satisfying $\theta_3 = \mathrm{id}_G$, viewing $\theta$ as a tuple $(\theta_i)_{1 \leq i \leq 6}$ with each $\theta_i \in A_i$ (Theorem \ref{thm:autdecomp}).

Instead, our next steps are to write $\ker \epi_1$ as a semidirect product $\ker \epi_2 \rtimes \aut(\Gamma)$, and $\ker \epi_2 = \ker \epi_3 \rtimes Z \Gamma$ using the split epimorphisms $\epi_2: \ker \epi_1 \to \aut(\Gamma)$, and $\epi_3: \ker \epi_2 \to \aut(\Gamma)$ from the proof of Theorem \ref{thm:autdecomp}. Now every $\theta \in \ker \epi_3$ satisfies $\theta(V) = a V a^{-1}$ for some $a \in \prod_{\QQ_2} \Gamma$. Following Houghton, we can write $\theta = \theta_0 \circ \theta_1$, where $\theta_1$ is an inner automorphism of $G$ implemented by an element $h \in \prod_{\QQ_2} \Gamma$. This product can be described using a split epimorphism $\epi_4: \ker \epi_3 \to A_4$, so it is unique.

The description of $\aut(W)$ is now reduced to the description of the automorphisms of $W$ that fix $B$ pointwise, and the description of $\aut(G)$ is reduced to the description of $\ker \epi_4$. Every element of $\ker \epi_4$ fixes $V$ pointwise, but the converse is not always true; there can be elements of $\aut(G)$ which fix $V$ pointwise, but are not in $\ker \epi_4$ (e.g. any non-trivial element of $A_2$).

After writing each automorphism $\theta \in \ker \phi$ as a product $\theta_0 \circ \theta_1$ with $\theta_0$ fixing $B$ pointwise, and $\theta_1$ an inner automorphism of $W$ implemented by an element of $\prod_{B} \Gamma$, Houghton decomposes $\theta_0$ uniquely as a product $\theta_{00} \circ \theta_{01}$. Here $\theta_{00}$ fixes the diagonal of $W$ along with $B$ pointwise, and $\theta_{01}$ can identified with an element of $\aut(\Gamma)$, in essentially the same way as $A_2 \leq \aut(G)$ can be identified with $\aut(\Gamma)$. This completes the description of $\aut(W)$ given in \cite{autstandwreath}, arriving at the decomposition 
\[
\aut(W) = ( \left(A_{D, B} \rtimes \aut(\Gamma)\right) \cdot I) \rtimes \aut(B) \tag{$\bigstar$},
\]
where $A_{D,B}$ denotes the automorphisms of $W$ which fix the diagonal subgroup $D \leq \prod_{B} \Gamma$ and $B$ pointwise, and $I$ denotes the inner automorphisms of $W$ corresponding to elements of $\prod_{B}\Gamma$.

In Theorem \ref{thm:autdecomp}, we finish the description of $\aut(G)$ by decomposing $\ker \epi_4$ as a semidirect product $\ker \epi_5 \rtimes A_5$, where $\ker \epi_5 \cong \mathrm{Hom}(G, ZG)$.

The least understood group appearing in the description of $\aut(W)$ given by Houghton is the group $A_{D,B}$. Houghton's strategy is to restrict $\Gamma$ and $B$ to be finite cyclic groups, so that the group $A_{D,B}$ has an adequate description. 
For our description of $\aut(G)$, the analogous group is $A_5$. We know that $A_5$ is trivial when $Z \Gamma$ is trivial or $\Gamma$ is perfect (Remark \ref{rmk:exoticauto}), but we have no idea what $A_5$ is in other cases (Question \ref{question:z2auto}).

In \cite{autpermwreath}, Hassanabadi extends Houghton's description to wreath products $W = \prod_{X} \Gamma \rtimes B$, where $X = B / H$ is the set of left cosets of some subgroup $H \leq B$. In Theorem 1.1 of \cite{autpermwreath}, a similar description as $(\bigstar)$ is given for the subgroup
\[
\mathcal{B}(W) := \left\{\theta \in \aut(W) \ | \ \theta(B) = f B f^{-1} \text{ for some $f \in \prod_{X} \Gamma$}\right\} \leq \aut(W).
\]
If $H = \{e\}$, then $W$ becomes a standard wreath product and we have $\mathcal{B}(W) = \aut(W)$.

Since the action $V \curvearrowright \QQ_2$ is transitive, taking $B := V$ and $H := \mathrm{Stab}(x)$ for some $x \in \QQ_2$, we have that $W \cong G$. Theorem 1.1 of \cite{autpermwreath} asserts that the decomposition $(\bigstar)$ holds for $\bg$ after replacing $\aut(B)$ with $A_1$, i.e.
\[
\bg = \left(\left(A_{D,V} \rtimes \aut(\Gamma)\right) \cdot I \right) \rtimes A_1.
\]
Recalling that
\[
\mathcal{B}(G) = (((A_6 \rtimes A_5) \rtimes A_4 ) \rtimes A_2 ) \rtimes A_1,
\]
Theorem \ref{thm:autdecomp} gives an alternative description of $\mathcal{B}(G)$ to \cite{autpermwreath}.

We can also use Theorem \ref{thm:autdecomp} to describe the group $A_{D,V}$. If $\theta \in A_{D,V}$, an application of Theorem \ref{thm:autdecomp} and a quick computation shows that $\theta = \theta_6 \circ \theta_5 \circ \theta_2$, where each $\theta_i \in A_i$. Since $\theta_5$ fixes $V$ pointwise, it must be that $\theta_5(D) = D$. Letting $g \in \Gamma$, and $g_{\QQ_2}$ be the constant map $\QQ_2 \to \Gamma$ and equal to $g$, then $\theta_5(g_{\QQ_2}) = \gamma(g)_{\QQ_2}$ for some unique $\gamma(g) \in \Gamma$. The map $\gamma: \Gamma \to \Gamma$ then defines an automorphism of $\Gamma$, and using the notation of the proof of Lemma \ref{lem:autsplit}, we have $\theta_2 = \overline{\gamma^{-1}}$. From Lemma \ref{lem:etalem}, we also have $\gamma(g) = g \cdot z(g)$ for some (unique) group morphism $z: \Gamma \to Z \Gamma$.

This implies that if $Z \Gamma = e_\Gamma$ or $\Gamma$ is perfect, then $A_{D, V} = \{\mathrm{id}_G\}$. Since $\bg = \aut(G)$ in this case, Hassanabadi's characterisation yields
\[
\aut(G) = \left( \aut(\Gamma) \cdot I \right) \rtimes A_1.
\]
\end{remark}

\begin{remark}\label{rmk:restrictedautembedding}
    In \cite{jonestechI}, Brothier gave a complete description of the automorphism group of the restricted untwisted permutational wreath product $G_0 := \oplus_{\QQ_2} \Gamma \rtimes V$. We now compare $\aut(G_0)$ and $\aut(G)$, where $G = \prod_{\QQ_2} \rtimes V$ is the corresponding unrestricted wreath product. 
    
    Given $\theta \in \aut(G_0)$, as shown in \cite{jonestechI} and re-explained by Theorem \ref{thm:isodecomp}, there exists a family $(\kappa_x)_{x \in \mathbb{Q}_2}$ of automorphisms of $\Gamma$ satisfying
    \[\theta(a)(\varphi(x)) = \kappa_x(a(x)) \text{ for all $a \in \oplus_{\QQ_2} \Gamma$ and $x \in \QQ_2$}. \tag{$\heartsuit$}\]
    The application of Theorem \ref{thm:isodecomp} here comes from the fact that the group $\oplus_{\mathbb{Q}_2} \Gamma$ is isomorphic to its direct square via the bijection $\QQ_2 \to \QQ_2 \sqcup \QQ_2$ of Example \ref{eg:cantordyadicjonesiso}, and this isomorphism implements the action $V \curvearrowright \oplus_{\QQ_2} \Gamma$. Thanks to $(\heartsuit)$, we can extend $\theta$ to all of $G$ using the same formula:
    \[
    \overline{\theta}(a)(\varphi(x)) := \kappa_x(a(x)) \text{ and $\overline{\theta}(v) := \theta(v)$ for all $a \in \prod_{\QQ_2} \Gamma$, $x \in \QQ_2$ and $v \in V$.}
    \]
    The map $\overline{\theta} : G \to G$ is an automorphism of $G$, and the map $\iota: \aut(G_0) \to \aut(G), \theta \mapsto \overline{\theta}$ is an embedding of groups.
    
    To characterise $\aut(G_0)$, Brothier defined four subgroups $E_1, \dots, E_4 \leq \aut(G_0)$, whose elements were called elementary, and built a group $Q = (E_4 \times E_3) \rtimes (E_2 \times E_1)$ and a surjection $\Xi: Q \to \aut(G_0)$. Due to the small kernel of $\Xi$, this yields an almost unique factorisation of automorphisms of $G_0$ into elementary automorphisms. This description can be rephrased to be more comparable to Theorem \ref{thm:autdecomp}. Using a similar approach to the proof of Theorem \ref{thm:autdecomp}, one has the decomposition 
    \[ \aut(G_0) = ( ( \wt{A}_4 \rtimes \wt{A}_3)\rtimes \wt{A}_2)\rtimes \wt{A}_1, \tag{$\spadesuit$}\]
    where 
    \begin{itemize}
        \item $\wt{A}_1 = \stabn = A_1$,
        \item $\wt{A}_2 = \aut(\Gamma) = A_2$,
        \item $\wt{A}_3 = Z \Gamma \cong A_3$,
        \item $\wt{A}_4 = \{ h \in A_4 \ | \ \text{$h$ normalises $\oplus_{\QQ_2} \Gamma$}\}$.
    \end{itemize} Recall that $A_4$ consists of the maps $h : \QQ_2 \to \Gamma$ such that $h(00\cdots) \in Z \Gamma$, modulo constant maps $\QQ_2 \to Z \Gamma$.
    
    The decomposition $(\spadesuit)$ is achieved by pre-composing the inclusion $\iota: \aut(G_0) \to \aut(G)$ with the split epimorphisms $\epi_1: \aut(G) \to A_1$, $\epi_2: \ker \epi_1 \to A_2$ and $\epi_3: \ker \epi_2 \to A_3$ from the proof of Theorem \ref{thm:autdecomp}. For $i = 1$ and $i = 2$, the sections that we used for $\epi_i$ factor through $\iota$, giving sections for $\epi_i \circ \iota$. Equivalently, the automorphisms in $A_1$ and $A_2$ restrict to automorphisms of $G_0$. This yields decompositions 
    \[\aut(G_0) = \ker (\chi_1 \circ \iota) \rtimes \wt{A}_1 \text{ and } \ker (\chi_1 \circ \iota) = \ker (\chi_2 \circ \iota) \rtimes \wt{A}_2.\]
    We need to find a different section for $\epi_3 \ \circ \ \iota$, since an automorphism of $G$ defined by the formula $av \mapsto a \cdot s(z)_v \cdot v$ with $z \in Z \Gamma$ (an element of $A_3$) does not restrict to an automorphism of $G_0$ in general. This is remedied by replacing the cocycle $s(z)$ with a finitely supported one.

    Let $\nu: \QQ_2 \to \ZZ$ be the dyadic valuation, so that $\nu(x_1 x_2 x_3 \cdots) := 2^{-i}$, where $i$ is the largest positive integer satisfying $x_i \neq 0$. Fix $z \in Z \Gamma$, and define
    \[ s_0(z)_v(x) := \zeta^{p_v(x)}\text{, $v \in V$, $x \in \QQ_2$,} \] 
    where
    \[p_v(x) := \log_2(v'(v^{-1}(x))) - \nu(x) +\nu(v^{-1}x)\text{, $x \in \QQ_2$}.\] This yields a map $s_0(z): V \to K$, and it is shown in \cite[Proposition 4.2]{jonestechI} that $s_0(z)$ is a cocycle valued in $\oplus_{\QQ_2} \Gamma$, i.e.
    \[
    s_0(z)_v \in \oplus_{\QQ_2} \Gamma \text{ and $s_0(z)_{vw} = s_0(z)_{v} \cdot \pi(v)(s_0(z)_{w})$ for all $v, w \in V$.}
    \] 
    With a little bit of work, one shows that $\epi_3 \circ \iota$ splits via the embedding $Z \Gamma \to \ker(\epi_2 \circ \iota)$, $z \mapsto \left(av \mapsto a \cdot s_0(z) \cdot v \right)$. This gives us that
    \[\ker (\chi_2 \circ \iota) = \ker (\chi_3 \circ \iota) \rtimes \wt{A}_3.\]
   If $\theta \in \ker (\chi_3 \circ \iota)$, then by definition $\overline{\theta} \in \ker (\chi_3)$. One can then show that $\theta = \ad(h)$ for some $h \in \left(\prod_{\QQ_2} \Gamma\right) / Z \Gamma$. Moreover, $h(00\cdots) \in Z \Gamma$ since $\theta \in \ker (\epi_3 \circ \iota)$, and of course $h$ must normalise $\oplus_{\QQ_2} \Gamma$. This gives the inclusion $\wt{A}_4 \subseteq \{h \in A_4 \ | \ h \text{ normalises $\oplus_{\QQ_2} \Gamma$}\}$, while the reverse inclusion can be quickly checked. We have arrived at the decomposition $(\spadesuit)$.
    
    How well does the decomposition of $\aut(G_0)$ in $(\spadesuit)$ match up with the decomposition of $\aut(G)$ in Theorem \ref{thm:autdecomp} under the embedding $\iota: \aut(G_0) \hookrightarrow \aut(G)$? Let $\theta \in \aut(G_0)$, and via $(\spadesuit)$ write $\theta = \theta_4 \circ \theta_3 \circ \theta_2 \circ \theta_1$ with each $\theta_i \in \wt{A}_i$. Next, via Theorem \ref{thm:autdecomp}, write $\iota(\theta) = \lambda_6 \circ \lambda_5 \circ \cdots \circ  \lambda_1$, with each $\lambda_i \in A_i$. We then have that $\lambda_i = \iota(\theta_i)$ when $i = 1$ or $i = 2$. Writing $\theta_3: av \mapsto a \cdot s_0(z)_v \cdot v$ for a unique $z \in Z \Gamma$, we have
    \[
    s_0(z)_v = s(z)_v \cdot [f,v],
    \]
    where $f: (\prod_{\QQ_2} \Gamma) / Z \Gamma$ is defined by the formula 
    \[
    f(x) := z^{- \nu(x)}, \ x \in \QQ_2.
    \]
    Thus, $\iota(\theta_3) = \ad(f) \circ (av \mapsto a \cdot s(z)_v \cdot v)$, and writing $\theta^4 = \ad(h)$ for a unique $h \in \wt{A}_4$, we obtain that
    \[
    \lambda_3: av \mapsto a \cdot s(z)_v \cdot v \text{ and } \lambda_4 = \ad(hf).
    \]
    Finally, $\lambda_5 = \lambda_6 = \mathrm{id}_G$. I.e., when $i \neq 3$, $\iota(\theta_i) \in \aut(G)$ has one factor in the decomposition of Theorem \ref{thm:autdecomp}, while $\iota(\theta_3)$ splits into two factors.
\end{remark}

\bibliographystyle{plain}

\end{document}